\documentclass[12pt, reqno]{amsart}
\usepackage{mathrsfs}
\usepackage{amsfonts,bbm,bm,pifont,mathrsfs,mathtools,stmaryrd}
\usepackage[sorted,compressed-cites,sorted-cites,initials]{amsrefs}
\usepackage[centertags]{amsmath}
\usepackage{amssymb,array}
\usepackage{amsthm,youngtab}
\usepackage{graphicx}
\usepackage{caption}
\usepackage{ytableau}
\usepackage{enumerate,enumitem}
\SetLabelAlign{center}{\hfill #1\hfill}
\usepackage[textwidth=16cm, hmarginratio=1:1]{geometry}
\usepackage{tikz-cd}
\usepackage{rotating}
\usepackage[all,cmtip]{xy}
\usepackage[titletoc, title]{appendix}
\usepackage{amscd}
\usepackage[colorlinks]{hyperref}
\usepackage{float}
\hypersetup{
bookmarksnumbered,
pdfstartview={FitH},
breaklinks=true,
linkcolor=blue,
urlcolor=blue,
citecolor=blue,
bookmarksdepth=2
}

\newcommand{\la}{\langle}
\newcommand{\ra}{\rangle}
\newcommand{\fgfrm}[2]{\llparenthesis#1,#2\rrparenthesis}

\allowdisplaybreaks[4]
\usetikzlibrary{matrix,arrows,decorations.pathmorphing}

\newcommand{\lact}{\triangleright}
\newtheorem{theorem}{Theorem}[section]
\newtheorem{corollary}[theorem]{Corollary}
\newtheorem{lemma}[theorem]{Lemma}
\newtheorem{proposition}[theorem]{Proposition}
\theoremstyle{definition}
\newtheorem{definition}[theorem]{Definition}
\newtheorem{remark}[theorem]{Remark}

\newtheorem{problem}[theorem]{Problem}

\newtheorem{conjecture}[theorem]{Conjecture}

\newcommand{\lla}{\la\!\la}
\newcommand{\rra}{\ra\!\ra}

\numberwithin{equation}{section}

\newcommand{\tensor}{\otimes}

\newcommand{\kk}{\Bbbk}
\newcommand{\lie}[1]{{\mathfrak{#1}}}
\newcommand{\ZZ}{{\mathbb Z}}

\DeclareMathOperator{\id}{id}
\DeclareMathOperator{\wt}{wt}
\DeclareMathOperator*{\cl}{cl}
\DeclareMathOperator*{\supp}{supp}

\DeclareMathOperator{\ad}{ad}

\DeclareMathOperator{\Cact}{\mathsf{Cact}}

\DeclareMathOperator{\End}{End}

\DeclareMathOperator{\Br}{\mathsf{Br}}

\DeclareMathOperator*{\Ann}{Ann}

\DeclareMathOperator*{\dist}{dist}

\DeclareMathOperator{\Hom}{Hom}

\begin{document}

\title{On cacti and crystals}
\author{Arkady Berenstein, Jacob Greenstein and Jian-Rong Li}
\address{Arkady Berenstein, Department of Mathematics, University of Oregon, Eugene, OR 97403, USA}
\email{arkadiy@math.uoregon.edu}
\address{Jacob Greenstein, Department of Mathematics, University of California, Riverside, CA 92521, USA}
\email{jacobg@ucr.edu}
\address{Jian-Rong Li, Department of Mathematics, The Weizmann Institute of Science, Rehovot 76100, Israel}
\email{lijr07@gmail.com}
\date{}

\thanks{This work was partially supported by a BSF grant no.~2016363 (A.~Berenstein), the Simons foundation collaboration grant no.~245735 (J.~Greenstein)
and by the Minerva foundation
with funding from the Federal German Ministry for Education and Research (J.-R. Li).
}

\begin{abstract}
In the present work we study actions of various groups generated by involutions on the 
category~$\mathscr O^{int}_q(\lie g)$ of integrable highest weight 
$U_q(\lie g)$-modules and their crystal bases for any symmetrizable Kac-Moody algebra~$\lie g$.
The most notable of them are the cactus group and (yet conjectural) Weyl group action on
any highest weight integrable module and its lower and upper crystal bases. Surprisingly, some generators of cactus groups are anti-involutions of 
the Gelfand-Kirillov model for~$\mathscr O^{int}_q(\lie g)$
closely related to the 
remarkable quantum twists discovered
by Kimura and Oya in~\cite{KiO}.
\end{abstract}

\dedicatory{To Anthony Joseph,
with admiration}

\maketitle 

\tableofcontents

\section{Introduction}
In the present work we study the action of various groups generated by involutions on the 
category~$\mathscr O^{int}_q(\lie g)$ of integrable highest weight 
$U_q(\lie g)$-modules for any symmetrizable Kac-Moody algebra~$\lie g$ (the necessary notation is introduced in Section~\ref{sect:defn not}). 

Let $V\in \mathscr O^{int}_q(\lie g)$. We claim that for every node~$i$ of the Dynkin diagram~$I$ of~$\lie g$
there exists a unique linear operator~$\sigma^i_V$ on~$V$ such that 
\begin{equation}\label{eq:eta^i defn}\sigma^i_V(E_i^{(k)}(u))=E_i^{(l-k)}(u)
\end{equation}
for all $l\ge k\ge 0$ and for all $u\in \ker F_i\cap \ker(K_i-q_i^{-l})$.
Clearly, $(\sigma^i_V)^2=\id_V$.
Denote by $\mathsf W(V)$ be the subgroup of~$\operatorname{GL}_\kk(V)$ generated by the $\sigma^{i}_V$, $i\in I$.
\begin{theorem}\label{thm:crystal weyl group}
For any non-zero module~$V\in\mathscr O^{int}_q(\lie g)$, 
the assignments $$
\sigma^i_V\mapsto \begin{cases}1,& i\in J(V)\\
               s_i,& \text{otherwise}
              \end{cases}
$$
where $J(V)=\{ i\in I\,:\, F_i(V)=\{0\}\}$, 
define a homomorphism $\psi_V$ from~$\mathsf W(V)$ to the Weyl group~$W$ of~$\lie g$. 
\end{theorem}
We prove Theorem~\ref{thm:crystal weyl group} in \S\ref{subs:cryst weyl group} by
showing that the image of~$\psi_V$ can be described in terms of a natural action of~$W$ on a certain set of 
extremal vectors in~$V$. In particular, $\psi_V$ is surjective if and only if $J(V)=\emptyset$.
Moreover, we show that $\sigma^i_V=\id_V$ if and only if $i\in J(V)$. This suggests the following
\begin{conjecture}\label{conj:weyl}
The homomorphism~$\psi_V$ is injective for any~$V\in\mathscr O^{int}_q(\lie g)$. 
\end{conjecture}
Clearly, it is equivalent to $(\sigma^i\sigma^j)^{m_{ij}}=\id_V$, $i\not=j\in I$ for appropriate choices of~$m_{ij}$.
We proved it for $m_{ij}=2$ and 
we have ample evidence that this conjecture holds for~$m_{ij}=3$.
We also verified it for all modules in which weight spaces of non-zero weight are one dimensional (see Theorem~\ref{thm:thin modules}). This 
class of modules includes all miniscule and quasi-miniscule ones.
Conjecture~\ref{conj:weyl} combined with Theorem~\ref{thm:crystal weyl group} implies that $W$ acts naturally and faithfully on objects in~$\mathscr O^{int}_q(\lie g)$, 
which is quite surprising. Informally speaking, this conjecture asserts that Kashiwara's action of the Weyl group on crystal bases 
lifts to an action on the corresponding module (see Remark~\ref{rem:crystal Weyl group}).
\begin{remark}\label{rem:classical limit}
The definition~\eqref{eq:eta^i defn} of $\sigma^i_V$ makes sense for any integrable $U(\lie g)$-module where~$\lie g$ is a semisimple or a (not necessarily symmetrizable) Kac-Moody Lie algebra. 
The ``classical'' Theorem~\ref{thm:crystal weyl group} holds verbatim. Moreover, Conjecture~\ref{conj:weyl} implies its classical
version for all (even not symmetrizable) Kac-Moody algebras.
\end{remark}

It turns out that we can extend the group $\mathsf W(V)$ by adding involutions $\sigma^J$ for any non-empty $J\subset I$ such that the subgroup $W_J=\la s_i\,:\, i\in J\ra$ is finite;
we denote the set of all such~$J$ by~$\mathscr J$. Note that~$\{i\}\in\mathscr J$ for all~$i\in I$ and in particular~$\mathscr J$ is non-empty.
\begin{proposition}[Proposition~\ref{prop:prop eta^J}]\label{prop:existence-eta-J}
For any $V\in\mathscr O^{int}_q(\lie g)$, $J\in\mathscr J$ there exists a unique $\kk$-linear map $\sigma^J=\sigma^J_V:V\to V$ such that 
\begin{enumerate}[label={\rm(\alph*)},leftmargin=*]
 \item\label{prop:existence-eta-J.a} $\sigma^J(v)=v^J$ for any $v\in \bigcap_{i\in J} \ker E_i$ where 
 $v^J$ is a distinguished element in~$\bigcap\limits_{i\in J}\ker F_i\cap U_q(\lie g^J)v$ defined in~Proposition~\ref{prop:prop eta^J}\ref{prop:prop eta^J.a'};
 \item\label{prop:existence-eta-J.b} $\sigma^J(F_j(v))=E_{j^\star}(\sigma^J(v))$, $\sigma^J(E_j(v))=F_{j^\star}(\sigma^J(v))$ for all~$j\in J$, $v\in V$
 where ${}^\star:J\to J$ is the involution on~$J$ induced by the longest element~$w_\circ^J$ of~$W_J$ via $s_{j^\star}=w_\circ^J s_j w_\circ^J$, $j\in J$
 (see~\S\ref{subs:coxeter groups}).
\end{enumerate} 
Moreover, for any morphism $f:V\to V'$ in~$\mathscr O^{int}_q(\lie g)$ the following diagram commutes
 \begin{equation}\label{eq:eta^J funct}
\vcenter{ \xymatrix{ V\ar[d]_{f}\ar[r]^{\sigma^J_V} & V\ar[d]^{f}\\
 V'\ar[r]^{\sigma^J_{V'}} &V'}}
 \end{equation}
\end{proposition}

By definition, $\sigma^J=\sigma^i$ if~$J=\{i\}$.
The following is the main result of this paper.
\begin{theorem}\label{thm:main thm}
Let~$V\in\mathscr O^{int}_q(\lie g)$. Then for any~$J\in\mathscr J$ we have in~$\operatorname{GL}_\kk(V)$
\begin{enumerate}[label={\rm(\alph*)},leftmargin=*]
 \item\label{thm:main thm.a} $\sigma^J\circ\sigma^J=1$;
 \item\label{thm:main thm.b} If $J=J'\cup J''$ where $J'$ and~$J''$ are orthogonal (that is, $J'\cap J''=\emptyset$ and
 $s_{j'}s_{j''}=s_{j''}s_{j'}$ for all $j'\in J'$, $j''\in J''$)
 then $\sigma^J=\sigma^{J'}\circ \sigma^{J''}$; in particular, $\sigma^{J'}\circ \sigma^{J''}=\sigma^{J''}\circ\sigma^{J'}$ if $J',J''\in\mathscr J$ are orthogonal.

 \item\label{thm:main thm.c} $\sigma^{J}\circ \sigma^{K}=\sigma^{K^\star}\circ \sigma^J$ for any~$K\subset J$, where ${}^\star:J\to J$ is as in 
 Proposition~\ref{prop:existence-eta-J}\ref{prop:existence-eta-J.b}.
 \end{enumerate}
\end{theorem}
We prove Theorem~\ref{thm:main thm} in~\S\ref{subs:parab inv} using appropriate modifications of Lusztig's symmetries (which we introduce in
\S\ref{subs:mod lus sym}).

Following (and slightly generalizing) \cite{Los} (see also~\cite{Bon}), we denote $\Cact_W$ the group generated by
the $\tau_J$, $J\in\mathscr J$ subject to all relations of Theorem~\ref{thm:main thm}. Indeed, this definition coincides with that in~\cite{Los}*{(1.1)}
if~$W$ is finite because $\tau_J=\tau_{J'}\tau_{J''}$ for any~$J$ as in Theorem~\ref{thm:main thm}\ref{thm:main thm.b}.
By definition, the assignments $\tau_J\mapsto \sigma^J_V$, $J\in\mathscr J$ define a representation of~$\Cact_W$ on~$V$. 
In view of~\eqref{eq:eta^J funct} we obtain the following immediate corollary of Theorem~\ref{thm:main thm} (see Section~\ref{sect:mod Lusz symm} for the notation)
\begin{corollary}\label{cor:cactus categorical-action}
The group~$\Cact_W$ acts on the category~$\mathscr O^{int}_q(\lie g)$ via~$\tau_J\mapsto \sigma^J_\bullet$, $J\in\mathscr J$.
\end{corollary}
The study of cactus groups began with $\Cact_n:=\Cact_{S_n}$ which appeared, to name but a few, in~\cites{De,DJS,DR,Ryb,White} in connection with the study of moduli spaces of rational 
curves with $n+1$ marked points and their applications in mathematical physics. It is easy to see that $\Cact_n$ is generated by 
involutions $\tau_{i,j}=\tau_{\{i,\dots,j-1\}}$, $1\le i<j\le n$ subject to 
the relations 
\begin{alignat*}{2}
&\tau_{i,j}\tau_{k,l}=\tau_{k,l}\tau_{i,j},&\qquad & i<j<k<l\\
&\tau_{i,l}\tau_{j,k}=\tau_{i+l-k,i+l-j}\tau_{i,l},&& i\le j<k\le l.
\end{alignat*}
Categorical actions of~$\Cact_n$ on $n$-fold tensor products in symmetric coboundary categories (first introduced in~\cite{Dr}) were studied in~\cites{HK,Sav} 
and also implicitly in~\cite{BZw} where the braided structure on the category~$\mathscr O_q^{int}(\lie g)$ was converted into a symmetric coboundary structure
for any complex reductive Lie algebra~$\lie g$ (for non-abelian  examples of coboundary categories see the discussion after Theorem~\ref{thm:main thm2}).
It would be interesting to compare these actions of~$\Cact_n$ with the one given by Corollary~\ref{cor:cactus categorical-action}. We expect that 
they are connected in some cases via the celebrated Howe duality (see e.g. the forthcoming paper~\cite{BGL}). 
In view of Corollary~\ref{cor:cactus categorical-action} it is natural to seek other categorical representations of~$\Cact_W$ for all Coxter groups~$W$.

Conjecture~\ref{conj:weyl} suggests that our representation of~$\Cact_W$ on~$\mathscr O^{int}_q(\lie g)$ is not faithful. Therefore, we can pose the following 
\begin{problem}\label{prob:kernel}
Find the kernel~$\mathsf K_{\lie g}$ of this representation of~$\Cact_W$.
\end{problem}
For example, if $\lie g=\lie{sl}_3$ then $\tau_{1,2}\tau_{1,3}=\tau_{1,3}\tau_{2,3}$ and so
$\Cact_W$ is freely generated by involutions $\tau_{1,2}$ and $\tau_{1,3}$.
It is easy to see that $\tau_{1,2}\notin \mathsf K_{\lie{sl}_3}$, while $\tau_{1,3}\notin\mathsf K_{\lie{sl}_3}$ by 
Remark~\ref{rem:tau 13 not in subgroup}. Thus, if Conjecture~\ref{conj:weyl} holds then 
$\mathsf K_{\lie{sl}_3}=\{ (\tau_{1,2}\tau_{1,3})^{6n}\,:\, n\in\ZZ\}$ which would 
solve the problem for $W=S_3$.

To outline an approach to Problem~\ref{prob:kernel} in general, 
denote $\Phi_V$, $V\in\mathscr O^{int}_q(\lie g)$ the subgroup of~$\operatorname{GL}_\kk(V)$ generated by the $\sigma^J_V$, $J\in\mathscr J$.
Then clearly $\mathsf K_{\lie g}$ is the intersection of kernels of canonical homomorphisms $\Cact_W\to \Phi_V$ over all $V\in\mathscr O^{int}_q(\lie g)$.
We show (Proposition~\ref{prop:ficus}) that 
$\Phi_V\cong \Phi_{\underline V}$ where $\underline V=\bigoplus\limits_{\lambda\in P^+\,:\, \Hom_{U_q(\lie g)}(V_\lambda,V)\not=0} V_\lambda$.
In particular, $\Cact_W/\mathsf K_{\lie g}$ is isomorphic to $\Phi_{\mathcal C_q(\lie g)}$ where $\mathcal C_q(\lie g)=\bigoplus_{\lambda\in P^+} V_\lambda$
is the Gelfand-Kirillov model for~$\mathscr O^{int}_q(\lie g)$; in fact, it has a structure of an associative algebra (see Section~\ref{sect:can bas}).
Thus, in view of the above we expect that $\Cact_3/\mathsf K_{\lie{sl}_3}$ is isomorphic to the dihedral group of order~$12$. However,
it is likely that~$\Phi_{\mathcal C_q(\lie g)}$ is infinite for simple $\lie g$ different from $\lie{sl}_2$ and $\lie{sl}_3$.

It turns out that the action of~$\Cact_W$ on~$\mathscr O^{int}_q(\lie g)$ descends to a permutation representation 
on any crystal basis of any object~$V$ (see \S\ref{subs:prelim crystal bases} for definitions and notation). Thus, we obtain 
the following refinement of~\cite{HKRW}*{Theorem~5.9}.
\begin{theorem}\label{thm:main thm2}
Let $V\in\mathscr O^{int}_q(\lie g)$. Then for any lower or upper crystal basis $(L,B)$ of~$V$ at $q=0$ the group $\Phi_V$ preserves~$L$ and acts 
on~$B$ by permutations.
\end{theorem}
We prove Theorem~\ref{thm:main thm2} in Section~\ref{sect:action on crystals} by means of what we call $\mathbf c$-crystal bases, which allow one
to treat lower and upper crystal bases uniformly. Taking into account that $B$ is graded by the weight lattice of~$\lie g$, all weights occur 
in a crystal basis of~$\mathcal C_q(\lie g)$ and that $W$ acts faithfully on the weight lattice, we obtain an immediate
\begin{corollary}\label{cor:kernel to W}
The assignments $\sigma^J\mapsto w_\circ^J$, $J\in\mathscr J$ define a surjective homomorphism $\Cact_W/\mathsf K_{\lie g}\to W$ which 
refines the natural epimorphism $\Cact_W\to W$ from~\cite{Los}.
\end{corollary}
Analogously to the notion of the pure braid group, one calls the kernel of the natural homomorphism $\Cact_W\to W$
the pure cactus group
(this term was used for~$\Cact_n$ in e.g.~\cites{DR,Ryb,White}). Thus, Corollary~\ref{cor:kernel to W} asserts 
that~$\mathsf K_{\lie g}$ is pure.

The involution $\sigma^I_V$ was first defined in~\cite{BZ1} for $\lie g=\lie{gl}_n$ and simple polynomial representations~$V_\lambda$
and explicitly computed on the corresponding crystal in~\cite{KB}. In fact, it coincides with the famous Sch\"uztenberger involution
(see Remark~\ref{rem:eta inv}). 
Following a suggestion of the first author and~\cite{KB}, an action of~$\Cact_n$ on the category of crystal bases was constructed in~\cite{HK} thus turning it 
into a symmetric coboundary category.

We expect that to solve Problem~\ref{prob:kernel} it suffices to find kernels of permutation representations of~$\Cact_W$ on all~$B$.

Since $\mathsf W(V)$ is naturally a subgroup of~$\Phi_V$, its action on~$V$ induces an action on~$B$ by permutations
which coincides with Kashiwara's crystal Weyl group (see Remark~\ref{rem:crystal Weyl group}).

In case when $\lie g$ is reductive we can refine Theorem~\ref{thm:main thm2} as follows.
\begin{theorem}\label{thm:main thm3}
Let $\lie g$ be a reductive Lie algebra and let $V\in\mathscr O^{int}_q(\lie g)$. Then for any 
crystal basis $(L,B)$ of~$V$ the involution $\sigma^I_V$ preserves the corresponding upper global crystal basis~$\mathbf B_V$ of~$V$.
\end{theorem}
An analogous result for $J\subsetneq I$ is weaker. We prove (Proposition~\ref{prop:sigma bar}) that
the image of any element of~$\mathbf B_V$ under~$\sigma^J$, $J\in\mathscr J$ is a $\bar\cdot$-invariant element of~$V$
where $\bar\cdot$ is the anti-linear involution fixing~$\mathbf B_V$. However, as explained in Remark~\ref{rem:GT bas and sigma},
$\sigma^J$ does not need to preserve~$\mathbf B_V$ if $J\subsetneq I$. For example, if $V$ is the $27$-dimensional simple 
module $V_{2\rho}$ for $\lie g=\lie{sl}_3$ then the $\sigma^i_V$, $i=1,2$ do not preserve the canonical 
basis of~$V$

An analogue of Theorem~\ref{thm:main thm3} for a simple~$V$ and its lower global crystal basis was deduced from~\cite{Lus}*{Proposition~21.1.2} in~\cite{HK}*{Theorem~5}. 

We prove Theorem~\ref{thm:main thm3} in~\S\ref{subs:pf main thm3}. A central role in our argument is played by the following surprising property 
of $\sigma^I$ on the aforementioned quantum Gelfand-Kirillov model~$\mathcal C_q(\lie g)$ of~$\mathscr O^{int}_q(\lie g)$.
\begin{theorem}[Theorem~\ref{thm:twist}]\label{thm:twist-intrd}
For~$\lie g$ reductive finite dimensional,
$\sigma^I_{\mathcal C_q(\lie g)}$ is an anti-involution on~$\mathcal C_q(\lie g)$.
\end{theorem}
We do not expect an analogous result for $J\subsetneq I$; 
for example, for $\lie g=\lie{sl}_3$, the $\sigma^i$, $i\in\{1,2\}$ are not anti-automorphisms of~$\mathcal C_q(\lie g)$.
Our proof of Theorems~\ref{thm:main thm3} and~\ref{thm:twist-intrd} rely in a crucial way on properties of a remarkable quantum twist defined in~\cite{KiO}.

In view of Theorem~\ref{thm:main thm2} we can refine Conjecture~\ref{conj:weyl} for every~$V\in\mathscr O^{int}_q(\lie g)$ with 
$\underline V=\mathcal C_q(\lie g)$ as follows. We expect that in the notation
of Theorem~\ref{thm:main thm2} the group~$\Phi_V$ acts on~$B$ faithfully. Morally, this means that each element of~$\Phi_V$ is semisimple in~$\operatorname{GL}_\kk(V)$.

Similarly to Remark~\ref{rem:classical limit}, 
our constructions, results and conjectures make sense if one replaces $U_q(\lie g)$ by~$U(\lie g)$ for any (symmetrizable or not) Kac-Moody algebra $\lie g$. 
Some results (for example, Theorem~\ref{thm:main thm2}) should be possible to rescue even when $W$ is not crystallographic (and so $\lie g$ does not exists)
with the aid of theory of continuous crystals initiated by A.~Joseph in~\cites{Jos}.

\subsection*{Acknowledgements}
The first two authors are grateful to Anton Alekseev and Universit\'e de Gen\`eve, Switzerland, for their hospitality. An important part of this work was done during 
the second author's stay at the Weizmann Institute of Science, Israel and during the conference in honor of Anthony Joseph's 75th birthday, at the Weizmann Institute 
and at the University of Haifa. The authors would like to use
this opportunity to
thank Maria Gorelik and Anna Melnikov for organizing that wonderful event.

\section{Preliminaries}\label{sect:defn not}

\subsection{Coxeter groups}\label{subs:coxeter groups}
Let~$I$ be a finite set.
Let $W$ be a Coxeter group with Coxeter generators $s_i$, $i\in I$ subject to the relations $(s_i s_j)^{m_{ij}}=1$ where
$m_{ii}=1$, $m_{ij}=m_{ji}$ and $m_{ij}\in\{0\}\cup\mathbb Z_{\ge 2}$ for $i\not=j\in I$.
Let $\ell:W\to\ZZ_{\ge 0}$ be the Coxeter length function, that is, $\ell(w)$ is the minimal length of a presentation of~$w$ as a product of the~$s_i$, $i\in I$.
We say that $\mathbf i=(i_1,\dots,i_r)\in I^r$ is reduced if $\ell(s_{i_1}\cdots s_{i_r})=r$ 
and denote by~$R(w)$ the set of reduced words for~$w$, that is, 
$R(w)=\{ (i_1,\dots,i_{\ell(w)})\in I^{\ell(w)}\,:\, w=s_{i_1}\cdots s_{i_{\ell(w)}}\}$.

Given $J\subset I$ we denote by $W_J$ the subgroup of~$W$ generated by the $s_i$, $i\in J$.
We will need the following standard fact (see~\cite{Bou}*{IV.1.8, Th\'eor\`eme 2}).
\begin{lemma}\label{lem:intersect parabolics}
For any $J,J'\subset I$
\begin{enumerate}[label={\rm(\alph*)},leftmargin=*]
 \item \label{lem:intersect parabolics.a} $W_J\cap W_{J'}=W_{J\cap J'}$;
 \item \label{lem:intersect parabolics.b} $W_J\subset W_{J'}$ if and only if $J\subset J'$.
\end{enumerate}
\end{lemma}
Let~$\mathscr J=\{ J\subset I,\,:\, |W_J|<\infty\}$.
If~$J\in\mathscr J$ we denote by $w_\circ^J$ the unique longest element of~$W_J$; thus, $\ell(s_j w_\circ^J)<\ell(w_\circ^J)$ for all $j\in J$.
If $I\in\mathscr J$ we abbreviate $w_\circ=w_\circ^I$. 
Given $J\in\mathscr J$ and $j\in J$, there exists a unique $j^\star\in J$ such that $s_{j^\star}=w_\circ^J s_j w_\circ^J$; the assignments
$j\mapsto j^\star$ define an involution ${}^\star:J\to J$.

Given $J\subset I$, we set $J^\perp=\{i\in I\setminus J\,:\, m_{ij}=2,\, \forall\,j\in J\}=\{ i\in I\setminus J\,:\, s_i s_j=s_j s_i,\, \forall\, j\in J\}$. We say that 
$J,J'\subset I$ are {\em orthogonal} if $J\cap J'=\emptyset$ and $J'\subset J^\perp$ (whence $J\subset J'{}^\perp$).

Define a relation $\sim$ on~$I$ by $i\sim j$ if $i=j$ or $m_{ij}>2$. Then the transitive closure of this relation is an equivalence on~$I$ which we still denote by~$\sim$.
In particular, if $i\sim i'$ then there exists a sequence (called {\em admissible}) $(i_0,\dots,i_d)\in I^{d+1}$ with $i_0=i$, $i_d=i'$
and~$m_{i_{r-1},i_r}>2$, $1\le r\le d$. 
Define $\dist(i,i')$ to be the minimal length of an admissible sequence beginning with~$i$ and ending with~$i'$. Clearly, this defines a metric on~$I$. 

Define a topology on~$I$ by declaring that the fundamental neighborhood of each $i\in I$ is its equivalence class with respect to~$\sim$. In particular,
each open set is closed and vice versa and is a union of equivalence classes. For $J\subset I$ we denote by $\cl(J)$ its closure in that topology, that
is, the union of equivalence classes of elements of~$J$. Denote $\partial(J)$ the boundary of~$J$, that is, the complement of~$J$ in~$\cl(J)$.
The following is immediate.
\begin{lemma}\label{lem:topology}
Let $J\subset I$. Then $I=J\cup J^\perp$ if and only if $J$ is closed in the above topology.
\end{lemma}
The following is a reformulation of a well-known fact (\cite{Bou}*{IV.1.9, Proposition~2})
\begin{lemma}\label{lem:factorization of W}
Let $J\subset I$ be a closed subset. Then $W$ is the internal direct product of~$W_J$ and~$W_{J^\perp}$.
\end{lemma}

Given a group $G$ acting on a set~$X$, denote by $\mathsf K_X(G)$ the kernel of the natural homomorphism of groups~$G\to \operatorname{Bij}(X)$
induced by the action. By definition, the action of~$G$ on~$X$ is faithful if and only if~$\mathsf K_X(G)=\{1\}$.

The following is the main result of~\S\ref{subs:coxeter groups} (which is probably known although we could not find it in the literature). 
\begin{theorem}\label{thm:kernel of natural action}
We have $\mathsf K_J:=\mathsf K_{W/W_J}(W)=W_{I\setminus \cl(I\setminus J)}$ for any $J\subset I$. In particular, if $I$ is connected and~$J\subsetneq I$ then $W$ acts faithfully on~$W/W_J$. 
\end{theorem}
\begin{proof}
The following is immediate.
\begin{lemma}\label{lem:kern natural act}
Let $G$ be a group and~$H$ be a subgroup of~$G$. Then 
$\mathsf K_{G/H}(G)=\{ k\in H\,:\, g^{-1}k g\in H,\, \forall\, g\in G\}$ is a subgroup of~$H$. 
\end{lemma}

The following Lemmata are apparently well-known. We provide their proof for the reader's convenience.
\begin{lemma}\label{lem:Tits}
Let $w\in W$ and let $J\subset I$ be such that $\ell(s_j w)=\ell(w)-1$ for all $j\in J$. Then $W_J$ is finite 
and $w=w_\circ^J w'$ for some~$w'\in W$ with $\ell(w)=\ell(w')+\ell(w_\circ^J)$.
\end{lemma}
\begin{proof}
By \cite{Bou}*{Ch.~IV, Ex. 3}, every 
$u\in W$ can be written uniquely as $[u]_J \cdot {}^J[u]$ where $[u]_J\in W_J$, ${}^J[u]\in {}^J W=\{ x\in W\,:\, \ell(s_j x)>\ell(x),\,\forall\, j\in J\}$
and $\ell(u)=\ell([u]_J)+\ell({}^J[u])$.
The uniqueness of such a presentation implies that $[s_j w]_J=s_j[w]_J$ and ${}^J[s_j w]={}^J[w]$ for all $j\in J$ and so that 
$\ell(s_j[w]_J)<\ell([w]_J)$ for all $j\in J$. This implies that $W_J$ is finite 
and $[w]_J$ is its longest element~$w_\circ^J$. The assertion follows with $w'={}^J[w]$.
\end{proof}
\begin{lemma}\label{lem:perp}
For $i\in I$ and $u\in W_{I\setminus\{i\}}$ the following are equivalent.
\begin{enumerate}[label={\rm(\alph*)},leftmargin=*]
 \item\label{lem:perp.a} $u\in W_{\{i\}^\perp}$ (in particular, $s_i u=u s_i$);
\item\label{lem:perp.b} $s_i u s_i\in W_{I\setminus\{i\}}$.
\end{enumerate}
\end{lemma}
\begin{proof}
The implication \ref{lem:perp.a}$\implies$\ref{lem:perp.b} is obvious. To prove 
the opposite implication, note that 
the assumption in~\ref{lem:perp.b} implies that $s_i u=u' s_i$ for some $u'\in W_{I\setminus\{i\}}$.
Then $\ell(s_i u)=\ell(u)+1$ and $\ell(u's_i)=\ell(u')+1$
whence $\ell(u)=\ell(u')$.
We prove the assertion
\begin{equation}\label{eq:ind hyp W}
s_i u=u's_i\implies u=u'\in W_{\{i\}^\perp}
\end{equation}
by induction on~$\ell(u)=\ell(u')$, the case $\ell(u)=\ell(u')=0$ being obvious.
If $\ell(u)=\ell(u')>0$ then there exists $j\not=i\in I$ such that $\ell(s_j u')<\ell(u')$.
Let $w=s_i u$. Then $\ell(s_j w)<\ell(w)$ and
$\ell(s_i w)<\ell(w)$. Applying Lemma~\ref{lem:Tits} to $w$ and~$J=\{i,j\}$ we conclude that 
$w=(\underbrace{s_i s_j\cdots}_{m_{ij}})u'$ with $\ell(w)=m_{ij}+\ell(u')$ and so 
$u=(\underbrace{s_js_i\cdots}_{m_{ij}-1})u'$ with $\ell(u)=m_{ij}-1+\ell(u')$. Since $u\in W_{I\setminus\{i\}}$, a reduced word for~$u$
cannot contain~$i$, yet for any $(i_1,\dots,i_r)\in R(u')$, $(\underbrace{j, i,\dots}_{m_{ij}-1},i_1,\dots,i_r)\in R(u)$. Thus,  
$m_{ij}=2$ and so $j\in \{i\}^\perp$. Then $s_i(s_j u)=s_j s_i u=(s_ju')s_i$. Thus,  
$s_j u$, $s_j u'$ satisfy~\eqref{eq:ind hyp W} and $\ell(s_j u)<\ell(u)$. Then the induction hypothesis implies that $s_j u=s_j u'\in W_{\{i\}^\perp}$ and hence 
$u=u'\in W_{\{i\}^\perp}$.
\end{proof}

\begin{lemma}\label{lem:perp III}
Let $i$, $i'$ be connected in~$I$ and let $(i=i_0,i_1,\dots,i_d=i')\in I^{d+1}$ be an admissible sequence with $d=\dist(i,i')$.
Suppose that $w\in W_{I\setminus\{i\}}$ and $s_{i_0}\cdots s_{i_d}ws_{i_d}\cdots s_{i_0}\in W_{I\setminus\{i\}}$.
Then 
$w\in W_{\{i_0,\dots,i_d\}^\perp}$.
\end{lemma}
\begin{proof}
The argument is by induction on~$d$. The case~$d=0$ (that is, $i=i'$) is established in Lemma~\ref{lem:perp}.
Suppose that $d>0$. 
Let $u=s_{i_1}\cdots s_{i_d} w s_{i_d}\cdots s_{i_1}$. By Lemma~\ref{lem:perp}, $u\in W_{\{i\}^\perp}$. Since $m_{i,i_1}>2$, $i_1\notin\{i\}^\perp$.
Thus, $u\in W_{I'\setminus\{i_1\}}$ where $I'=I\setminus \{i\}$ and $\dist(i_1,i')=d-1$. By the induction hypothesis,
$u\in W_{\{i_1,\dots,i_d\}^\perp}$ and in particular $u=w$. But then $w\in W_{\{i\}^\perp}\cap W_{\{i_1,\dots,i_d\}^\perp}=
W_{\{i\}^\perp\cap\{i_1,\dots,i_d\}^\perp}=W_{\{i_0,\dots,i_d\}^\perp}$ where we used Lemma~\ref{lem:intersect parabolics}\ref{lem:intersect parabolics.a}
and the observation that $J^\perp\cap J'{}^\perp=(J\cup J')^\perp$.
\end{proof}

By Lemma~\ref{lem:kern natural act}, $\mathsf K_J=\{ w\in W\,:\, uwu^{-1}\in W_J,\,\forall\, u\in W\}$
and is a subgroup of~$W_J$.
Suppose that $w\in \mathsf K_J$; in particular, $w\in W_J$. 
Furthermore, using Lemma~\ref{lem:perp} with $u=w$ and $i\in I\setminus J$, we conclude that 
$w\in \bigcap_{i\in I\setminus J} W_{\{i\}^\perp}=W_{(I\setminus J)^\perp}$. Let $i'\in \partial(I\setminus J)$. By definition,
there exists $i\in I\setminus J$ and an admissible sequence $(i_0,\dots,i_d)$ with $d=\dist(i,i')$, $i_0=i$ and $i_d=i'$. 
Since $u w u^{-1}\in W_J$ with $u=s_{i_0}\cdots s_{i_d}$, it follows from Lemma~\ref{lem:perp III} that $w\in W_{\{ i_0,\dots,i_d\}^\perp}\subset W_{\{i'\}^\perp}$.
Thus, $w\in W_{ (I\setminus J)^\perp\cap \partial(I\setminus J)^\perp}=W_{ (\cl(I\setminus J))^\perp}=W_{ I\setminus \cl(I\setminus J)}$.
We proved that $\mathsf K_J\subset W_{J_0}$ where~$J_0=I\setminus\cl(I\setminus J)$.

To complete the proof of Theorem~\ref{thm:kernel of natural action} we need the following.
\begin{lemma}\label{lem:closure}
Let $J'\subset J$ which is closed in~$I$. Then $W_{J'}\subset \mathsf K_J$.
\end{lemma}
\begin{proof}
Since~$J'$ is closed, $W_J=W_{J'}\times W_{J\setminus J'}$ and $W=W_{J'}\times W_{I\setminus J'}$ by Lemma~\ref{lem:factorization of W}. Then
$W/W_J=W_{I\setminus J'}/W_{J\setminus J'}$. Since $W_{J'}$ acts by left multiplication in the first factor,
this implies that $W_{J'}$ acts trivially on $W/W_J$.
\end{proof}
Applying Lemma~\ref{lem:closure} with $J'=J_0=I\setminus\cl(I\setminus J)$ we conclude that $W_{J_0}\subset \mathsf K_J$. Thus,
$\mathsf K_J=W_{J_0}$. This completes the proof of Theorem~\ref{thm:kernel of natural action}.
\end{proof}

\subsection{Cartan data and Weyl group}
In this section we mostly follow~\cite{Kac}.
Let 
$A=(a_{ij})_{i,j\in I}$ be a symmetrizable generalized Cartan matrix, that is 
$a_{ii}=2$, $i\in I$, $-a_{ij}\in\ZZ_{\ge 0}$ and $a_{ij}=0\implies a_{ji}=0$, $i\not=j$ and 
$d_i a_{ij}=d_j a_{ji}$ for some $\mathbf d=(d_i)_{i\in I}\in \ZZ_{>0}^I$.
We fix the following data:
\begin{enumerate}[label={-},leftmargin=*]
 \item a finite dimensional complex vector space $\lie h$;
 \item linearly independent subsets $\{\alpha_i\}_{i\in I}$ of~$\lie h^*$ and $\{\alpha_i^\vee\}_{i\in I}$ of~$\lie h$;
 \item a symmetric non-degenerate bilinear form $(\cdot,\cdot)$ on~$\lie h^*$, and 
 \item a lattice $P\subset \lie h^*$ of rank $\dim\lie h^*$  
\end{enumerate}
such that
\begin{enumerate}[label={$\arabic*^\circ$},leftmargin=*]
 \item $\alpha_j(\alpha_i^\vee)=a_{ij}$, $i,j\in I$;
 \item $(\alpha_i,\alpha_i)\in 2\ZZ_{>0}$;
 \item $\lambda(\alpha_i^\vee)=2(\lambda,\alpha_i)/(\alpha_i,\alpha_i)$ for all $\lambda\in\lie h^*$;
 \item $\alpha_i\in P$ for all $i\in I$;
 \item $\lambda(\alpha_i^\vee)\in\ZZ$ for all $\lambda\in P$;
 \item $(P,P)\subset \frac1d\ZZ$ for some $d\in\ZZ_{>0}$.
\end{enumerate}
These assumptions imply, in particular, that $\dim\lie h\ge 2|I|-\operatorname{rank}A$.

Denote by~$Q$ (respectively, $Q^+$) the subgroup (respectively, the submonoid) of $P$ generated by the~$\alpha_i$.
Let $P^+=\{ \lambda\in P\,:\, \lambda(\alpha_i^\vee)\in\ZZ_{\ge 0},\,\forall i\in I\}$. 

Define $\omega_i\in \lie h^*$, $i\in I$,
by $\omega_i(\alpha_j^\vee)=\delta_{i,j}$, $j\in J$
and $\omega_i(h)=0$ for all $h\in \bigcap_{i\in I}\ker \alpha_i$.
We will assume that 
$\omega_i\in P$, $i\in I$ and denote by $P_{int}$ (respectively, $P_{int}^+$) the subgroup (respectively, the submonoid) of $P$ generated by the $\omega_i$, $i\in I$.
Given any $J\subset I$, denote $\rho_J=\sum_{j\in J}\omega_j\in P$; we abbreviate $\rho_I=\rho$.

Let $W$ be the Weyl group associated with the matrix~$A$, that is, the Coxeter group 
with $m_{ij}=2$ if $a_{ij}=0$,
$m_{ij}=3$ if $a_{ij}a_{ji}=1$, $m_{ij}=4$ if $a_{ij}a_{ji}=2$, $m_{ij}=6$ if $a_{ij}a_{ji}=3$ and $m_{ij}=0$ if $a_{ij}a_{ji}>3$.
It is well-known that $W$ is finite if and only if $A$ is positive  definite. It should be noted that in that case $\alpha_i\in P_{int}$ for all~$i\in I$.
The group $W$ acts on $\lie h$ (respectively, on~$\lie h^*$)
by $s_ih=h-\alpha_i(h)\alpha_i^\vee$ (respectively, $s_i\lambda=\lambda-\lambda(\alpha_i^\vee)\alpha_i$), $h\in\lie h$, $\lambda\in\lie h^*$ and 
$i\in I$. Then we have $(w\lambda)(h)=\lambda(w^{-1}h)$ for all $w\in W$, $h\in \lie h$ and~$\lambda\in\lie h^*$. Clearly, $W(P)=P$
and $P=P_{int}\oplus P^W$ where $P^W=\{ \lambda\in P\,:\, w\lambda=\lambda,\,\forall\, w\in W\}=\{ \lambda\in P\,:\, \lambda(\alpha_i^\vee)=0,\,\forall\, i\in I\}$.

Given~$J\subset I$ we define 
a linear map $\rho_J^\vee:\lie h^*\to \mathbb C$ by $\rho_J^\vee(\alpha_i)=1$, $i\in J$ and $\rho_J^\vee(\lambda)=0$ if $(\lambda,\alpha_i)=0$ for all~$i\in J$.
As before, we abbreviate $\rho^\vee_I=\rho^\vee$. If~$J\in\mathscr J$ then it can be shown that $\rho_J^\vee(\lambda)$ is equal to $\frac12\lambda(\sum_{h\in R_J^\vee}
h)$ where $R_J^\vee=\{ h\in \lie h\,:\, h\in (\bigcup_{i\in J} W_J\alpha_i^\vee)\cap \sum_{i\in J} \ZZ_{\ge 0}\alpha_i^\vee\}$ is the set of 
positive co-roots of~$W_J$. In particular, this implies that $\rho^\vee_J(P)\subset\frac 12\ZZ$.

If~$J\in\mathscr J$ then for each $j\in J$ 
we have
$w_\circ^J(\alpha_j)=-\alpha_{j^\star}$. 

Given~$\lambda\in P^+$, denote $J_\lambda=\{ i\in I\,:\, \lambda(\alpha_i^\vee)=0\}=\{i\in I\,:\, s_i\lambda=\lambda\}$. 
It is well-known that $\operatorname{Stab}_W\lambda=W_{J_\lambda}$ for~$\lambda\in P^+$.

\subsection{Quantum groups}
We associate with the datum $(A, \lie h, \{\alpha_i\}_{i\in I},\{\alpha_i^\vee\}_{i\in I})$ a complex Lie algebra~$\lie g$ generated by the $e_i$, $f_i$, $i\in I$ and~$h\in \lie h$
subject to the relations 
\begin{gather*}
[h,h']=0,\quad [h,e_i]=\alpha_i(h)e_i,\quad [h,f_i]=-\alpha_i(h)f_i,\quad [e_i,f_j]=\delta_{i,j}\alpha_i^\vee,\quad h,h'\in \lie h,\,i,j\in I\\
(\ad e_i)^{1-a_{ij}}(e_j)=0=(\ad f_i)^{1-a_{ij}}(f_j),\qquad i\not=j.
\end{gather*}
If~$A$ is positive definite then $\lie g$ is a reductive finite dimensional Lie algebra. For~$J\subset I$ we denote by~$\lie g^J$ the subalgebra 
of~$\lie g$ generated by the $e_i$, $f_i$, $i\in J$ and $\lie h$. It can also be regarded as the Lie algebra corresponding to the 
datum $(A|_{J\times J},\lie h,\{\alpha_i\}_{i\in J},\{\alpha^\vee_i\}_{i\in J})$.
In particular, if $J\in\mathscr J$ then $\lie g^J$ is a reductive finite dimensional Lie algebra.

Let $\kk$ be any field of characteristic zero containing $q^{\frac1{2d}}$ which is purely transcendental over~$\mathbb Q$.
Given any~$v\in\kk^\times$ with $v^2\not=1$ define  
$$
(n)_v=\frac{v^{n}-v^{-n}}{v-v^{-1}},\qquad (n)_v!=\prod_{s=1}^n (s)_v,\qquad \binom{n}{k}_v=\prod_{s=1}^k \frac{(n-s+1)_v}{(s)_v}.
$$
Let $q_i=q^{\frac12(\alpha_i,\alpha_i)}$. Henceforth, given any associative algebra~$\mathcal A$ over~$\kk$ and $X_i\in\mathcal A$, $i\in I$
denote $X_i^{(n)}:=X_i^n/(n)_{q_i}!$. We will always use the convention that $X_i^{(n)}=0$ if~$n<0$.

Define the Drinfeld-Jimbo quantum group $U_q(\lie g)$ corresponding to~$\lie g$ as the associative algebra over~$\kk$ with generators
$K_\lambda$, $\lambda\in\frac12 P$ and $E_i,F_i$, $i\in I$ subject to the relations 
\begin{gather*}
K_\lambda E_i=q^{(\lambda,\alpha_i)}E_i K_\lambda,\,\, K_\lambda F_i=q^{-(\lambda,\alpha_i)}F_i K_\lambda,\,\,
[E_i,F_j]=\delta_{ij}\, \frac{ K_{\alpha_i}-K_{-\alpha_i}}{q_i-q_i^{-1}},\quad \lambda\in\tfrac12 P,\, i,j\in I\\
\sum_{r+s=1-a_{ij}} (-1)^r E_i^{(r)}E_j E_i^{(s)}=0=\sum_{r+s=1-a_{ij}} (-1)^r F_i^{(r)}F_j F_i^{(s)}.
\end{gather*}
This is a Hopf algebra with with 
the ``balanced'' comultiplication
\begin{equation}\label{eq:comult}
\Delta(E_i)=E_i\tensor K_{\frac12\alpha_i}+K_{-\frac12\alpha_i}\tensor E_i,\,\, 
\Delta(F_i)=F_i\tensor K_{\frac12\alpha_i}+K_{-\frac12\alpha_i}\tensor F_i,\, i\in I,
\end{equation}
while $\Delta(K_\lambda)=K_\lambda\tensor K_\lambda$, $\lambda\in\frac12 P$. Denote $U_q^+(\lie g)$ (respectively, $U_q^-(\lie g)$) the subalgebra of~$U_q(\lie g)$
generated by the $E_i$ (respectively, the $F_i$), $i\in I$. Then $U_q^\pm(\lie g)$ is graded by $\pm Q^+$ with $\deg E_i=\alpha_i=-\deg F_i$. 
Given~$\nu\in Q^+$, denote by~$U_q^\pm(\lie g)(\pm\nu)$
the subspace of homogeneous elements of~$U_q^\pm(\lie g)$
of degree~$\pm\nu$.

Given $J\subset I$ we denote by $U_q(\lie g^J)$ the subalgebra of~$U_q(\lie g)$ generated by the $E_j$, $F_j$, $j\in I$ and $K_\lambda$, $\lambda\in\frac12 P$
and set $U_q^\pm(\lie g^J)=U_q^\pm(\lie g)\cap U_q(\lie g^J)$.

If~$J\in\mathscr J$ then the algebra~$U_q(\lie g^J)$ admits an automorphism $\theta_J$ defined by $\theta_J(E_i)=F_{i^\star}$,
$\theta_J(F_i)=E_{i^\star}$ and $\theta_J(K_\lambda)=K_{w_\circ^J\lambda}$, $\lambda\in \frac12 P$. If~$I\in\mathscr J$ we abbreviate $\theta_I=\theta$.

\subsection{Integrable modules}
We say that a $U_q(\lie g)$-module $M$ is {\em integrable} if $M=\bigoplus\limits_{\beta\in P}M(\beta)$ where $M(\beta)=\{ m\in M\,:\, 
K_\lambda(m)=q^{(\lambda,\beta)}m,\,\forall \lambda\in\frac12 P\}$ and the $E_i$, $F_i$, $i\in I$ act locally nilpotently on~$M$. 
Given any $m\in M$ we can write uniquely 
\begin{equation}\label{eq:weight-components}m=\bigoplus_{\beta\in P} m(\beta)
\end{equation}
where $m(\beta)\in M(\beta)$ and $m(\beta)=0$ for all but finitely many $\beta\in P$.
Denote $\supp m=\{ \beta\in P\,:\, m(\beta)\not=0\}$. By definition, if $m\in M(\beta)$ and $u_\pm \in U_q^\pm(\lie g)(\pm\nu)$,
$\nu\in \pm Q^+$ then $u_\pm(m)\in M(\beta\pm \nu)$. We say that $m\in M(\beta)$, $\beta\in P$ is homogeneous of weight~$\beta$ and 
call $M(\beta)$ a weight subspace of~$M$.

\begin{definition}\label{def:cat O}
The category $\mathscr O^{int}_q(\lie g)$ is the full subcategory of the category of~$U_q(\lie g)$-modules whose objects 
are integrable $U_q(\lie g)$-modules~$M$ 
with the following property: given $m\in M$, there exists~$N(m)\ge 0$ such that $U_q^+(\lie g)(\nu)(m)=0$ for all $\nu\in Q^+$ with
${\rho^\vee(\nu)\ge N(m)}$. 
\end{definition}

Given $V\in\mathscr O^{int}_q(\lie g)$, let $V_+=\bigcap_{i\in I} \ker E_i$ where the $E_i$ are regarded as linear endomorphisms of~$V$. 
For any subset~$S$ of an object~$V$ in~$\mathscr O^{int}_q(\lie g)$ we denote $S(\beta)=S\cap V(\beta)$ and $S_+=S\cap V_+$.

It is well-known (see e.g.~\cite{Lus}*{Theorem~6.2.2}) that 
$\mathscr O^{int}_q(\lie g)$ is semisimple and its simple objects are simple highest weight modules $V_\lambda$, $\lambda\in P^+$ with $(V_\lambda)_+=(V_\lambda)_+(\lambda)$ one-dimensional.
Furthermore, every $V\in\mathscr O^{int}_q(\lie g)$ is generated by $V_+$ as a $U_q(\lie g)$-module and $V_+(\lambda)\not=0$ implies that~$\lambda\in P^+$.
Given $\lambda\in P^+$ and $V\in \mathscr O^{int}_q(\lie g)$ denote $\mathcal I_\lambda(V)$ the $\lambda$-isotypical component of~$V$
as a $U_q(\lie g)$-module. Thus, every simple submodule (and hence a direct summand) of~$\mathcal I_\lambda(V)$ is isomorphic to~$V_\lambda$
and $\mathcal I_\lambda(V)_+=V_+(\lambda)$.
Furthermore, for any $v\in V_+$ we have the following equality of~$U_q(\lie g)$-submodules of~$V$
\begin{equation}\label{eq:sum  submodules}
U_q(\lie g)(v)=\sum_{\lambda\in\supp v} U_q(\lie g)(v(\lambda)),
\end{equation}
where the sum is direct and each summand is simple and isomorphic to~$V_\lambda$.

It is immediate from the definition that every object~$\mathscr O^{int}_q(\lie g)$ can be regarded as an object in~$\mathscr O^{int}_q(\lie g^J)$, $J\subset I$.
Denote $P^+_J=\{ \mu\in P\,:\,
\mu(\alpha_j^\vee)\in\ZZ_{\ge 0},\,\forall\, j\in J\}$. 
Given $V\in \mathscr O^{int}_q(\lie g)$ and $\lambda_J\in P^+_J$ 
denote by $\mathcal I^J_{\lambda_J}(V)$  the $\lambda_J$-isotypical component of~$V$ as a $U_q(\lie g^J)$-module.
Clearly, $\mathcal I^J_{\lambda_J}(\mathcal I_\lambda(V))=\mathcal I_\lambda(V)\cap \mathcal I^J_{\lambda_J}(V)$ for any $V\in\mathscr O^{int}_q(\lie g)$,
$\lambda\in P^+$, $\lambda_J\in P^+_J$. 
We denote $V_+^J=\bigcap_{j\in J} \ker E_j\subset V_+$. Then $\mathcal I^J_{\lambda_J}(V)$ is generated by $V_+^J(\lambda_J)$
as a $U_q(\lie g^J)$-module.

\subsection{Crystal operators, lattices and bases}\label{subs:prelim crystal bases}
Here we recall some necessary facts from Kashiwara's theory of crystal bases. To treat lower and upper 
crystal operators and lattices uniformly, we find it convenient to interpolate between them using {\em $\mathbf c$-crystal
operators} and lattices (for other generalizations see e.g.~\cites{CJM,GLam}).

The following fact is standard (for example, see~\cite{Lus}*{Lemma~16.1.4}). 
\begin{lemma}\label{lem:sl2 decomp}
Let $V\in\mathscr O^{int}_q(\lie g)$ 
and fix~$i\in I$. Then 
$$
V=\bigoplus_{0\le n\le l} F_i^n(\ker E_i\cap \ker(K_{\alpha_i}-q_i^l))
=\bigoplus_{0\le n\le l} E_i^n(\ker F_i\cap \ker(K_{\alpha_i}-q_i^{-l}))
$$
\end{lemma}
Let
$\mathbb D=\{ (l,k,s)\in\mathbb D\,:\, k-l\le s\le k\le l\}$.
Fix a map $\mathbf c:\mathbb D\to\mathbb Q(z)^\times$ and denote its value at~$(l,k,s)$ by $\mathbf c_{l,k,s}$. We use the convention that $\mathbf c_{l,k,s}=0$ whenever~$(l,k,s)\in\ZZ^3\setminus\mathbb D$. 
Using Lemma~\ref{lem:sl2 decomp} we can define {\em generalized Kashiwara operators} $\tilde e_{i,s}^{\mathbf c}\in\End_\kk V$, $s\in\ZZ$ by 
\begin{equation}\label{eq:kash-op-defn}
\tilde e_{i,s}^{\mathbf c}(F_i^k(u))=\mathbf c_{l,k,s}(q_i) F_i^{k-s}(u),
\end{equation}
for every $u\in \ker E_i\cap \ker(K_{\alpha_i}-q_i^l)$, $0\le k\le l$. 
Note that under these assumptions on~$u$,
$\tilde e_{i,s}^{\mathbf c}(F_i^{(k)}(u))\not=0$ if and only if~$(l,k,s)\in\mathbb D$.
Clearly, like lower or upper Kashiwara operators, the generalized ones commute with morphisms in~$\mathscr O^{int}_q(\lie g)$. 
\begin{lemma}\label{lem:2nd defn Kash op}
Let $u\in\ker E_i\cap \ker( K_{\alpha_i}-q_i^l)$ and 
$u'\in\ker F_i\cap \ker (K_{\alpha_i}-q_i^{-l})$, $0\le k\le l$. Then 
\begin{equation}\label{eq:kash-op-defn-alt}
\tilde e_{i,s}^{\mathbf c}(F_i^{(k)}(u))=\underline{\mathbf c}_{l,k,s}(q_i) F_i^{(k-s)}(u),\quad
\tilde e_{i,s}^{\mathbf c}(E_i^{(k)}(u'))
=\underline{\mathbf c}_{l,l-k,s}(q_i) E_i^{(k+s)}(u'),
\end{equation}
for all~$(l,k,s)\in\mathbb D$, where $\underline{\mathbf c}_{l,k',s'}={\mathbf c}_{l,k',s'}(k'-s')_z!/(k')_z!$.
\end{lemma}
\begin{proof}
The first identity in~\eqref{eq:kash-op-defn-alt} is immediate from~\eqref{eq:kash-op-defn}. 
To prove the second, note that 
$E_i^{(l+1)}(u')=0$ and so $u=E_i^{(l)}(u')\in\ker E_i\cap \ker (K_{\alpha_i}-q_i^l)$. 
It follows from~\cite{Lus}*{\S3.4.2} that $E_i^{(k)}(u')=F_i^{(l-k)}(u)$. 
Using the first identity in~\eqref{eq:kash-op-defn-alt} we obtain
\begin{equation*}
\tilde e_{i,s}^{\mathbf c}(E_i^{(k)}(u'))=\tilde e_{i,s}^{\mathbf c}(F_i^{(l-k)}(u))=\underline{\mathbf c}_{l,l-k,s}(q_i) F_i^{(l-k-s)}(u)=
\underline{\mathbf c}_{l,l-k,s}(q_i) E_i^{(k+s)}(u').\qedhere
\end{equation*}
\end{proof}
The following is immediate.
\begin{lemma}\label{lem:composition and identity}
Given $\mathbf c:\mathbb D\to \mathbb Q(z)^\times$ we have 
\begin{enumerate}[label={\rm(\alph*)},leftmargin=*]
 \item $\tilde e_{i,0}^{\mathbf c}=\id_V$ for all $V\in\mathscr O^{int}_q(\lie g)$ if and only if $\mathbf c_{l,k,0}=1$ for all $0\le k\le l$;
 \item 
$\tilde e_{i,t}^{\mathbf c}\circ\tilde e_{i,s}^{\mathbf c}=\tilde e_{i,s+t}^{\mathbf c}$ for all $s,t\in\ZZ$ with $st\ge 0$ 
and for all $V\in\mathscr O^{int}_q(\lie g)$
if and only if~$\mathbf c_{l,k,s+t}=
\mathbf c_{l,k,s}\mathbf c_{l,k-s,t}$ for all $0\le k\le l$, $s,t\in\ZZ$, $st\ge 0$.
\end{enumerate}
\end{lemma}

The following is easy to deduce from~\cite{Kas93}*{\S3.1}
\begin{lemma}
Define $\mathbf c^{low},\mathbf c^{up}:\mathbb D\to\mathbb Q(z)^\times$ by
\begin{equation}\label{eq:low upper Kash}
\mathbf c^{low}_{l,k,s}=\frac{(k)_z!}{(k-s)_z!},\qquad 
\mathbf c^{up}_{l,k,s}=\frac{(l-k+s)_z!}{(l-k)_z!},\qquad (l,k,s)\in\mathbb D.
\end{equation}
We have  
$$
(\tilde e_i^{low})^s=\tilde e^{\mathbf c^{low}}_{i,s},\qquad (\tilde e_i^{up})^s=\tilde e^{\mathbf c^{up}}_{i,s},\qquad i\in I,\,s\in\ZZ,
$$
where $\tilde e_i^{low}$ (respectively, $\tilde e_i^{up}$) are lower (respectively, upper) Kashiwara's operators
as defined in~\cite{Kas93}*{\S3.1}.
\end{lemma}

Fix $\mathbf c:\mathbb D\to \mathbb Q(z)^\times$ and let $V\in\mathscr O^{int}_q(\lie g)$.
Let~$\mathbb A$ be the local subring of $\mathbb Q(q)\subset \kk$ consisting of rational functions regular at~$0$.
Generalizing well-known definitions of Kashiwara, we say that an $\mathbb A$-submodule~$L$ of~$V$ is a $\mathbf c$-crystal lattice 
if $V=\kk\tensor_{\mathbb A} L$, $L=\bigoplus_{\beta\in P}(L\cap V(\beta))$,
and $\tilde e_{i,s}^{\mathbf c}(L)\subset L$ for all $i\in I$, $s\in\mathbb Z$.

We will be mostly interested in a special class of crystal lattices which we refer to as monomial.
We need the following notation. Given $v\in V$ set 
$$\mathsf M_J^{\mathbf c}(v)=\{v\}\cup \bigcup_{k\in\ZZ_{>0}} \{ \tilde e^{\mathbf c}_{i_1,m_1}\cdots\tilde e^{\mathbf c}_{i_k,m_k}(v)\,:\,
(i_1,\dots,i_k)\in J^k,\, (m_1,\dots,m_k)\in\ZZ^k\}.
$$
We abbreviate $\mathsf M^{\mathbf c}(v)=\mathsf M_I^{\mathbf c}(v)$. 
We call an $\mathbb A$-submodule $L$ of~$V$ a {\em $(\mathbf c,J)$-monomial lattice} if $$
L=\sum_{v_+} \mathsf M_J^{\mathbf c}(v_+)
$$
where the sum is over all $v_+\in L\cap V_+^J(\lambda_J)$, $\lambda_J\in P^+$.
Clearly, $L$ inherits a weight decomposition from~$V$ and $\tilde e_{j,a}^{\mathbf c}(L)\subset L$ for all $j\in J$, $a\in\ZZ$.
In particular, if $L$ is a $(\mathbf c,I)$-monomial and $\kk\tensor_{\mathbb A}L=V$ then $L$ is a $\mathbf c$-crystal lattice.

Denote $\tilde L$ the $\mathbb Q$-vector space $L/q L$. Given $\tilde v\in \tilde L$, denote 
$$\tilde{\mathsf M}^{\mathbf c}_J(\tilde v)=\{\tilde v\}\cup \bigcup_{k\in\ZZ_{>0}} \{ \tilde e^{\mathbf c}_{i_1,m_1}\cdots\tilde e^{\mathbf c}_{i_k,m_k}(\tilde v)\,:\,
(i_1,\dots,i_k)\in J^k,\, (m_1,\dots,m_k)\in\ZZ^k\}\subset \tilde L.
$$
As before, we abbreviate $\tilde{\mathsf M}^{\mathbf c}(v)=\tilde{\mathsf M}_I^{\mathbf c}(v)$.

By~\cite{Kas91}*{Theorem~3} and~\cite{Kas93}*{Theorem~3.3.1}, if $\mathbf c=\mathbf c^{low}$ or $\mathbf c=\mathbf c^{up}$ then
every object in~$\mathscr O^{int}_q(\lie g)$ admits a $\mathbf c$-crystal lattice. Moreover, in that case for any $\lambda\in P^+$,
and any $v_\lambda\in V_\lambda(\lambda)$ the 
smallest $\mathbb A$-submodule of~$V_\lambda$ containing $v_\lambda$ and 
invariant with respect to the $\tilde e_{i,s}^{\mathbf c}$, $i\in I$, $s\in\ZZ_{<0}$ is a $\mathbf c$-crystal lattice.

Let $L$ be a $\mathbf c$-crystal lattice of~$V\in\mathscr O^{int}_q(\lie g)$. Clearly, operators $\tilde e_{i,s}^{\mathbf c}$
commute with the action of~$q$ on~$L$ and thus
factor through to $\mathbb Q$-linear operators on~$\tilde L=L/q L$ denoted by the same symbols.
Similarly to~\cites{Kas91,Kas93}, we say that $(L,B)$, where $B$ is a weight basis of~$\tilde L$, 
is a 
$\mathbf c$-crystal basis of~$V$ at $q=0$ 
if $\tilde e_{i,s}^{\mathbf c}(B)\subset B\cup\{0\}$, $i\in I$, $s\in\ZZ$. 
By~\cites{Kas91,Kas93}, every object in~$\mathscr O^{int}_q(\lie g)$ admits a crystal basis provided that $\mathbf c\in\{\mathbf c^{up},\mathbf c^{low}\}$.

The following is well-known (cf.~\cites{Kas91,Kas93}).
\begin{lemma}\label{lem:facts on crystal bases}
Let $V\in\mathscr O^{int}_q(\lie g)$, $\mathbf c\in\{\mathbf c^{low},\mathbf c^{up}\}$ and let $(L,B)$ be a $\mathbf c$-crystal basis at~$q=0$. Then for any $J\subset I$ 
\begin{enumerate}[label={\rm(\alph*)},leftmargin=*]
 \item \label{lem:facts on crystal bases.a} $L$ is a $(\mathbf c,J)$-monomial lattice; 
 \item \label{lem:facts on crystal bases.b} $\tilde{\mathsf M}^{\mathbf c}_J(b)\subset B\cup\{0\}$ for any $b\in B$;
 \item \label{lem:facts on crystal bases.c} $B=\bigcup_{b_+\in B_+^J} \tilde{\mathsf M}^{\mathbf c}_J(b_+)\setminus \{0\}$ where $B_+^J=\bigcap_{j\in J}\ker\tilde e_{j,1}^{\mathbf c}\subset B$;
\end{enumerate}
\end{lemma}
\begin{remark}\label{rem:restr crys bas}
It is not hard to see that if for given $\mathbf c:\mathbb D\to \mathbb Q(z)^\times$ and~$J\subset I$
Lemma~\ref{lem:facts on crystal bases}\ref{lem:facts on crystal bases.a}--\ref{lem:facts on crystal bases.c} hold then 
$(L,B)$ is a $\mathbf c$-basis at $q=0$ of $V$ regarded as a $U_q(\lie g^J)$-module.
\end{remark}

\section{Properties of \texorpdfstring{$\sigma^i$}{eta i} and proof of Theorem~\ref{thm:crystal weyl group}}\label{sect:crystal weyl group}
\subsection{Special monomials in \texorpdfstring{$U_q^\pm(\lie g)$}{U\_q +-(g)}}\label{subs:spec mon}
Given any reduced sequence $\mathbf i=(i_1,\dots,i_m)\in I^m$ and~$\lambda\in P^+$ we define $F_{\mathbf i,\lambda}\in U_q^-(\lie g)$ and 
$E_{\mathbf i,\lambda}\in U_q^+(\lie g)$, $\lambda\in P^+$ by
\begin{equation}\label{eq:F i lambda}
F_{\mathbf i,\lambda}=F_{i_1}^{(a_{1})}
\cdots F_{i_m}^{(a_m)},\qquad 
E_{\mathbf i,\lambda}=E_{i_1}^{(a_1)}
\cdots E_{i_m}^{(a_m)},
\end{equation}
where $a_k=a_{k}(\mathbf i,\lambda)=s_{i_{k+1}}\cdots s_{i_m}\lambda(\alpha_{i_k}^\vee)=\lambda(s_{i_m}\cdots s_{i_{k+1}}\alpha_{i_k}^\vee)\in\ZZ_{\ge 0}$.
\begin{lemma}\label{lem:EF lambda w}
Let $w\in W$, $\lambda\in P^+$, $i\in I$. Then 
\begin{enumerate}[label={\rm(\alph*)},leftmargin=*]
 \item\label{lem:EF lambda w.a}
$F_{\mathbf i,\lambda}=F_{\mathbf i',\lambda}$ and $E_{\mathbf i,\lambda}=E_{\mathbf i',\lambda}$ for any $\mathbf i,\mathbf i'\in R(w)$ and $\lambda\in P^+$.
Thus, we can define $F_{w,\lambda}:=F_{\mathbf i,\lambda}$ and $E_{w,\lambda}:=E_{\mathbf i,\lambda}$ for some $\mathbf i\in R(w)$;
\item\label{lem:EF lambda w.b} 
If $\ell(s_i w)=\ell(w)+1$ then $F_{s_iw,\lambda}=F_i^{(w\lambda(\alpha_i^\vee))}F_{w,\lambda}$
and $E_{s_i w,\lambda}=E_i^{(w\lambda(\alpha_i^\vee))}E_{w,\lambda}$;
\item\label{lem:EF lambda w.c} If $\ell(s_i w)=\ell(w)-1$ then $F_{w,\lambda}=F_i^{(-w\lambda(\alpha_i^\vee))} F_{s_i w,\lambda}$ and~$E_{w,\lambda}=E_i^{(-w\lambda(\alpha_i^\vee))} 
E_{s_i w,\lambda}$;
\item\label{lem:EF lambda w.c'} If $s_i\lambda=\lambda$ then $F_{ws_i,\lambda}=F_{w,\lambda}$;
\item\label{lem:EF lambda w.c''} $\deg F_{w,\lambda}=-\deg E_{w,\lambda}=w\lambda-\lambda$;
\item\label{lem:EF lambda w.d} Suppose that $W$ is finite. Then
$\theta(F_{w,\lambda})=E_{w_\circ ww_\circ,-w_\circ\lambda}$.
\end{enumerate}
\end{lemma}

\begin{proof}
It is well-known that~$\mathbf i$ can be obtained from~$\mathbf i'$ by a finite sequence of rank~$2$ braid moves
of the form $\underbrace{s_i s_j\cdots}_{m_{ij}}=\underbrace{s_j s_i\cdots}_{m_{ij}}$ with $m_{ij}$ finite. Thus, it 
suffices to prove part~\ref{lem:EF lambda w.a} in case when $w$ is the longest element in the subgroup of~$W$ generated 
by $s_i$, $s_j$, $i\not=j\in I$. But in that case it was established in~\cite{Lus}*{Proposition~39.3.7}.

Parts~\ref{lem:EF lambda w.b}, \ref{lem:EF lambda w.c} and~\ref{lem:EF lambda w.c'} are obtained from part~\ref{lem:EF lambda w.a} by choosing appropriate
reduced decompositions.
To prove parts~\ref{lem:EF lambda w.c''} and~\ref{lem:EF lambda w.d} we use 
induction on~$\ell(w)$, the case $\ell(w)=0$ being obvious. For the inductive step, suppose that $\ell(s_iw)=\ell(w)+1$. 
Since $\mu(\alpha_i^\vee)\alpha_i=\mu-s_i\mu$, $\mu\in P$,
by part~\ref{lem:EF lambda w.b} and the induction hypothesis we have $\deg F_{s_iw,\lambda}=-w\lambda(\alpha_i^\vee)\alpha_i+w\lambda-\lambda
=s_iw\lambda-\lambda$. This proves the inductive step in part~\ref{lem:EF lambda w.c''}. 
Since $s_{i^\star}=w_\circ s_i w_\circ$ we have by Lemma~\ref{lem:EF lambda w}\ref{lem:EF lambda w.b}
\begin{align*}
\theta(F_{s_i w,\lambda})&=\theta(F_i^{(w\lambda(\alpha_i^\vee))}F_{w,\lambda})
=E_{i^\star}^{(w\lambda(\alpha_i^\vee))}E_{w_\circ w w_\circ,-w_\circ\lambda}=E_{i^{\star}}^{(-w\lambda(w_\circ\alpha^\vee_{i^\star}))}E_{w_\circ w w_\circ,-w_\circ\lambda}
\\&=E_{i^{\star}}^{((w_\circ ww_\circ(-w_\circ\lambda))(\alpha^\vee_{i^\star}))}E_{w_\circ w w_\circ,-w_\circ\lambda}
=E_{s_{i^\star}w_\circ ww_\circ,-w_\circ\lambda}=E_{w_\circ s_i w w_\circ,-w_\circ\lambda}.
\end{align*}
The inductive step in part~\ref{lem:EF lambda w.d} is proven.
\end{proof}

\subsection{Extremal vectors}\label{subs:extrem vecs}
Let $V\in\mathscr O^{int}_q(\lie g)$. Given $v\in V_+(\lambda)\setminus\{0\}$, $w\in W$, define the {\em standard extremal vectors} 
$[v]_w$ of~$v$ by 
$[v]_{w}:=F_{w,\lambda}(v)$. 
Furthermore, given $v\in V_+\setminus\{0\}$, define 
\begin{equation}\label{eq:v_W}
[v]_W:=\{ [v(\lambda)]_{w}\,:\, w\in W,\,\lambda\in \supp v\}.
\end{equation}
\begin{proposition}\label{prop:w-translates of hw}
Let $V\in\mathscr O^{int}_q(\lie g)$. Then for any $v\in V_+\setminus\{0\}$
\begin{enumerate}[label={\rm(\alph*)},leftmargin=*]
\item\label{prop:w-translates of hw.a}
if $v$ is homogeneous then  
$[v]_{w}=[v]_{w'}$ if and only if $w'\in w W_{J(v)}$. In particular, the assignments $[v]_w\mapsto w W_{J(v)}$ define a bijection $\mathsf J_v:[v]_W\to W/W_{J(v)}$;
\item\label{prop:w-translates of hw.b}
for any $v\in V_+\setminus\{0\}$ the set $[v]_W$ is linearly independent.
\end{enumerate}
\end{proposition}
\begin{proof}
To prove~\ref{prop:w-translates of hw.a}, let $v\in V_+(\lambda)\setminus\{0\}$ for some $\lambda\in P^+$ and 
recall that $\operatorname{Stab}_W\lambda=W_{J_\lambda}$.
It follows from Lemma~\ref{lem:EF lambda w}\ref{lem:EF lambda w.c'}
by an obvious induction on~$\ell(w'')$ that $F_{ww'',\lambda}=F_{w,\lambda}$ for all $w''\in W_{J_\lambda}$. 
Since $w\lambda=w'\lambda$ implies that $w'=w w''$ for some~$w''\in W_{J_\lambda}$, it follows that $F_{w',\lambda}(v)=F_{w,\lambda}(v)$.
Conversely, by Lemma~\ref{lem:EF lambda w}\ref{lem:EF lambda w.c''} we have $F_{w,\lambda}(v)\in V(w\lambda)$. Thus, 
if $w\lambda\not=w'\lambda$ then $F_{w,\lambda}(v)$ and $F_{w',\lambda}(v)$ are
in different weight subspaces of~$V$ and are linearly 
independent (and hence not equal).

In particular, we proved that $[v(\lambda)]_W$ is linearly independent for all $v\in V_+\setminus\{0\}$ and $\lambda\in\supp v$.
This, together with~\eqref{eq:sum submodules}, proves part~\ref{prop:w-translates of hw.b}.
\end{proof}

Given $v\in V$, denote $J(v)=\{i\in I\,:\, F_i(v)=0\}$ and define $J(S)=\bigcap_{v\in S} J(v)$, $S\subset V$.
\begin{remark}
It is not hard to show that $E_i(V)=\{0\}$ for every $i\in J(V)$.
\end{remark}
We will need the following basic properties of these sets.
\begin{proposition}\label{prop:stabilizer}
Let $V\in\mathscr O^{int}_q(\lie g)$. Then 
\begin{enumerate}[label={\rm(\alph*)},leftmargin=*]
\item\label{prop:stabilizer.a} $J(v)=J_\lambda$ for any $v\in V_+(\lambda)\setminus\{0\}$, $\lambda\in P^+$;
\item\label{prop:stabilizer.e} There exists $v\in V_+$ such that $J(V)=J(U_q(\lie g)(v))$.
\end{enumerate}
\end{proposition}
\begin{proof}
It is well-known (see e.g.~\cite{Lus}*{Chap. 6}) that the annihilating ideal of~$v$ in~$U_q^-(\lie g)$ 
is generated by the $F_i^{\lambda(\alpha_i^\vee)+1}$, $i\in I$. Thus, $F_i(v)=0$ if and only if~$\lambda(\alpha_i^\vee)=0$.
This proves part~\ref{prop:stabilizer.a}. 
To prove part~\ref{prop:stabilizer.e}, note the following obvious fact.
\begin{lemma}\label{lem:annihilators}
Let $R$ be a ring and let $M=\bigoplus_{\alpha\in A} M_\alpha$ as $R$-modules. Let~$S$ be a subset of~$R$. Then 
$\Ann_S M=\bigcap_{\alpha\in A'} \Ann_S M_\alpha=\Ann_S M'$ where $A'$ is any subset of~$A$ such that for each $\alpha\in A$
there exists $\alpha'\in A'$ such that $M_\alpha\cong M_{\alpha'}$ and $M'=\bigoplus_{\alpha\in A'} M_\alpha$. In particular, if $S$ is finite 
then $\Ann_S M=\bigcap_{\alpha\in A_0}\Ann_S M_\alpha=\Ann_S M_0$ where $A_0$ is a finite subset of~$A'$ and $M_0=\bigoplus_{\alpha\in A_0} M_\alpha$.
\end{lemma}

Apply this Lemma to $R=U_q(\lie g)$ and $S=\{ F_i\,:\, i\in I\}$, which identifies with~$I$, and $M=V$. Clearly, $J(V)=\{ i\in I\,:\, F_i\in \Ann_S V\}$.
Since~$S$ is finite and $V$ is a direct sum of simple modules, it follows from Lemma~\ref{lem:annihilators} that
$\Ann_S V=\Ann_S V'$ where $V'=\bigoplus_{\lambda\in\Omega} U_q(\lie g)(v_\lambda)$ for some finite $\Omega\subset 
\{\lambda\in P^+\,:\,\mathcal I_\lambda(V)\not=0\}$ and $v_\lambda\in V_+(\lambda)\setminus\{0\}$, $\lambda\in\Omega$. 
Since $V'=U_q(\lie g)(v)$ with $v=\sum_{\lambda\in\Omega} v_\lambda$, part~\ref{prop:stabilizer.e} follows.
\end{proof}

\begin{proposition}\label{prop:kernels}
Let $V\in\mathscr O^{int}_q(\lie g)$. For each $i\in I$, $v\in V_+\setminus\{0\}$ the following are equivalent.
\begin{enumerate}[label={\rm(\alph*)},leftmargin=*]
\item\label{prop:kernels.a} $(\lambda,w\alpha_i)=0$ for all $\lambda\in\supp v$, $w\in W$;
 \item\label{prop:kernels.d} 
 $\cl(\{i\})\in J(U_q(\lie g)(v))$.
\end{enumerate}
\end{proposition}
\begin{proof}
 Let $J$ be the neighborhood of~$i$. In particular, $J\cup J^\perp=I$.

\ref{prop:kernels.a}$\implies$\ref{prop:kernels.d}
We need the following Lemma.
\begin{lemma}\label{lem:train}
For every $i\sim j\in I$, $i\not=j$ there exists $w=w_{i,j}\in W$ such that 
$w_{i,j}\alpha_i\in \ZZ_{>0} \alpha_j+\sum_{k\in I\setminus\{j\}} \ZZ_{\ge 0}\alpha_k$.
\end{lemma}
\begin{proof}
Let $\mathbf i=(i=i_0,i_1,\dots,i_d=j)\in I^{d+1}$ be an admissible sequence with $d=\dist(i,j)$. In particular, this sequence is repetition free.
Denote $\beta_k=s_{i_k}\cdots s_{i_1}(\alpha_i)$. We claim that $\beta_k\in \sum_{0\le r\le k} \ZZ_{>0} \alpha_{i_r}$.
We argue by induction on~$k$, the case $k=0$ being obvious. For the inductive step, note that $$
\beta_k=s_{i_k}(\beta_{k-1})
\in \sum_{0\le r\le k-1}\ZZ_{>0} s_{i_k}\alpha_{i_r}=\sum_{0\le r\le k-1} \ZZ_{>0} (\alpha_{i_r}-\alpha_{i_r}(\alpha_{i_k}^\vee)\alpha_{i_k}).
$$
Since $\mathbf i$ is admissible, 
$\alpha_{i_{k-1}}(\alpha_{i_k}^\vee)<0$ while $\alpha_{i_r}(\alpha_{i_k}^\vee)\le 0$ for all $0\le r\le k-2$. 
Therefore, $\beta_k\in\sum_{0\le r\le k} \ZZ_{>0} \alpha_{i_r}$. In particular, $w=s_{i_d}\cdots s_{i_1}$ is the desired $w_{i,j}\in W$.
\end{proof}
 Write $\lambda\in\supp v$ as $\lambda=\lambda'+\sum_{k\in I} l_k\omega_k$ where 
$\lambda'\in P^W$ and $l_k\in\ZZ_{\ge 0}$, $k\in I$. Let~$j\in J$. In the notation of Lemma~\ref{lem:train}, we have 
$w_{i,j}\alpha_i=\sum_{k\in I} n_k\alpha_k$ with $n_j\in\ZZ_{>0}$ and $n_k\in\ZZ_{\ge 0}$, $k\in I\setminus\{j\}$, for some~$w_{i,j}\in W$. Then $0=2(\lambda,w_{i,j}\alpha_i)
=(\alpha_j,\alpha_j)l_jn_j+\sum_{k\in I\setminus\{j\}} (\alpha_k,\alpha_k) l_k n_k$. Since $n_j>0$ and $n_k\ge 0$ for all $k\not=j$
this forces $l_j=0$.

In particular, $J\subset J(v)$.
Since $J$ is closed in~$I$,
$[U_q^-(\lie g^{J^\perp}),F_j]=0$.
Let $V'=U_q(\lie g)(v)$. Then $V'=U_q^-(\lie g)(v)=U_q^-(\lie g^{J^\perp})(v)$ and so $J\subset J(V')$.

\ref{prop:kernels.d}$\implies$\ref{prop:kernels.a}
Since $J\subset J(U_q(\lie g)(v))$ it follows that $(\lambda,\alpha_j)=0$ for all~$j\in J$. Since
$W\alpha_j\in\sum_{j'\in J}\ZZ \alpha_{j'}$ for any $j\in J$, the assertion follows.
\end{proof}

\begin{lemma}\label{lem:set I_V}
Let $V\in\mathscr O^{int}_q(\lie g)$ such that $V=U_q(\lie g)(v)$ for some~$v\in V_+$. The following are equivalent for~$i\in I$.
\begin{enumerate}[label={\rm(\alph*)},leftmargin=*]
\item\label{lem:set I_V.a} $i\in J(V)$.
\item\label{lem:set I_v.b} $[v(\lambda)]_{s_i w}=[v(\lambda)]_w$ for all $\lambda\in\supp v$ and for all $w\in W$;
\end{enumerate}
\end{lemma}
\begin{proof}
The condition in part~\ref{lem:set I_v.b} implies that $s_i w\lambda=w\lambda$ for all $\lambda\in\supp v$ and $w\in W$.
Since $s_i w\lambda=w\lambda-(w\lambda)(\alpha_i^\vee)\alpha_i$, it follows that $(\lambda,w\alpha_i)=0$ for all $\lambda\in \supp v$
and $w\in W$ and so $i\in J(V)$ by Proposition~\ref{prop:kernels}.

Conversely, if $i\in J(V)$ then $F_i([v]_w)=0$ for all $w\in W$. In particular, if $\ell(s_i w)=\ell(w)+1$
then, since $[v]_{s_i w}=F_i^{(w\lambda(\alpha_i^\vee))}[v]_w\not=0$ it follows that $(w\lambda,\alpha_i)=0$
and thus $[v]_{s_i w}=[v]_w$. Similarly, if $\ell(s_i w)=\ell(w)-1$ applying the previous argument 
to $w'=s_i w$ we obtain the same equality.
\end{proof}

\subsection{Proof of Theorem~\ref{thm:crystal weyl group}}\label{subs:cryst weyl group}
We will now express the action of the~$\sigma^i$, $i\in I$ on extremal vectors in terms of 
the natural action of $W$ on~$W/W_J$.
\begin{proposition}\label{prop:ficus action}
Let $V\in\mathscr O^{int}_q(\lie g)$ and $v\in V_+\setminus\{0\}$. Then 
\begin{enumerate}[label={\rm(\alph*)},leftmargin=*]
 \item\label{prop:ficus action.a} The set $[v]_W$ is $\mathsf W(V)$-invariant. More precisely, $\sigma^i([v(\lambda)]_w)=[v(\lambda)]_{s_i w}$ for 
 all $i\in I$, $w\in W$ and $\lambda\in\supp v$;
 \item\label{prop:ficus action.b} The canonical image of $\mathsf W(V)$ in $\operatorname{Bij}([v]_W)$ is isomorphic to $W_{J_0}$ where 
 $J_0=I\setminus\cl(J(U_q(\lie g)(v)))$.
 \end{enumerate}
\end{proposition}
\begin{proof}
To prove part~\ref{prop:ficus action.a}, let $\lambda\in P^+$, $w\in W$, $i\in I$ and 
suppose first that $\ell(s_i w)=\ell(w)+1$. Then $w\lambda(\alpha_i^\vee)>0$ and 
$[v(\lambda)]_w\in\ker E_i$, whence $\sigma^i([v(\lambda)]_w)=F_i^{(w\lambda(\alpha_i^\vee))}([v(\lambda)]_w)=F_{s_iw,\lambda}(v(\lambda))=[v(\lambda)]_{s_iw}$
where we used Lemma~\ref{lem:EF lambda w}\ref{lem:EF lambda w.b}. 
If $\ell(s_i w)=\ell(w)-1$ then $w=s_i w'$. By the above, $[v(\lambda)]_w=[v(\lambda)]_{s_i w'}=\sigma^i([v(\lambda)]_{w'})=\sigma^i([v(\lambda)]_{s_iw})$.
Since $\sigma^i$ is an involution, it follows that $\sigma^i([v(\lambda)]_{w})=[v(\lambda)]_{s_iw}$.

To prove~\ref{prop:ficus action.b}, the bijections from Proposition~\ref{prop:w-translates of hw}\ref{prop:w-translates of hw.a} allow one to identify
$[v]_W$ with $\bigsqcup\limits_{\lambda\in\supp v} W/W_{J_\lambda}$. In paricular, this induces an action of~$W$ on~$[v]_W$ via
$w\cdot [v(\lambda)]_{w'}=[v(\lambda)]_{ww'}$, $\lambda\in\supp v$, $w,w'\in W$. By part~\ref{prop:ficus action.a} the canonical images of 
$\mathsf W(V)$ and $W$ in $\operatorname{Bij}([v]_W)$ coincide. We need the following general fact.
\begin{lemma}
Let $G$ be a group acting on $X=\bigsqcup_{\alpha\in A} X_\alpha$. 
Then the canonical image of $G$ in $\operatorname{Bij}(X)$
is isomorphic to $G/K$ where $K=\bigcap_{\alpha\in A} K_\alpha$ and $K_\alpha=\{ g\in G\,:\, g x=x,\,\forall x\in X_\alpha\}$ is the kernel of the 
action of~$G$ on $X_\alpha$.
\end{lemma}
Applying this Lemma to $X=[v]_W$, $X_\lambda=[v(\lambda)]_W$ and $G=W$ and using the fact that $K_\lambda=W_{J_\lambda}$ by 
Theorem~\ref{thm:kernel of natural action} we conclude that $K=\bigcap_{\lambda\in\supp v}W_{J_\lambda}$. By Lemma~\ref{lem:intersect parabolics}, 
$K=W_{J'}$ where $J'$ is the set of~$i\in I$ such that $s_i$ fixes~$[v]_W$ elementwise. By Lemma~\ref{lem:set I_V}, $J'=J(V')$
where $V'=U_q(\lie g)(v)$. Since $J'$ is closed, being an intersection of closed sets, $W/W_{J'}\cong W_{I\setminus J'}$.
\end{proof}

\begin{proof}[Proof of Theorem~\ref{thm:crystal weyl group}]

It follows from Proposition~\ref{prop:w-translates of hw}\ref{prop:w-translates of hw.a} that
the assignments $[v]_w\mapsto w W_J$ define a bijection $\mathsf J:[v]_W\to W/W_J$.
This induces a group homomorphism $\xi_V:\mathsf W(V)\to \operatorname{Bij}(W/W_J)$ 
via $\xi_V(\sigma^i)(w W_J)=\mathsf J(\sigma^i([v]_w))=\mathsf J([v]_{s_i w})=s_i w W_J$.
It follows that $\xi_V(\mathsf W(V))$ coincides with the image of $W$ in $\operatorname{Bij}(W/W_J)$ given 
by the natural action. By Lemma~\ref{lem:factorization of W}, the latter is canonically isomorphic to $W_{I\setminus J_0}$ where $J_0=\cl(I\setminus J)$.
\end{proof}

\section{Modified Lusztig symmetries and involutions \texorpdfstring{$\sigma^J$}{eta J}}\label{sect:mod Lusz symm}

Let $\mathscr C$ be a $\kk$-linear category whose objects are~$\kk$-vector spaces and let~$G$ be a group. An action of $G$ on~$\mathscr C$
is an assignment $g\mapsto g_\bullet=\{ g_V\,:\, V\in\mathscr C\}$, $g\in G$, 
where $g_V\in\operatorname{GL}_\kk(V)$ such that $(gg')_V=g_V\circ g'_V$ for all $g,g'\in G$ and $V\in\mathscr C$
and $g_{V'}\circ f=f\circ g_V$ for any $g\in G$ and any morphism $f:V\to V'$ in~$\mathscr C$.

Recall that the braid group~$\Br_W$ associated with a Coxeter group~$W$ is generated by the $T_i$, $i\in I$ subject to the relations $\underbrace{T_i T_j\cdots}_{m_{ij}}=
\underbrace{T_j T_i\cdots }_{m_{ij}}$ for all $i\not=j\in I$ in the notation of~\S\ref{subs:coxeter groups}.
In this section we discuss modified Lusztig symmetries
which provide an action of~$\Br_W$ associated with the Weyl group~$W$ of~$\lie g$
on the category~$\mathscr O^{int}_q(\lie g)$ and use them to construct an action of 
$\Cact_W$ on the same category.

\subsection{Modified Lusztig's symmetries}\label{subs:mod lus sym}
Given $i\in I$ and $V\in\mathscr O^{int}_q(\lie g)$ define $T_i^\pm\in\End_\kk V$ by
$$
T_i^+= T'_{i,1}K_{\frac12\alpha_i},\qquad T_i^-= T''_{i,1}K_{-\frac12\alpha_i}
$$
where 
$T'_{i,1}, T''_{i,1}\in\End_\kk V$ are Lusztig symmetries (see~\cite{Lus}*{\S5.2}). We refer to these operators as {\em modified Lusztig symmetries}. By definition,
$T_i^\pm(V(\beta))=V(s_i\beta)$ and so $T_i^\pm K_\lambda=K_{s_i\lambda} T_i^\pm$, $\lambda\in\frac12 P$.

\begin{lemma}\label{cor:T i +-}
The assignments $T_i\mapsto T_i^+$ (respectively, $T_i\mapsto T_i^-$) define 
an action of~$\Br_W$ on~$\mathscr O^{int}_q(\lie g)$.
\end{lemma}
\begin{proof}
It can be deduced from~\cite{Lus}*{Proposition~39.4.3} along the lines of~\cite{BG-dcb}*{Lemma~5.2} that 
the $T_i^+$ (and the $T_i^-$), $i\in I$, satisfy the defining relations of~$\Br_W$ as endomorphisms of $V$ for each 
$V\in\mathscr O^{int}_q(\lie g)$. 
To prove that this action of~$\Br_W$ commutes with morphisms, write, using~\cite{LusProb}*{\S3.1}, for $V\in\mathscr O^{int}_q(\lie g)$ and~$v\in V$
\begin{equation}\label{eq:T_i+- expl}
\begin{aligned}
&T_i^+(v)=\sum_{(a,b,c)\in\ZZ_{\ge 0}^3} (-1)^b q_i^{b-a c} K_{\frac12(a-c-1)\alpha_i}F_i^{(a)}E_i^{(b)}F_i^{(c)}K_{\frac12(a-c)\alpha_i}(v),\\
&T_i^-(v)=\sum_{(a,b,c)\in\ZZ_{\ge 0}^3} (-1)^b q_i^{b-a c} K_{\frac12(c-a+1)\alpha_i}E_i^{(a)}F_i^{(b)}E_i^{(c)}K_{\frac12(c-a)\alpha_i}(v)
\end{aligned}
\end{equation}
where the sum is finite since $V$ is integrable. It is now obvious that the~$T_i^\pm$, $i\in I$ commute with homomorphisms of~$U_q(\lie g)$-modules.
\end{proof}

 It is well-known (see e.g.~\cite{Lus}*{\S39.4.7}) 
that the element $T_{i_1}\cdots T_{i_r}$
with $\mathbf i=(i_1,\dots,i_r)\in I^r$ reduced depends only on~$w=s_{i_1}\cdots s_{i_r}$ and not on~$\mathbf i$. This allows to define the canonical section of the 
natural 
group homomorphism $\Br_W\to W$, $T_i\mapsto s_i$, $i\in I$,
by $w\mapsto T_w:=T_{i_1}\cdots T_{i_r}$ where $(i_1,\dots,i_r)\in R(w)$.
Denote $T^\pm_w$ the
linear endomorphisms of any $V\in \mathscr O^{int}_q(\lie g)$ arising from Lemma~\ref{cor:T i +-} which correspond to the canonical element~$T_w$ of~$\Br_W$.
The elements $T_w$, $T^\pm_w$ are characterized by the following well-known property.
\begin{lemma}[\cite{Lus}*{\S39.4.7}]\label{lem:product can elts br}
Let $w,w'\in W$ be such that $\ell(w)+\ell(w')=\ell(ww')$. Then $T_{ww'}=T_w T_{w'}$ and 
$T^\pm_{ww'}=T^\pm_w\circ T^\pm_{w'}$ as linear endomorphisms of~$V\in\mathscr O^{int}_q(\lie g)$.
\end{lemma}

\begin{proposition}\label{prop:T_w extreme vectors}
Let $V\in\mathscr O^{int}_q(\lie g)$. Then 
for any $w,w'\in W$, $\lambda\in P^+$ and $v\in V_+(\lambda)$ we have:
\begin{enumerate}[label={\rm(\alph*)},leftmargin=*]
 \item\label{prop:T_w extreme vectors.a} if $\ell(ww')=\ell(w)+\ell(w')$ then 
\begin{align*}
 T_w^{+}(F_{w',\lambda}(v))&=q^{\frac12(w'\lambda,\rho-w^{-1}\rho)}F_{ww',\lambda}(v),
 \\
 T_w^{-}(F_{w',\lambda}(v))&=(-1)^{\rho^\vee(w'\lambda-ww'\lambda)}q^{\frac12(w'\lambda,\rho-w^{-1}\rho)}F_{ww',\lambda}(v).
\end{align*}
\item\label{prop:T_w extreme vectors.b}
 if $\ell(ww')=\ell(w')-\ell(w)$ then 
\begin{align*}
 T_w^{+}(F_{w',\lambda}(v))&=(-1)^{\rho^\vee(w'\lambda-ww'\lambda)}q^{-\frac12(w'\lambda,\rho-w^{-1}\rho)}F_{ww',\lambda}(v),
 \\
 T_w^{-}(F_{w',\lambda}(v))&=q^{-\frac12(w'\lambda,\rho-w^{-1}\rho)}F_{ww',\lambda}(v).
\end{align*}
\item\label{prop:T_w extreme vectors.c}
if $\ell(ww')=\ell(w)-\ell(w')$ then 
\begin{align*}
T_w^+(F_{w',\lambda}(v))&=(-1)^{\rho^\vee(w'\lambda-\lambda)}q^{\frac12(\lambda,2\rho-w'^{-1}(\rho+w^{-1}\rho))} F_{ww',\lambda}(v)\\
T_w^-(F_{w',\lambda}(v))&=(-1)^{\rho^\vee(\lambda-ww'\lambda)}q^{\frac12(\lambda,2\rho-w'^{-1}(\rho+w^{-1}\rho))} F_{ww',\lambda}(v).
\end{align*}
\end{enumerate}
\end{proposition}
\begin{proof}
To prove~\ref{prop:T_w extreme vectors.a} we argue by induction on~$\ell(w)$, the case~$\ell(w)=0$ being vacuously true. 
The following Lemma is the main ingredient in the proof of inductive steps in Proposition~\ref{prop:T_w extreme vectors}.
\begin{lemma}
In the notation of Proposition~\ref{prop:T_w extreme vectors} we have, for all $i\in I$
\begin{equation}\label{eq:T_i+- extrem}
\begin{aligned}
&T^+_i(F_{w',\lambda}(v))=\begin{cases}q^{\frac12 (w'\lambda,\alpha_i)} F_{s_iw',\lambda}(v),& (w'\lambda,\alpha_i)\ge 0,\\
                       (-1)^{w'\lambda(\alpha_i^\vee)} q^{-\frac12 (w'\lambda,\alpha_i)} F_{s_iw',\lambda}(v),&(w'\lambda,\alpha_i)\le 0,
                      \end{cases}
\\        
& 
T^-_i(F_{w',\lambda}(v))=\begin{cases}(-1)^{w'\lambda(\alpha_i^\vee)} q^{\frac12 (w'\lambda,\alpha_i)} F_{s_iw',\lambda}(v),\phantom{\scriptscriptstyle-}& (w'\lambda,\alpha_i)\ge 0,\\
q^{-\frac12 (w'\lambda,\alpha_i)} F_{s_i w',\lambda}(v),& (w'\lambda,\alpha_i)\le 0.
\end{cases}
\end{aligned}
\end{equation}
\end{lemma}
\begin{proof}
Clearly, $F_{w',\lambda}(v)$ is either a highest or a lowest weight vector in the $i$th simple quantum $\lie{sl}_2$-submodule~$V_m$
it generates where $m=|(w'\lambda,\alpha_i^\vee)|$. 
Then~\eqref{eq:T_i+- extrem} follows from~\cite{Lus}*{Propositions~5.2.2, 5.2.3}.
Namely, let $V_m$ be the standard simple $U_\mathbf v(\lie{sl}_2)$-module with the standard basis $\{z_k\}_{0\le k\le m}$ such that $K(z_k)=\mathbf v^{m-2k} z_k$
and $z_k=F^{(k)}(z_0)=E^{(m-k)}(z_m)$, $0\le k\le m$. Recall that $T^+=T'_1 K^{\frac12}$ and $T^-=T''_1 K^{-\frac12}$. Then 
by~\cite{Lus}*{Propositions~5.2.2, 5.2.3} we have 
\begin{equation}\label{eq:T sl2 std bas}
T^+(z_k)=(-1)^k \mathbf v^{k(m-k)+\frac12m} z_{m-k},\,\, T^-(z_k)=(-1)^{m-k} \mathbf v^{k(m-k)+\frac12 m}z_{m-k},\,\, 0\le k\le m.
\end{equation}
Thus, \eqref{eq:T_i+- extrem} is obtained by applying~\eqref{eq:T sl2 std bas} with $k=0$ if $w'\lambda(\alpha_i^\vee)\ge 0$
and $k=m$ if $w'\lambda(\alpha_i^\vee)\le 0$.
\end{proof}

To prove inductive steps in part~\ref{prop:T_w extreme vectors.a} of Proposition~\ref{prop:T_w extreme vectors}, suppose that $w,w'\in W$ and $i\in I$ satisfy $\ell(s_iww')=\ell(w)+\ell(w')+1$. 
In particular, $\ell(s_i w)=\ell(w)+1$ and $\ell(ww')=\ell(w)+\ell(w')$.
Then we have, by the induction hypothesis and \eqref{eq:T_i+- extrem}
\begin{multline*}
T_{s_iw}^+(F_{w',\lambda}(v))=T_i^+ T_{w}^+(F_{w',\lambda}(v))=q^{\frac12(w'\lambda,\rho-w^{-1}\rho)} T_i^+(F_{ww',\lambda}(v))\\
=q^{\frac12 (w'\lambda,\rho-w^{-1}\rho)+\frac12 (ww'\lambda,\alpha_i)} F_{s_iww',\lambda}(v)
=q^{\frac12 (w'\lambda,\rho-w^{-1}\rho+w^{-1}\alpha_i)} F_{s_iww',\lambda}(v)\\
=q^{\frac12 (w'\lambda,\rho-(s_iw)^{-1}\rho)} F_{s_iww',\lambda}(v),
\end{multline*}
where we used the $W$-invariance of~$(\cdot,\cdot)$ and the obvious observation that~$\alpha_i=\rho-s_i\rho$. 
The second identity in part~\ref{prop:T_w extreme vectors.a} is proved similarly using the observation that 
$\mu(\alpha_i^\vee)=\rho^\vee(\mu-s_i \mu)$.

The proof of part~\ref{prop:T_w extreme vectors.b} is identical, the only difference being that we assume $\ell(s_i ww')=
\ell(w')-\ell(w)-1$ which implies that $\ell(s_iw)=\ell(w)+1$ and $\ell(ww')=\ell(w)-\ell(w')$.

To prove part~\ref{prop:T_w extreme vectors.c}, denote $w_1=ww'$. Then $\ell(w)=\ell(w_1)+\ell(w'{}^{-1})$, $w=w_1w'{}^{-1}$ 
and so $T^+_w=T^+_{w_1}\circ T^+_{w'{}^{-1}}$ by Lemma~\ref{lem:product can elts br}.
Applying part~\ref{prop:T_w extreme vectors.b} with 
$w=w'{}^{-1}$ and then part~\ref{prop:T_w extreme vectors.a} with $w'=1$ and $w=w_1$ we obtain 
\begin{multline*}
 T_w^+(F_{w',\lambda}(v))=T_{w_1}^+(T^+_{w'{}^{-1}}(F_{w',\lambda}(v)))\\
=(-1)^{\rho^\vee(w'\lambda-\lambda)}q^{-\frac12(w'\lambda,\rho-w'\rho)}T_{w_1}^+(v)
=(-1)^{\rho^\vee(w'\lambda-\lambda)}q^{-\frac12(\lambda,w'{}^{-1}\rho+w_1^{-1}\rho)} F_{w_1,\lambda}(v)
\\
=(-1)^{\rho^\vee(w'\lambda-\lambda)}q^{(\lambda,\rho)-\frac12(\lambda,w'{}^{-1}\rho)+\frac12(\lambda,\rho-w'^{-1}w^{-1}\rho)} F_{w_1,\lambda}(v)
\\
=(-1)^{\rho^\vee(w'\lambda-\lambda)}q^{\frac12(\lambda,2\rho-w'^{-1}(\rho+w^{-1}\rho))} F_{ww',\lambda}(v).
\end{multline*}
The identity for~$T_w^-$ is proved similarly.
\end{proof}

Recall from~\cite{Lus}*{Chapter 37} that $\Br_W$ also acts on~$U_q(\lie g)$ via Lusztig symmetries and let $T_i^\pm $ be the automorphisms 
of~$U_q(\lie g)$ defined as $T_i^+=T'_{i,1}\ad K_{\frac12\alpha_i}$ and $T_i^-=T''_{i,1}\ad K_{-\frac12 \alpha_i}$, $i\in I$ where $\ad K_{\lambda}(u)=
K_\lambda u K_{-\lambda}$, $\lambda\in\frac12 P$, $u\in U_q(\lie g)$.
\begin{remark}
The operators $T_i^\pm$, viewed as automorphisms of~$U_q(\lie g)$, were already used in~\cite{BG-dcb} for studying double canonical bases of~$U_q(\lie g)$.
\end{remark}

\begin{lemma}\label{lem: T^+ T Lus}
On~$U_q(\lie g)$ we have 
\begin{enumerate}[label={\rm(\alph*)},leftmargin=*]
 \item\label{lem: T^+ T Lus.a} $T^\pm_i\circ \ad K_\lambda=\ad K_{s_i\lambda}\circ T^\pm_i$ for all $\lambda\in \frac12 P$, $i\in I$;
 \item\label{lem: T^+ T Lus.b} $T^+_{w}=T'_{w,1}\circ \ad K_{\frac12(\rho-w^{-1}\rho)}$ and $T^-_{w}=T''_{w,1}\circ\ad K_{\frac12(w^{-1}\rho-\rho)}$ for all~$w\in W$.
\end{enumerate}
\end{lemma}
\begin{proof}
Part~\ref{lem: T^+ T Lus.a} is immediate, while part~\ref{lem: T^+ T Lus.b} follows from~\ref{lem: T^+ T Lus.a} by induction similar to that 
in Proposition~\ref{prop:T_w extreme vectors}.
\end{proof}

\begin{lemma}\label{lem:T w0 F_i}
Suppose that $W$ is finite. Then for all~$i\in I$ we have
\begin{equation*}
\begin{aligned}
T^\pm_{w_\circ}(E_i)=-q^{-\frac12(\alpha_i,\alpha_i)}F_{i^\star }K_{\alpha_{i^\star }},\quad 
T^\pm_{w_\circ}(F_i)=-q^{-\frac12(\alpha_i,\alpha_i)}E_{i^\star }K_{-\alpha_{i^\star }}.
\end{aligned}
\end{equation*}
\end{lemma}
\begin{proof}
By~\cite{BG}*{Lemma~2.8} and Lemma~\ref{lem: T^+ T Lus}\ref{lem: T^+ T Lus.b} we have $T^\pm_w(E_i)=q^{\pm\frac12 (\rho-w^{-1}\rho,\alpha_i)} E_j$
provided that $w\alpha_i=\alpha_j$. Since $(w^{-1}\rho,\alpha_i)=(\rho,w\alpha_i)=(\rho,\alpha_j)$ and $(\alpha_i,\alpha_i)=(\alpha_j,\alpha_j)$
it follows that $T^\pm_w(E_i)=E_j$. In particular, since $w_\circ s_i\alpha_i=\alpha_{i^\star }$ we have $T^\pm_{w_\circ s_i}(E_i)=E_{i^\star }$.

On the other hand, $T^\pm_{w_\circ s_i}(E_i)=T^\pm_{s_{i^\star }w_\circ}(E_i)$ and $\ell(w_\circ)=\ell(s_{i^\star }w_\circ)+1$ whence 
$T^\pm_{w_\circ}(E_i)=T^\pm_{i^\star }(E_{i^\star })=-q^{-\frac12(\alpha_i,\alpha_i)}F_{i^\star }K_{\alpha_{i^\star }}$. The argument for~$T^\pm_{w_\circ}(F_i)$ is similar.
\end{proof}

The following Lemma is immediate from~\cite{Lus}*{Proposition~37.1.2}. 
\begin{lemma}\label{lem:compat symm braid}
Let $V\in\mathscr O_q^{int}(\lie g)$. Then 
$T_i^\pm(x(v))=T_i^\pm(x)(T_i^\pm(v))$ for all $i\in I$, $v\in V$, $x\in U_q(\lie g)$.
\end{lemma}

 \begin{lemma}\label{lem:T w weights isotyp}
  We have $T_w^\pm(\mathcal I_\lambda(V))=\mathcal I_\lambda(V)$ and $T_w^\pm(V(\beta))=V(w\beta)$
 for any $w\in W$, $V\in\mathscr O^{int}_q(V)$, $\lambda\in P^+$ and $\beta\in P$.
 \end{lemma}

\subsection{\texorpdfstring{$\sigma$}{eta} via modified Lusztig's symmetries} 
Let $\lie g$ be finite dimensional reductive and 
define
\begin{equation}\label{eq:defn-eta}
\sigma^\pm(v)=(-1)^{\rho^\vee(\lambda\mp \beta)} q^{\frac12((\beta,\beta)-(\lambda,\lambda))-(\lambda,\rho)}
T_{w_\circ}^\pm(v),\quad v\in V(\beta)\cap\mathcal I_\lambda(V).
\end{equation}
The main ingredient in our proof of Theorem~\ref{thm:main thm} is the following result
which generalizes~\cite{LusII}*{Proposition~5.5} to all reductive algebras including those whose 
semisimple part is not necessarily simply laced.
\begin{theorem}\label{thm:prop-eta}
Let~$V\in\mathscr O^{int}_q(\lie g)$. Then
\begin{enumerate}[label={\rm(\alph*)},leftmargin=*]
\item\label{thm:prop-eta.a}
$\sigma^+=\sigma^-$ and is an involution which thus will be denoted by~$\sigma$;

\item\label{thm:prop-eta.b}
$\sigma(x(v))=\theta(x)(\sigma(v))$ for any $x\in U_q(\lie g)$, $v\in V$;

\item\label{thm:prop-eta.a'}
$\sigma$ commutes with morphisms in~$\mathscr O^{int}_q(\lie g)$.
\end{enumerate}

\end{theorem}

\begin{remark}\label{rem:eta inv}
For $\lie g=\lie{sl}_n$ the involution~$\sigma$ coincides with the famous Sch\"utzenberger involution on Young tabeleaux
which was established for the first time in~\cite{BZ1}. Thus, we can regard~$\sigma$ as the generalized Sch\"utzenberger involution.
\end{remark}

\begin{proof}
We need the following properties of~$\sigma^\pm$.

\begin{lemma}\label{lem:exist eta}
For any $V\in\mathscr O^{fin}_q(\lie g)$ we have 
\begin{enumerate}[label={\rm(\roman*)},align=center,leftmargin=*]
\item\label{lem:exist eta.a} $\sigma^\pm(F_{w,\lambda}(v_\lambda))=F_{w_\circ w,\lambda}(v_\lambda)$ for any $v_\lambda\in \mathcal I_\lambda(V)(\lambda)$, $\lambda\in P^+$;
\item \label{lem:exist eta.b} $\sigma^\pm(x(v))=\theta(x)(\sigma^\pm(v))$ for any $x\in U_q(\lie g)$, $v\in V$;
\item \label{lem:exist eta.c}
$\sigma^+=\sigma^-$ and 
$(\sigma^\pm)^2=\id_V$.
\end{enumerate}
\end{lemma}
\begin{proof}
Let $v_\lambda\in \mathcal I_\lambda(V)(\lambda)$, $\lambda\in P^+$.
Since $\rho+w_\circ\rho=0$, $\rho^\vee(w_\circ\mu)=-\rho^\vee(\mu)$ for any~$\mu\in P$ and $\ell(w_\circ u)=\ell(w_\circ)-\ell(u)$ for any~$u\in W$,
Proposition~\ref{prop:T_w extreme vectors}\ref{prop:T_w extreme vectors.c} with $w=w_\circ$ and $w'=u$ yields
$$
T_{w_\circ}^\pm(F_{u,\lambda}(v_\lambda))=(-1)^{\rho^\vee(\pm u\lambda-\lambda)}q^{(\lambda,\rho)} F_{w_\circ u,\lambda}(v_\lambda),\\
$$
Since $F_{u,\lambda}(v_\lambda)\in V(u\lambda)$, \eqref{eq:defn-eta} yields
$$
\sigma^\pm(F_{u,\lambda}(v_\lambda))=(-1)^{\rho^\vee(\lambda\mp u\lambda)} q^{-(\lambda,\rho)}T_{w_\circ}^\pm(F_{u,\lambda}(v_\lambda))
=F_{w_\circ u,\lambda}(v_\lambda).
$$
This proves part~\ref{lem:exist eta.a}. To prove part~\ref{lem:exist eta.b}, let $v\in V(\beta)$. Then $E_i(v)\in V(\beta+\alpha_i)$,\
and we obtain, by~\eqref{eq:defn-eta} and
Lemmata~\ref{lem: T^+ T Lus}, \ref{lem:T w0 F_i}
\begin{align*}
\sigma^\pm(E_i(v))&=(-1)^{\rho^\vee(\lambda\mp(\beta+\alpha_i))}q^{\frac12((\beta+\alpha_i,\beta+\alpha_i)-(\lambda,\lambda))-(\lambda,\rho)}
T_{w_\circ}^\pm(E_i(v))
\\
&=-(-1)^{\rho^\vee(\lambda\mp \beta)}q^{(\beta,\alpha_i)+\frac12((\alpha_i,\alpha_i)+(\beta,\beta)-(\lambda,\lambda))-(\lambda,\rho)}
T_{w_\circ}^\pm(E_i)(T_{w_\circ}^\pm(v))
\\
&=(-1)^{\rho^\vee(\lambda\mp \beta)}q^{(\beta,\alpha_i)+\frac12((\beta,\beta)-(\lambda,\lambda))-(\lambda,\rho)}
F_{i^\star}K_{i^\star}(T_{w_\circ}^\pm(v))
\\
&=\theta(E_i)(q^{(\beta,\alpha_i+w_\circ\alpha_{i^\star})}\sigma^\pm(v))
=\theta(E_i)\sigma^\pm(v).
\end{align*}
The identity $\sigma^\pm(F_i(v))=\theta(F_i)(\sigma^\pm(v))$ is proved similarly. Finally, for any $\mu\in\frac12 P$ we 
have $\sigma^\pm(K_\mu(v))=q^{(\beta,\mu)}\sigma^\pm(v)=q^{(w_\circ\beta,w_\circ\mu)}\sigma^\pm(v)=\theta(K_\mu)(\sigma^\pm(v))$.

Let $\epsilon,\epsilon'\in\{+,-\}$.
It follows from part~\ref{lem:exist eta.b} that 
$$
\sigma^\epsilon\circ\sigma^{\epsilon'}(x(v))=\sigma^\epsilon(\theta(x)(v))=\theta^2(x)(v)=x(v)
$$
for any $x\in U_q(\lie g)$ and~$v\in V$ since~$\theta$ is an involution. Thus, 
$\sigma^\epsilon\circ\sigma^{\epsilon'}$ is an endomorphism of~$V$ as a $U_q(\lie g)$-module. 
By part~\ref{lem:exist eta.a} we have 
$$
\sigma^\epsilon\circ\sigma^{\epsilon'}(F_{w,\lambda}(v_\lambda))=\sigma^\epsilon(F_{w_\circ w,\lambda}(v_\lambda))
=F_{w,\lambda}(v_\lambda)
$$
for any $w\in W$, $\lambda\in P^+$ and $v_\lambda\in \mathcal I_\lambda(V)(\lambda)$. In particular, $\sigma^\epsilon\circ\sigma^{\epsilon'}(v_\lambda)=v_\lambda$
and so~$\sigma^\epsilon\circ\sigma^{\epsilon'}$ is the identity map on the (simple) $U_q(\lie g)$-submodule of~$V$ generated by~$v_\lambda$.
Since $V$ is generated by $\bigoplus_{\lambda\in P^+}\mathcal I_\lambda(V)(\lambda)$ as a 
$U_q(\lie g)$-module, $\sigma^{\epsilon}\circ\sigma^{\epsilon'}=\id_V$. This proves 
part~\ref{lem:exist eta.c}.
\end{proof}

Parts~\ref{thm:prop-eta.a} and~\ref{thm:prop-eta.b} of Theorem~\ref{thm:prop-eta} were established in Lemma~\ref{lem:exist eta}.
To prove part~\ref{thm:prop-eta.b}, note the following obvious fact.
\begin{lemma}\label{lem:map commutes}
Let $\xi_\bullet=\{ \xi_V\in \End_{U_q(\lie g)}V\,:\, V\in\mathscr O_q^{int}(\lie g)\}$ and suppose that $\xi_\bullet$ commute with morphisms 
in~$\mathscr O^{int}_q(\lie g)$, that is
$\xi_{V'}\circ f=f\circ\xi_V$
for any morphism $f:V\to V'$ in~$\mathscr O^{int}_q(\lie g)$.
Let $\chi:P^+\times P\to \kk$ and define $\xi^\chi_\bullet$ by 
by $\xi^\chi_V(v)=\sum_{\beta\in\supp v}\chi(\lambda,\beta) \xi_V(v(\beta))$, $v\in \mathcal I_\lambda(V)$, $V\in\mathscr O^{int}_q(\lie g)$. 
Then $\xi^\chi_\bullet$ also commutes with morphisms 
in~$\mathscr O_q^{int}(\lie g)$.
\end{lemma}
It is immediate from the definition~\eqref{eq:defn-eta}  of~$\sigma$ that $\sigma_\bullet=(T^\pm_\bullet)^{\chi_\pm}$
where $\chi_\pm(\lambda,\beta)=(-1)^{\rho^\vee(\lambda\mp \beta)} q^{\frac12((\beta,\beta)-(\lambda,\lambda))-(\lambda,\rho)}$, $(\lambda,\beta)\in P^+\times P$.
Then part~\ref{thm:prop-eta.a'} follows from Lemma~\ref{lem:map commutes}.
\end{proof}

\subsection{Parabolic involutions and proof of Theorem~\ref{thm:main thm}}\label{subs:parab inv}
In view of Theorem~\ref{thm:prop-eta}, given $V\in\mathscr O^{int}_q(\lie g)$ and $J\in\mathscr J$, let $\sigma^J:V\to V$ 
be $\sigma$ defined by~\eqref{eq:defn-eta} with $U_q(\lie g)$ replaced by~$U_q(\lie g^J)$. Thus,
\begin{equation}\label{eq:defn-eta^J}
\sigma^J(v)=(-1)^{\rho^\vee_J(\lambda_J\mp\beta)} q^{-\frac12((\lambda_J,\lambda_J)-(\beta,\beta))-(\lambda_J,\rho_J)}T_{w_\circ^J}^\pm(v),
\end{equation}
for any~$\lambda_J\in P^+_J$, $\beta\in P$ and  
$v\in \mathcal I_{\lambda_J}^J(V)(\beta)$. 
Note that $\mathcal I_{\lambda_J}^J(V)(\beta)=0$ unless $\lambda_J-\beta\in\sum_{j\in J}\ZZ_{\ge 0}\alpha_j$.
We need the following properties of~$\sigma^J$.
\begin{proposition}\label{prop:prop eta^J}
Let $V\in\mathscr O^{int}_q(\lie g)$. Then 
\begin{enumerate}[label={\rm(\alph*)},leftmargin=*]
\item\label{prop:prop eta^J.a'}
$\sigma^J(F_{w,\lambda_J}(v))=F_{w_\circ^J w,\lambda}(v)$ for any $v\in \mathcal I_{\lambda_J}^J(V)(\lambda_J)$, $\lambda_J\in P^+$.
In particular, $\sigma^J(v)=v^J:=F_{w_\circ^J,\lambda}(v)$.
\item\label{prop:prop eta^J.a} $\sigma^J$ is an involution;
\item\label{prop:prop eta^J.b} $\sigma^J$ commutes with morphisms in~$\mathscr O^{int}_q(\lie g^J)$ and satisfies 
$\sigma^J(x(v))=\theta_J(x)(\sigma^J(v))$, $x\in U_q(\lie g^J)$, $v\in V$;
\item\label{prop:prop eta^J.b'} $\sigma^J(V(\beta))=V(w_\circ^J\beta)$, $\beta\in P$;

\item\label{prop:prop eta^J.c} for any $\lambda_J\in P^+_J$, $\beta\in P$ and $v\in\mathcal I_{\lambda_J}^J(V)(\beta)$ we have 
$$
\sigma^J(v)
=(-1)^{\rho_J^\vee(-\lambda_J\mp\beta)} q^{\frac12((\lambda_J,\lambda_J)-(\beta,\beta))+(\lambda_J,\rho_J)}(T_{w_\circ^J}^\pm)^{-1}(v).
$$
\end{enumerate}
\end{proposition}
\begin{proof}
Replacing $\lie g$ by~$\lie g^J$ we obtain 
part~\ref{prop:prop eta^J.a'}
from Lemma~\ref{lem:exist eta}\ref{lem:exist eta.a} and 
parts~\ref{prop:prop eta^J.a}, \ref{prop:prop eta^J.b} from Theorem~\ref{thm:prop-eta}. 
Part~\ref{prop:prop eta^J.b'}
is immediate from~\eqref{eq:defn-eta^J} and Lemma~\ref{lem:T w weights isotyp}.
To prove part~\ref{prop:prop eta^J.c},
let $v'=\sigma^J(v)$.  Then $v=\sigma^J(v')$ and
$v'\in \mathcal I_{\lambda_J}^J(V)(\beta')$ where~$\beta'=w_\circ^J\beta$.
Applying $(T^\pm_{w_\circ})^{-1}$ to~\eqref{eq:defn-eta^J} with~$v$ replaced by~$v'$ we obtain 
$$
(T^\pm_{w_\circ})^{-1}(\sigma^J(v'))=(-1)^{\rho^\vee_J(\lambda_J\mp\beta')} q^{-\frac12((\lambda_J,\lambda_J)-(\beta',\beta'))-(\lambda_J,\rho_J)}v'.
$$
Since $\rho^\vee_J(\beta')=\rho^\vee_J(w_\circ^J\beta)=-\rho^\vee_J(\beta)$ and $(\cdot,\cdot)$ is $W$-invariant, it follows that 
$$
(T^\pm_{w_\circ})^{-1}(v)=(-1)^{\rho^\vee_J(\lambda_J\pm\beta)} q^{-\frac12((\lambda_J,\lambda_J)-(\beta,\beta))-(\lambda_J,\rho_J)}\sigma^J(v).
$$
The assertion is now immediate.
\end{proof}

We need the following results.
\begin{proposition}\label{prop:factor eta_J}
For any $J\subset J'\in\mathscr J$, $\sigma^{J'}\circ \sigma^J=\sigma^{J^\star}\circ \sigma^{J'}$ where ${}^\star:J'\to J'$ is the 
unique involution satisfying $\alpha_{j^\star}=-w_\circ^{J'}\alpha_j$, $j\in J'$. 
\end{proposition}
\begin{proof}
We may assume, without loss of generality, that $J'=I$ (and so~$\lie g$ is reductive finite dimensional).
Let $w_J=w_\circ w_\circ^J$. 
Note that $w_\circ=w_J w_\circ^J=w_\circ^{J^\star} w_J$ and $\ell(w_\circ)=\ell(w_J)+\ell(w_\circ^J)=\ell(w_J)+\ell(w_\circ^{J^\star})$.
Then by Lemma~\ref{lem:product can elts br}
\begin{equation}\label{eq:factor T_wJ}
T^\pm_{w_J}=T^\pm_{w_\circ}\circ (T^\pm_{w_\circ^J})^{-1}
=(T^\pm_{w_\circ^{J^\star}})^{-1}\circ T^\pm_{w_\circ}.
\end{equation}

Let $v\in V(\beta)\cap \mathcal I_\lambda(V)\cap \mathcal I^J_{\lambda_J}(V)$, $\lambda\in P^+$, $\lambda_J\in P^+_J$, $\beta\in P$.
Using Lemma~\ref{prop:prop eta^J}\ref{prop:prop eta^J.c}, \eqref{eq:defn-eta}, Lemma~\ref{lem:T w weights isotyp}
and~\eqref{eq:factor T_wJ} 
we obtain 
\begin{multline*}
\sigma\circ \sigma^J(v)=(-1)^{\rho^\vee_J(-\lambda_J\mp\beta)} q^{\frac12((\lambda_J,\lambda_J)-(\beta,\beta))+(\lambda_J,\rho_J)}\sigma( (T^\pm_{w_\circ^J})^{-1}(v))
\\
=(-1)^{\rho^\vee(\lambda\mp w_\circ^J\beta)+\rho^\vee_J(-\lambda_J\mp\beta)} q^{\frac12((\lambda_J,\lambda_J)-(\lambda,\lambda))+(\lambda_J,\rho_J)-(\lambda,\rho)}T_{w_\circ}^\pm
(T^\pm_{w_\circ^J})^{-1}(v)\\
=(-1)^{\rho^\vee(\lambda\mp w_\circ^J\beta)+\rho^\vee_J(-\lambda_J\mp\beta)} q^{\frac12((\lambda_J,\lambda_J)-(\lambda,\lambda))+(\lambda_J,\rho_J)-(\lambda,\rho)}
T^\pm_{w_J}(v).
\end{multline*}
Similarly,
\begin{multline*}
 \sigma^{J^\star}\circ \sigma(v)=(-1)^{\rho^\vee(\lambda\mp\beta)} q^{-\frac12((\lambda,\lambda)-(\beta,\beta))-(\lambda,\rho)}\sigma^{J^\star}(T_{w_\circ}^\pm(v))\\
 =(-1)^{\rho^\vee(\lambda\mp\beta)+\rho^\vee_{J^\star}(w_\circ\lambda_J\mp w_\circ\beta)}  q^{\frac12((\lambda_J,\lambda_J)-(\lambda,\lambda))
 -(w_\circ\lambda_J,\rho_{J^\star})-(\lambda,\rho)}
(T_{w_\circ^{J^\star}}^\pm)^{-1}(T_{w_\circ}^\pm(v))
\\
=(-1)^{\rho^\vee(\lambda\mp\beta)+\rho^\vee_{J}(-\lambda_J\pm \beta)}  q^{\frac12((\lambda_J,\lambda_J)-(\lambda,\lambda))
 +(\lambda_J,\rho_{J})-(\lambda,\rho)}
T_{w_J}^\pm(v)
\end{multline*}
since $\rho^\vee_{J^\star}(w_\circ \mu)=-\rho^\vee_J(\mu)$, $\mu\in P$ and $w_\circ\rho_{J^\star}=-\rho_J$.
Since
$-\rho^\vee_J(\beta)=\rho^\vee_J(w_\circ^J\beta)$ and $\rho^\vee(\beta-w\beta)=\rho^\vee_J(\beta-w\beta)$
for any $w\in W_J$, it follows that
$\sigma^{J^\star}\circ \sigma=\sigma\circ \sigma^J$.
\end{proof}

\begin{proposition}\label{prop:cacti}
Let $J,J'\in\mathscr J$ be orthogonal. Then 
$\sigma^{J\sqcup J'}=\sigma^J\circ\sigma^{J'}$.
\end{proposition}
\begin{proof}
As before we may assume, without loss of generality, that $I=J\sqcup J'$. Then $w_\circ=w_\circ^{J}w_\circ^{J'}=
w_\circ^{J'}w_\circ^J$ and hence $T^\pm_{w_\circ}=T^\pm_{w_\circ^J}\circ T^\pm_{w_\circ^{J'}}=T^\pm_{w_\circ^{J'}}\circ T^\pm_{w_\circ^J}$ by Lemma~\ref{lem:product can elts br}.
Let $\lambda\in P^+$, $\lambda_J\in P^+_J$, $\lambda_{J'}\in P^+_{J'}$, $\beta\in P$
and $v\in \mathcal I_\lambda(V)\cap \mathcal I^J_{\lambda^J}(V)\cap \mathcal I^J_{\lambda^{J'}}(V)$. Then $\gamma_J=\lambda_J-\beta\in 
\sum_{j\in J}\ZZ_{\ge 0}\alpha_j$, $\gamma_{J'}=\lambda_{J'}-\beta\in\sum_{j\in J'}\ZZ_{\ge 0}\alpha_{j'}$,
and $\gamma=\lambda-\beta=\gamma_{J}+\gamma_{J'}$.
Then we can rewrite~\eqref{eq:defn-eta^J} and~\eqref{eq:defn-eta} as 
\begin{align*}
&\sigma^{J}(v)=(-1)^{\rho^\vee_{J}(\gamma_{J})} q^{-\frac12(\gamma_{J},\gamma_{J})+(\lambda_{J},\gamma_{J})-(\lambda_{J},\rho_{J})} T^+_{w_\circ^{J}}(v)
\\
&\sigma^{J'}(v)=(-1)^{\rho^\vee_{J'}(\gamma_{J'})} q^{-\frac12(\gamma_{J'},\gamma_{J'})+(\lambda_{J'},\gamma_{J'})-(\lambda_{J'},\rho_{J'})} T^+_{w_\circ^{J'}}(v)
\\
&\sigma(v)=(-1)^{\rho^\vee(\gamma)} q^{-\frac12(\gamma,\gamma)+(\lambda,\gamma)-(\lambda,\rho)} T^+_{w_\circ}(v).
\end{align*}
Since $w_\circ^J(\gamma_{J'})=\gamma_{J'}$ we have 
\begin{multline*}
\sigma^{J}(\sigma^{J'}(v))
=(-1)^{\rho^\vee_{J}(\gamma_J)+\rho^\vee_{J'}(\gamma_{J'})} q^{-\frac12((\gamma_J,\gamma_{J})+(\gamma_{J'},\gamma_{J'}))-(\lambda_{J},\rho_{J})
-(\lambda_{J'},\rho_{J'})+(\lambda_J,\gamma_J)+(\lambda_{J'},\gamma_{J'})}
T^+_{w_\circ}(v)
\\
=(-1)^{\rho^\vee(\gamma)} q^{-\frac12 (\gamma,\gamma)-(\lambda,\rho)+(\lambda,\gamma)}T^+_{w_\circ}(v)=\sigma(v),
\end{multline*}
since $\rho^\vee(\gamma)=\rho^\vee_J(\gamma_J)+\rho^\vee_{J'}(\gamma_{J'})$, $(\gamma,\gamma)=(\gamma_J,\gamma_J)+(\gamma_{J'},\gamma_{J'})$,
$(\lambda_J,\zeta)+(\lambda_{J'},\zeta')=(\lambda_J+\lambda_{J'},\zeta+\zeta')=(\lambda,\zeta+\zeta')$ for any $\zeta\in\sum_{j\in J}\mathbb Q\alpha_j$,
$\zeta'\in\sum_{j\in J'}\mathbb Q\alpha_{j'}$.
\end{proof}

\begin{proof}[Proof of Theorem~\ref{thm:main thm}]
Parts~\ref{thm:main thm.a} (respectively, \ref{thm:main thm.b}, \ref{thm:main thm.c}) of Theorem~\ref{thm:main thm} were established in 
Lemma~\ref{prop:prop eta^J}\ref{prop:prop eta^J.a} (respectively, Proposition~\ref{prop:cacti}, Proposition~\ref{prop:factor eta_J}).
\end{proof}

\subsection{Kernels of actions of cactus groups}
For any $V\in\mathscr O^{int}_q(\lie g)$, denote $\Phi_V$ be the subgroup of~$\operatorname{GL}_\kk(V)$ generated by the $\sigma^J_V$, $J\in\mathscr J$.
We need the following basic properties of~$\Phi_V$.
\begin{lemma}\label{lem:invariance}
For any injective morphism $f:V'\to V$ in~$\mathscr O^{int}_q(\lie g)$ the assignments $\sigma^J_V\mapsto \sigma^J_{V'}$, $J\in\mathscr J$
define a surjective homomorphism 
 $f^*:\Phi_{V}\to\Phi_{V'}$. In particular, if $f$ is an isomorphism then so is~$f^*$.
\end{lemma}
\begin{proof}
Let $V''=f(V')$
By Theorem~\ref{thm:prop-eta}\ref{thm:prop-eta.a'}, we have $\sigma^J_V\circ f=f\circ \sigma^J_{V'}$ for all $J\in\mathscr J$
and so the group~$\Phi_V$ acts on~$V''$, that is, there is a canonical homomorphism of groups $\rho:\Phi_V\to \operatorname{GL}_\kk(V'')$.
Clearly, the assignments $g\mapsto f\circ g \circ f^{-1}$, $g\in \operatorname{GL}_\kk(V'')$ define 
an isomorphism $\rho_f:\operatorname{GL}_\kk(V'')\to \operatorname{GL}_\kk(V')$. 
Let $f^*=\rho_f\circ\rho:\Phi_V\to \operatorname{GL}_\kk(V')$. We claim that~$f^*(\Phi_V)=\Phi_{V'}$. Indeed, 
$f^*(\sigma^J_V)=\sigma^J_{V'}$ for all $J\in\mathscr J$. Since~$\Phi_{V'}$ is generated by the $\sigma^J_{V'}$, $\Phi_V$ is generated by the $\sigma^J_V$,
$J\in\mathscr J$, and $f^*$ is a 
homomorphism of groups, the assertion follows.
\end{proof}

\begin{proposition}\label{prop:ficus}
Let $V\in\mathscr O^{int}_q(\lie g)$. Then 
$\Phi_V\cong \Phi_{\underline V}$ where $\underline V=\bigoplus\limits_{\lambda\in P^+\,:\, \Hom_{U_q(\lie g)}(V_\lambda,V)\not=0} V_\lambda$.
In particular, for any $V\in\mathscr O^{int}_q(\lie g)$ the group~$\Phi_V$ is a quotient of~$\Phi_{\mathcal C_q(\lie g)}$
where $\mathcal C_q(\lie g)=\bigoplus_{\lambda\in P^+} V_\lambda$.
\end{proposition}
\begin{proof}
Fix $f_\lambda\in \Hom_{U_q(\lie g)}(V_\lambda,V)\setminus\{0\}$ 
for all $\lambda\in P^+$ with $\Hom_{U_q(\lie g)}(V_\lambda,V)\not=0$ and let $f:\underline V\to V$ be the direct sum of these~$f_\lambda$. Then~$f$ is 
injective.
Applying Lemma~\ref{lem:invariance} with $V'=\underline V$ we obtain 
a surjective group homomorphism $f^*:\Phi_V\to\Phi_{\underline V}$. It remains to prove that its kernel is trivial. We apply 
Lemma~\ref{lem:annihilators} with 
$R=\kk[\Phi_V]$ and $S=\{ g-1\,:\, g\in\Phi_V\}\subset R$. Since $\Phi_V$ is a subgroup of~$\operatorname{GL}_\kk(V)$,
$\Ann_S V=\{0\}$. By our choice of~$\underline V$, $M=V$ and $M'=f(\underline V)$ satisfy the assumptions of Lemma~\ref{lem:annihilators} and so 
$\Ann_S f(\underline V)=\{0\}$. Since $\ker f^*=\{ g\in \Phi_V\,:\, g\circ f=\id_{\underline V}\}$,
it follows that $\ker f^*$ is trivial. The second assertion is immediate from the first one and Lemma~\ref{lem:invariance}.
\end{proof}

\section{An action of \texorpdfstring{$\Cact_W$}{Cact\_W} on \texorpdfstring{$\mathbf c$}{c}-crystal bases and proof of Theorem~\ref{thm:main thm2}}\label{sect:action on crystals}

Retain the notation of~\S\ref{subs:prelim crystal bases} and observe that the assignment $(l,k,s)\mapsto (l,l-k,-s)$, $(l,k,s)\in\mathbb D$, defines an involution 
on~$\mathbb D$.
The following is the main result of this section.
\begin{theorem}\label{thm:eta c-crystal basis}
Let $\lie g$ be reductive. 
Let $\mathbf c:\mathbb D\to\mathbb Q(z)^\times$ satisfying
\begin{equation}\label{eq:cnd for eta}
\underline{\mathbf c}_{l,k,s}=
\underline{\mathbf c}_{l,l-k,-s},\qquad \underline{\mathbf c}_{l,0,-l}=1,\qquad 
(l,k,s)\in\mathbb D
\end{equation}
in the notation of Lemma~\ref{lem:2nd defn Kash op}. 
Then for any $V\in\mathscr O^{int}_q(\lie g)$
\begin{enumerate}[label={\rm(\alph*)},leftmargin=*]
 \item\label{thm:eta c-crystal basis.a} $\sigma^I(L)=L$ for any $(\mathbf c,I)$-monomial lattice~$L$ in~$V$;
 \item\label{thm:eta c-crystal basis.b} If $(L,B)$ is a $\mathbf c$-crystal basis such that $B_+=\{b\in B\,:\, \tilde e_{i,1}^{\mathbf c}(b)=0,\, i\in I\}$
 is a basis of~$L_+/qL_+$ where $L_+=L\cap V_+$, 
 then the induced $\mathbb Q$-linear map 
 $\tilde\sigma^I: L/qL\to L/qL$ preserves~$B$.
\end{enumerate}

\end{theorem}
\begin{proof}
We abbreviate $\sigma=\sigma^I$, $V_+^I=V_+$ and 
$\mathsf M^{\mathbf c}(v_+)=\mathsf M_I^{\mathbf c}(v_+)$ for any homogeneous $v_+\in V_+$.
The key ingredient of our argument is the following
\begin{proposition}\label{prop:eta-crystal-op}
Let $\lie g$ be reductive,   
let $V\in\mathscr O^{int}_q(\lie g)$ and let $\mathbf c:\mathbb D\to \mathbb Q(z)^\times$. 
Then
\begin{enumerate}[label={\rm(\alph*)},leftmargin=*]
 \item\label{prop:eta-crystal-op.a} If $\mathbf c$ satisfies the first condition in~\eqref{eq:cnd for eta} then 
 $\sigma\circ\tilde e_{i,s}^{\mathbf c}=\tilde e_{i^\star,-s}^{\mathbf c}\circ\sigma$ in~$\End_\kk V$ for any $i\in I$, $s\in\ZZ$.
 \item\label{prop:eta-crystal-op.b} If $\mathbf c$ satisfies~\eqref{eq:cnd for eta}  then

 $\sigma(\mathsf M^{\mathbf c}(v_+))=\mathsf M^{\mathbf c}(v_+)$ for any homogeneous $v_+\in V_+$.
 \end{enumerate}
\end{proposition}
\begin{proof}
In view of Lemma~\ref{lem:sl2 decomp} and~\eqref{eq:kash-op-defn-alt},
to prove~\ref{prop:eta-crystal-op.a} 
it suffices to verify the identity for all~$v\in V$ of the form 
$v=F_i^{(k)}(u)$, $u\in\ker E_i\cap \ker(K_{\alpha_i}-q_i^l)$, $0\le k\le l$. We have 
\begin{equation}\label{eq:lhs}
\sigma\circ \tilde e_{i,s}^{\mathbf c}(v)=\underline{\mathbf c}_{l,k,s}(q_i)\sigma(F_i^{(k-s)}(u))=\underline{\mathbf c}_{l,k,s}(q_i) E_{i^\star}^{(k-s)}(\sigma(u)).
\end{equation}
We need the following 
\begin{lemma}\label{lem:eta ker to ker}
$\sigma(u)\in\ker F_{i^\star}\cap \ker(K_{\alpha_{i^\star}}-q_{i^{\star}}^{-l})$.
\end{lemma}
\begin{proof}
Indeed, $F_{i^\star}(\sigma(u))=\sigma(E_i(u))=0$ and $K_{\alpha_i^\star}(\sigma(u))=\sigma(K_{-\alpha_i}(u))=q_i^{-l}\sigma(u)=q_{i^\star}^{-l}\sigma(u)$.
\end{proof}
Using Lemmata~\ref{lem:eta ker to ker} and~\ref{lem:2nd defn Kash op}, we obtain 
\begin{equation}\label{eq:rhs}
\tilde e_{i^\star,-s}^{\mathbf c}(\sigma(v))=\tilde e_{i^\star,-s}^{\mathbf c}(E_{i^\star}^{(k)}(\sigma(u)))
=\underline{\mathbf c}_{l,l-k,-s}(q_{i^\star}) 
E_{i^\star}^{(k-s)}(\sigma(u)).
\end{equation}
Since $q_{i^\star}=q_i$, by assumptions of Proposition~\ref{prop:eta-crystal-op} we have $\underline{\mathbf c}_{l,l-k,-s}(q_{i^\star})=\underline{\mathbf c}_{l,k,s}(q_i)$.
Then~\eqref{eq:lhs} and~\eqref{eq:rhs} imply that $ \tilde e_{i^\star,-s}^{\mathbf c}(\sigma(v))=
\sigma(\tilde e_{i,s}^{\mathbf c}(v))$.

To prove part~\ref{prop:eta-crystal-op.b}, we need the following
\begin{lemma}\label{lem:extrem vecs}
Suppose that $\mathbf c_{l,0,-l}=1/(l)_z!$ for all $l\in\ZZ_{\ge0}$ (that is, $\underline{\mathbf c}_{l,0,-l}=1$ in the notation of Lemma~\ref{lem:2nd defn Kash op}).
Then for any $v_+\in V_+(\lambda)$ and $\mathbf i=(i_1,\dots,i_m)\in I^m$ reduced, $\lambda\in P^+$ we have in the notation of~\eqref{eq:F i lambda},
$F_{\mathbf i,\lambda}(v_+)=\tilde e_{i_1,-a_1(\mathbf i,\lambda)}^{\mathbf c}\cdots \tilde e_{i_m,-a_m(\mathbf i,\lambda)}^{\mathbf c}(v_+)$. In particular, 
$$
\sigma(v_+)=\tilde e_{i_1,-a_1(\mathbf i,\lambda)}^{\mathbf c}\cdots \tilde e_{i_N,-a_N(\mathbf i,\lambda)}^{\mathbf c}(v_+)
$$
where $\mathbf i=(i_1,\dots,i_N)\in R(w_\circ)$.
\end{lemma}
\begin{proof}
We use induction on~$m$, the case $m=0$ being trivial. For $\mathbf i$ and~$\lambda\in P^+$ fixed we abbreviate $a_k=a_k(\mathbf i,\lambda)$.
For the inductive step, note that $F_{\mathbf i,\lambda}=F_{i_1}^{(a_1)}F_{\mathbf i',\lambda}$
where $\mathbf i'=(i_2,\dots,i_m)$ and so $F_{\mathbf i,\lambda}(v_+)=F_{i_1}^{(a_1)}(v')$ where $v'=F_{\mathbf i',\lambda}(v_+)=
\tilde e_{i_2,-a_2}^{\mathbf c}\cdots \tilde e_{i_m,-a_m}^{\mathbf c}(v_+)$ by the induction hypothesis. Since $v'\in\ker E_{i_1}$, it follows from
assumptions of the lemma and the first identity in~\eqref{eq:kash-op-defn-alt} with $i=i_1$, $k=0$ and $l=a_1=-s$
that $F_{i_1}^{(a_1)}(v')=\tilde e_{i_1,-a_1}^{\mathbf c}(v')=\tilde e_{i_1,-a_1}^{\mathbf c}\cdots \tilde e_{i_m,-a_m}^{\mathbf c}(v_+)$. Since $\sigma(v_+)=F_{w_\circ,\lambda}(v_+)$
by Lemma~\ref{lem:exist eta}\ref{lem:exist eta.a} with $w=1$, the 
second assertion follows from the first and Lemma~\ref{lem:EF lambda w}\ref{lem:EF lambda w.a}.
\end{proof}

Suppose now that $v\in\mathsf M^{\mathbf c}(v_+)$ that is $v=\tilde e_{j_1,m_1}^{\mathbf c}\cdots \tilde e_{j_r,m_r}^{\mathbf c}(v_+)\in \mathsf M^{\mathbf c}(v_+)$,
for some $(j_1,\dots,j_r)\in I^r$ and $(m_1,\dots,m_r)\in\ZZ$. Using Lemma~\ref{lem:extrem vecs} and Proposition~\ref{prop:eta-crystal-op}\ref{prop:eta-crystal-op.a}, we obtain  
$$
\sigma(v)=\tilde e_{j_1^\star,-m_1}^{\mathbf c}\cdots \tilde e_{j_r^\star,-m_r}^{\mathbf c}(\sigma(v_+))
=\tilde e_{j_1^\star,-m_1}^{\mathbf c}\cdots \tilde e_{j_r^\star,-m_r}^{\mathbf c}\tilde e_{i_1,-a_1}^{\mathbf c}\cdots \tilde e_{i_N,-a_N}^{\mathbf c}(v_+)\in\mathsf M^{\mathbf c}(v_+),
$$
where $\mathbf i=(i_1,\dots,i_N)\in R(w_\circ)$ and $a_k=a_k(\mathbf i,\lambda)$, $1\le k\le N$. Thus, $\sigma(\mathsf M^{\mathbf c}(v_+))\subset \mathsf M^{\mathbf c}(v_+)$.
Since~$\sigma$ is an involution, it follows that $\sigma(\mathsf M^{\mathbf c}(v_+))=\mathsf M^{\mathbf c}(v_+)$.
\end{proof}

Part~\ref{thm:eta c-crystal basis.a} is immediate from Proposition~\ref{prop:eta-crystal-op}\ref{prop:eta-crystal-op.b}.
In particular, for each $(\mathbf c,I)$-monomial lattice $L$ in~$V$ the involution~$\sigma_V$ induces 
an involution $\tilde \sigma$ on the $\mathbb Q$-vector space~$\tilde L=L/q L$ satisfying 
\begin{equation}\label{eq:eta cr op}
\tilde \sigma\circ \tilde e_{i,s}^{\mathbf c}=\tilde e_{i^\star,-s}^{\mathbf c}\circ\tilde\sigma.
\end{equation}
The following is immediate from Proposition~\ref{prop:eta-crystal-op}\ref{prop:eta-crystal-op.b}.
\begin{corollary}\label{lem:tilde eta mons}
Let $L$ be a $(\mathbf c,I)$-monomial lattice in~$V$. Then $\tilde\sigma(\tilde{\mathsf M}^{\mathbf c}(\tilde v_+))=\tilde{\mathsf M}^{\mathbf c}(\tilde v_+)$
for any $\tilde v_+\in L_+/qL_+$.
\end{corollary}
Using the assumptions of part~\ref{thm:eta c-crystal basis.b} of Theorem~\ref{thm:eta c-crystal basis} we conclude that 
$\bigcup_{b_+\in B^+}\tilde{\mathsf M}^{\mathbf c}(b_+)=B\cup\{0\}$.
Then it follows from Corollary~\ref{lem:tilde eta mons}
that $\tilde\sigma$ preserves $B\cup\{0\}$. Since $\tilde\sigma$ is an involution, it follows that $\tilde\sigma(B)=B$. This completes the proof of
Theorem~\ref{thm:eta c-crystal basis}\ref{thm:eta c-crystal basis.b}.
\end{proof}
Note that~\eqref{eq:eta cr op} implies that 
for any upper crystal basis $(L,B)$ of~$V\in\mathscr O^{int}_q(\lie g)$ the operator $\tilde \sigma^I_V$ satisfies 
\begin{equation}\label{eq:eta cr op up}
\tilde\sigma\circ (\tilde e_i^{up})^s=(\tilde e_{i^\star}^{up})^{-s}\circ\tilde\sigma^I_V.
\end{equation}
In particular, we obtain the following
\begin{corollary}\label{cor:char eta}
Let $\lambda\in P^+$ and $(L_\lambda,B_\lambda)$ be the upper crystal basis of~$V_\lambda$. If $f$ is any non-zero map $B_\lambda\cup\{0\}\to B_\lambda\cup\{0\}$ 
satisfying~\eqref{eq:eta cr op up} then~$f=\tilde\sigma^I_{V_\lambda}|_{B_\lambda\cup\{0\}}$.
\end{corollary}

\begin{proof}[Proof of Theorem~\ref{thm:main thm2}]
Note that  
\begin{equation}\label{eq:low upper Kash-underl}
\underline{\mathbf c}^{low}_{l,k,s}=1,\qquad 
\underline{\mathbf c}^{up}_{l,k,s}=\frac{(l-k+s)_z! (k-s)_z!}{(l-k)_z! (k)_z!},\qquad (l,k,s)\in\mathbb D.
\end{equation}
It is now immediate that~\eqref{eq:cnd for eta} holds for $\mathbf c\in\{\mathbf c^{up},\mathbf c^{low}\}$.

Furthermore, by Lemma~\ref{lem:facts on crystal bases} and Remark~\ref{rem:restr crys bas}, Theorem~\ref{thm:eta c-crystal basis} applies to every $\mathbf c$-crystal 
basis at $q=0$ for any $J\in\mathscr J$ with $\lie g$ replaced by $\lie g^J$ and $\mathbf c\in\{\mathbf c^{up},\mathbf c^{low}\}$.
Thus, $\sigma^J$ preserves a lower or upper crystal lattice $L$ and $\tilde\sigma^J$ preserves~$B$.

In particular, $\Phi_V$ acts on $L$ and this action factors through to an action on $L/qL$ and induces an action on~$B$ by permutations. 
\end{proof}
\begin{remark}\label{rem:crystal Weyl group}
Let $L$ be an upper crystal lattice for~$V\in\mathscr O^{int}_q(\lie g)$.
It follows from the definition of~$\sigma^i_V$ that for any $v\in L(\beta)$, $\beta\in P$ we have $\sigma^i_V(v)=\tilde e_i^{-\beta(\alpha_i^\vee)}(v)$.
In particular, the action of~$\tilde\sigma^i_V$ on an upper crystal basis~$(L,B)$ of~$V$ coincides with Kashiwara's crystal Weyl group action (see~\cite{Kas94}).
\end{remark}

We conclude this section with a discussion of the action of~$\Cact_W$ on upper global crystal bases.
Let $\bar\cdot$ be any field involution of~$\kk$ such that $\overline{q^{\frac1{2d}}}=q^{-\frac1{2d}}$. 
\begin{proposition}\label{prop:sigma bar}
Let~$(L,B)$ be an upper crystal basis of~$V\in\mathscr O^{int}_q(\lie g)$ and let $G^{up}(B)$ be the corresponding upper global crystal basis.
Denote by $\bar\cdot$ the (unique) additive map $V\to V$ satisfying $\overline{f\cdot b}=\overline f\cdot \overline b$ for all $f\in\kk$, $b\in G^{up}(B)$.
Then $\overline{\sigma^J(b)}=\sigma^J(b)$ for any $J\in\mathscr J$.
\end{proposition}
\begin{proof}
Denote by~$\bar\cdot$ the
ring automorphism of~$U_q(\lie g)$ satisfying $\overline{E_i}=E_i$, $\overline{F_i}=F_i$, $i\in I$, $\overline{K_\lambda}=K_{-\lambda}$, $\lambda\in\frac12P$
and $\overline{f u}=\overline f\cdot \overline u$ for all $f\in\kk$, $u\in U_q(\lie g)$. 
The following is immediate from the properties of the upper global crystal basis (\cite{Kas93}).
\begin{lemma}
The map $\bar\cdot:V\to V$ defined in Proposition~\ref{prop:sigma bar}
satisfies 
\begin{equation}\label{eq:bar property}
\overline{u(v)}=\overline u(\overline v),\qquad v\in V,\, u\in U_q(\lie g).
\end{equation}
\end{lemma}

The following is immediate.
\begin{lemma}\label{lem:sigma bar}
Let $\eta:U_q(\lie g)\to U_q(\lie g)$ be any algebra automorphism commuting with~$\bar\cdot:U_q(\lie g)\to U_q(\lie g)$ and 
let $V\in\mathscr O^{int}_q(\lie g)$ with a fixed set~$V_0\subset V$ generating $V$ as a $U_q(\lie g)$-module. 
Let $\bar\cdot:V\to V$ be any map satisfying~\eqref{eq:bar property} and let 
$\sigma\in\End_\kk V$ be such that:
\begin{enumerate}[label={\rm(\roman*)},leftmargin=*]
 \item\label{lem:sigma bar.i} $\sigma(u(v))=\eta(u)(\sigma(v))$, $u\in U_q(\lie g)$, 
$v\in V$;
\item\label{lem:sigma bar.ii} $\sigma(\overline v)=\overline{\sigma(v)}$ for any $v\in V_0$. 
\end{enumerate}
Then $\sigma(\overline v)=\overline{\sigma(v)}$ for all~$v\in V$.
\end{lemma}
The set $V_0=G^{up}(B)\cap V_+^J$ generates $V$ as a $U_q(\lie g^J)$-module.
Clearly, $\theta_J$ commutes with $\bar\cdot$-involution on~$U_q(\lie g^J)$.
The condition~\ref{lem:sigma bar.i} of Lemma~\ref{lem:sigma bar} holds with $\eta=\theta_J$ by Proposition~\ref{prop:prop eta^J}\ref{prop:prop eta^J.b'}.
By Proposition~\ref{prop:prop eta^J}\ref{prop:prop eta^J.a'}, $\sigma^J(b)=F_{w_\circ^J,\lambda}(b)$ for any $b\in V_0(\lambda)$, $\lambda\in P^+_J$. 
Since $\overline{F_{w_\circ^J,\lambda}}=
F_{w_\circ^J,\lambda}$ and $\overline b=b$, the condition~\ref{lem:sigma bar.ii} of Lemma~\ref{lem:sigma bar} is also satisfied.
The assertion follows by Lemma~\ref{lem:sigma bar}.
\end{proof}

\section{\texorpdfstring{$\sigma^I$}{eta\^I} and the canonical basis}\label{sect:can bas}

\subsection{Automorphisms and skew derivations of localizations}
Let $R$ be a unital $\kk$-algebra. Given a monoid $\Gamma$ written multiplicatively and acting on~$R$ by algebra automorphisms, 
define the semidirect product of $R$ with the monoidal algebra~$\kk[\Gamma]$ of~$\Gamma$ as $R\tensor \kk[\Gamma]$
with the multiplication defined by
$$
(r\tensor \gamma)\cdot (r'\tensor \gamma')=r(\gamma\triangleright r')\tensor \gamma\gamma',\qquad r,r'\in R,\, \gamma,\gamma'\in\Gamma,
$$
where $\triangleright$ denotes the action of~$\Gamma$ on~$R$. Since $(r\tensor 1)(1\tensor \gamma)=r\tensor \gamma$, we will henceforth
omit the symbol~$\tensor$ when writing elements of~$R\rtimes\kk[\Gamma]$. In other words, $R\rtimes\kk[\Gamma]$
is generated by $R$ as a subalgebra and $\Gamma$ subject to the relations 
$$
\gamma\cdot r=(\gamma\triangleright r)\cdot \gamma,\qquad r\in R,\,\gamma\in\Gamma.
$$
The following characterization of cross products is immediate.
\begin{lemma}\label{lem:lift homs}
Let $f:R\to \widehat R$ be a homomorphism of $\kk$-algebras and let $g:\Gamma\to \widehat R$ 
be a homomorphism of multiplicative monoids. Suppose that $R$ is a $\kk[\Gamma]$-module algebra. Then 
assignments $r\cdot \gamma\mapsto f(r)\cdot g(\gamma)$, $r\in R$, $\gamma\in\Gamma$ define a homomorphism of $\kk$-algebras
if and only if 
\begin{equation}\label{eq:cross prod cond}
f(\gamma\triangleright r)g(\gamma)=g(\gamma)f(r),\qquad r\in R,\,\gamma\in\Gamma.
\end{equation}
\end{lemma}

Let $S$ be a submonoid of~$R\setminus\{0\}$. Denote $S^{op}$ the opposite monoid of~$S$ and denote its elements 
by $[s]$, $s\in S$. 
Suppose that $R$ is a $\kk[S^{op}]$-module algebra with $[s]\lact r=\Sigma_s(r)$, $s\in S$ where $\Sigma_s$ is an algebra automorphism of~$R$ and 
assume that 
\begin{equation}\label{eq:admiss aut}
\Sigma_s(s)=s,\qquad s\in S.
\end{equation}
Denote $R[S^{-1}]:=(R\rtimes \kk[S^{op}])/\langle s[s]-1\,:\, s\in S\rangle$. We say that $S$ as above is an Ore submonoid if 
\begin{equation}\label{eq:Ore subm}r s=s\Sigma_s(r),\qquad r\in R,\, s\in S.
\end{equation}
We use the convention that $\Sigma_{\lambda s}=\Sigma_s$ for all $\lambda\in\kk^\times$.
This notation is justified by the following
\begin{lemma}\label{lem:spec Ore}
Suppose that~\eqref{eq:admiss aut} holds. Then 
the following are equivalent:
\begin{enumerate}[label={\rm(\roman*)},leftmargin=*]
\item\label{lem:spec Ore.ii} the natural homomorphism $\mathbf 1_{R,S}:R\to R[S^{-1}]$ is injective;
 \item\label{lem:spec Ore.i}
 $S$ is an Ore submonoid of~$R$ and
 the assignments $r\cdot[s]\mapsto rs^{-1}$, $r\in R$, $s\in S$ define
  an isomorphism 
$R[S^{-1}]\to\underline{R[S^{-1}]}$ where $\underline{R[S^{-1}]}$ is the Ore localization of~$R$ by~$S$;
\end{enumerate}
\end{lemma}
\begin{proof}
In $R\rtimes \kk[S^{op}]$ we have 
\begin{equation}\label{eq:comm-semi-dir-prod}[s]\cdot r=\Sigma_s(r)\cdot [s],\qquad s\in S,\, r\in R.
\end{equation}
In particular,
$[s]\cdot s=s\cdot [s]$, $s\in S$. 
Multiplying both sides of~\eqref{eq:comm-semi-dir-prod} by~$s$ on the left and on the right we
conclude that $r s=s \Sigma_s(r)$ in~$R[S^{-1}]$ for all $r\in R$, $s\in S$. This identity clearly holds in~$\mathbf 1_{R,S}(R)$. Since~$\mathbf 1_{R,S}$ is injective,
this implies that the corresponding identity holds in $R$ and so~$S$ satisfies the two-sided Ore condition and so
$R$ admits the Ore localization $\underline{R[S^{-1}]}$.
The assignments $r[s]\mapsto rs^{-1}$, $r\in R$, $s\in S$ define a surjective
homomorphism from $R[S^{-1}]\to \underline{R[S^{-1}]}$ which is easily seen to be injective.
Thus, \ref{lem:spec Ore.ii} implies \ref{lem:spec Ore.i}.

Conversely, the natural homomorphism $R\to \underline{R[S^{-1}]}$, $r\mapsto r\cdot 1$, $r\in R$ is injective. 
Since it equals the composition of~$\mathbf 1_{R,S}$ and the isomorphism $R[S^{-1}]\to\underline{R[S^{-1}]}$,
it follows that $\mathbf 1_{R,S}$ is injective.
\end{proof}
The following Lemma is immediate.
\begin{lemma}\label{lem:extrem vecs Ore subring}
Let $R$ be a $\kk$-algebra and $S\subset R\setminus\{0\}$ be an Ore submonoid. Let $R'$ be a $\kk$-subalgebra of~$R$
and suppose that $S'\subset R'\cap S$ is an Ore submonoid of~$R'$. Then $R'[S'{}^{-1}]$ is isomorphic to the subalgebra of $R[S^{-1}]$
generated by $R'$ and $\{ s'{}^{-1}\,:\, s'\in S'\}$.
\end{lemma}

\begin{lemma}\label{lem:lifts auto}
Suppose that~\eqref{eq:admiss aut} and the assumptions of Lemma~\ref{lem:spec Ore}\ref{lem:spec Ore.i} hold.
Let $\varphi:R\to R'$ be any $\kk$-algebra homomorphism, $S$ be an Ore submonoid of~$R$ and $S'$ be an Ore submonoid 
of~$R'$ such that $\varphi(S)\subset S'$.
Suppose that $\Sigma'_{\varphi(s)}\circ \varphi=\varphi\circ\Sigma_s$ for all $s\in S$.
Then there exists a unique homomorphism $\widehat\varphi:R[S^{-1}]\to R'[{S'}^{-1}]$ such that $\widehat\varphi|_R=\varphi$.
\end{lemma}

\begin{proof}
We apply Lemma~\ref{lem:lift homs} with $\Gamma=S^{op}$, $\widehat R=R'\rtimes \kk[{S'}^{op}]$ and $g:S^{op}\to \widehat R$
defined by $g([s])=[\varphi(s)]$. Then
$$
[\varphi(s)]\varphi(r)=\Sigma'_{\varphi(s)}(\varphi(r))[\varphi(s)]=\varphi(\Sigma_s(r))[\varphi(s)]=\varphi([s]\triangleright r)[\varphi(s)],
$$
and so~\eqref{eq:cross prod cond} holds.
Therefore, 
the assignments $r[s]\mapsto \varphi(r)[\varphi(s)]$, $r\in R$, $s\in S$,
define a homomorphism $\widehat{\widehat\varphi}:R\rtimes \kk[S^{op}]\to R'\rtimes \kk[{S'}^{op}]$.
Since $\widehat{\widehat\varphi}(s[s])=\varphi(s)[\varphi(s)]$ it follows that the image of the defining ideal of~$R[S^{-1}]$ 
under $\widehat{\widehat\varphi}$ is contained in the defining ideal of~$R'[{S'}^{-1}]$. Thus, $\widehat{\widehat\varphi}$ factors through to the desired homomorphism
$\widehat\varphi:R[S^{-1}]\to R'[{S'}^{-1}]$.
\end{proof}

Let $L_\pm:R\to R'$ be $\kk$-algebra homomorphisms and $E:R\to R'$ be a $\kk$-linear map. We say that $E$ is an $(L_+,L_-)$-derivation 
from $R$ to~$R'$ if $E(rr')=E(r)L_+(r')+L_-(r)E(r')$ for all $r,r'\in R$. Denote $\operatorname{Der}_{L_+,L_-}(R,R')$ the $\kk$-subspace of~$\Hom_\kk(R,R')$
of $(L_+,L_-)$-derivations from~$R$ to~$R'$.
We refer to an $(L,L^{-1})$-derivation as an $L$-derivation and abbreviate $\operatorname{Der}_{L_+,L_-}R=\operatorname{Der}_{L_+,L_-}(R,R)$. The following is immediate.
\begin{lemma}\label{lem:diffs restrs}
Let $R_0$ be a generating subset of~$R$ and let 
$D,D'\in \operatorname{Der}_{L_+,L_-}(R,R')$. Then $D|_{R_0}=D'|_{R_0}$ implies that~$D=D'$.
\end{lemma}

Given $r'\in R'$, denote by $D^\pm_{r'}$ the linear maps $R\to R'$
\begin{equation}\label{eq:D +-}
D^-_{r'}(x)=r' L_+(x)-L_-(x)r',\qquad D^+_{r'}(x)=L_-(x)r'-r'L_+(x),\qquad x\in R.
\end{equation}

\begin{lemma}\label{lem:construct diffs}
Let $L_\pm:R\to R'$ be $\kk$-algebra homomorphisms. The assignments $r'\mapsto D^+_{r'}$ (respectively, $r'\mapsto D^-_{r'}$), $r'\in R'$
define $\kk$-linear maps $R'\to \operatorname{Der}_{L_+,L_-}(R,R')$.
\end{lemma}
\begin{proof}
For any $x,x'\in R$ we have 
\begin{multline*}
D^-_{r'}(xx')=r' L_+(xx')-L_-(xx')r'\\=(r'L_+(x)-L_-(x)r')L_+(x')+L_-(x)(r'L_+(x')-L_-(x')r')\\=D^-_{r'}(x)L_+(x')+L_-(x)D^-_{r'}(x').
\end{multline*}
Thus, $D^-_{r'}\in\operatorname{Der}_{L_+,L_-}(R,R')$.
The computation for~$D^+_{r'}$ is similar and is omitted. The linearity of both maps in~$r'$ is obvious.
\end{proof}

\subsection{Gelfand-Kirillov model for the category \texorpdfstring{$\mathscr O^{int}_q(\lie g)$}{O int\_q(g)}}\label{subs:GK-model}
Throughout this section we mostly follow the notation from~\cite{BR}*{Section~6}. Let $\Gamma$ be the monoid~$P^+$ written multiplicatively,
with its elements denoted by~$v_\lambda$, $\lambda\in P^+$. Let~$\mathcal A_q(\lie g)$ be an isomorphic copy of~$U_q^-(\lie g)$ whose 
generators are denoted by~$x_i$, $i\in I$. We denote the degree of a homogeneous element $x\in\mathcal A_q(\lie g)$ with 
respect to its $Q$-grading by $|x|\in -Q^+$.
Define an action of~$\Gamma$ on~$\mathcal A_q(\lie g)$ by 
$v_\lambda\triangleright x=q^{(\lambda,|x|)} x$ for~$x\in \mathcal A_q(\lie g)$ homogeneous.
Let $\mathcal B_q(\lie g)=\mathcal A_q(\lie g)\rtimes\kk[\Gamma]$. In particular, we have 
\begin{equation}\label{eq:comm rel B}
v_\lambda x=q^{(\lambda,|x|)}x v_\lambda
\end{equation}
for all $\lambda\in P^+$ and for all $x\in \mathcal A_q(\lie g)$ homogeneous.
We extend the $Q$-grading on~$\mathcal A_q(\lie g)$ to a $P$-grading on $\mathcal B_q(\lie g)$ via $|v_\lambda|=\lambda$
for  $\lambda\in P^+$.
Let $\mathscr O_q(\lie g)$ be the category of~$U_q(\lie g)$-modules whose objects satisfy all assumptions on objects of~$\mathscr O^{int}_q(\lie g)$
except that we do not assume that the $E_i$, $F_i$, $i\in I$, act locally nilpotently while assuming that all weight subspaces are finite dimensional.
The following essentially coincides with~\cite{BR}*{Lemma~6.1}.
\begin{lemma}
\label{lem:action}
The algebra $\mathcal B_q(\lie g)$ is a module algebra in the category~$\mathscr O_q(\lie g)$ with respect to the action given by the following formulae for all $\lambda\in\frac12P$,
$i\in I$
\begin{itemize}
\item $K_\lambda(y)=q^{(\lambda,|y|)}y$ for all homogeneous elements $y\in\mathcal B_q(\lie g)$ and $\lambda\in\frac12 P$;
\item $\displaystyle F_i(y)=\frac{x_i K_{\frac12\alpha_i}(y)-K_{-\frac12\alpha_i}(y)x_i}{q_i-q_i^{-1}}$ for all $y\in \mathcal B_q(\lie g)$ 
and thus is a $K_{\frac12\alpha_i}$-derivation of~$\mathcal B_q(\lie g)$;
\item $E_i$ is the unique $K_{\frac12\alpha_i}$-derivation of $\mathcal B_q(\lie g)$ 
such that 
$E_i(x_j)=\delta_{ij}$ and $E_i(v_\mu)=0$ for all $i,j\in I$, $\mu\in P^+$.
\end{itemize}
\end{lemma}
Thus for all $x,y\in\mathcal B_q(\lie g)$ we have  
\begin{equation}\label{eq:leibnitz-1}
X_i(xy)=X_i(x)K_{\frac12\alpha_i}(y)+K_{-\frac12\alpha_i}(x) X_i(y)
\end{equation}
and more generally, for all $n\ge 0$
\begin{equation}\label{eq:leibnitz rule}
X_i^{(n)}(xy)=\sum_{r+s=n}X_i^{(r)}K_{-\frac s2\alpha_i}(x)X_i^{(s)}K_{\frac r2\alpha_i}(y)
\end{equation}
where $X_i$ is either~$E_i$ or~$F_i$.
The following is immediate from the defintion of~$\mathcal B_q(\lie g)$ and its $U_q(\lie g)$-module structure.
\begin{corollary}\label{cor:mod alg decomp}
$\mathcal B_q(\lie g)=\sum_{\lambda\in P^+} \mathcal A_q(\lie g)v_\lambda$ where $\mathcal A_q(\lie g)v_\lambda$ is a 
$U_q(\lie g)$-submodule of~$\mathcal B_q(\lie g)$ for each~$\lambda\in P^+$ and the sum is direct.
\end{corollary}
In the sequel we will also use~$E_i^*$ which is defined as the unique $K_{-\frac12\alpha_i}$-derivation of~$\mathcal B_q(\lie g)$
satisfying $E_i^*(x_j)=\delta_{i,j}$, $E_i^*(v_\lambda)=0$ for all $\lambda\in P^+$, $j\in I$. It is easy to check that 
$E_i^*(x)=(E_i(x^*))^*$, $x\in \mathcal A_q(\lie g)$, where ${}^*:\mathcal A_q(\lie g)\to \mathcal A_q(\lie g)$ is the unique anti-involution preserving the~$x_i$, $i\in I$.

In the spirit of~\cite{JL},
using the decomposition from Corollary~\ref{cor:mod alg decomp} we can define a linear map $\mathbf j:\mathcal B_q(\lie g)\to \mathcal A_q(\lie g)$ by 
\begin{equation}\label{eq:defn-joseph-map}\mathbf j(x\cdot v_\lambda)=q^{-\frac12(\lambda,|x|)} x
\end{equation}
for all $\lambda\in P^+$ and $x\in \mathcal A_q(\lie g)$ homogeneous. Clearly, $\mathbf j|_{\mathcal A_q(\lie g)v_\lambda}$ is a bijection onto~$\mathcal A_q(\lie g)$.
\begin{lemma}
For any symmetrizable Kac-Moody~$\lie g$ we have:
\begin{enumerate}[label={\rm(\alph*)},leftmargin=*]
\item\label{lem:joseph-map.a} $\mathbf j$ is a surjective homomorphism of~$U_q^+(\lie g)$-modules, with respect to the action defined in Lemma~\ref{lem:action}.
\item\label{lem:joseph-map.b}
$
\mathbf j(x\cdot y)
=q^{\frac12(\lambda,|\mathbf j(y)|)-\frac12(\mu,\mathbf j(|x|))}
\mathbf j(x)\cdot\mathbf j(y)
$ for all $x\in V_\lambda$, $y\in V_\mu$ homogeneous.
\end{enumerate}
\label{lem:joseph-map}
\end{lemma}
\begin{proof}
Part
~\ref{lem:joseph-map.a} is easily checked using Corollary~\ref{cor:mod alg decomp}. 
To prove part~\ref{lem:joseph-map.b} note that $x=q^{\frac12(\lambda,|\mathbf j(x)|)}\mathbf j(x)v_\lambda$ for all $x\in V_\lambda$ homogeneous and so 
we can write
\begin{multline*}
x\cdot y=q^{\frac12(\lambda+\mu,|\mathbf j(x)|+|\mathbf j(y)|)} \mathbf j(x\cdot y)v_{\lambda+\mu}=q^{\frac12(\lambda,|\mathbf j(x)|)+\frac12(\mu,|\mathbf j(y)|)}
\mathbf j(x)v_\lambda \mathbf j(y)v_\mu
\\=q^{\frac12(\lambda,|\mathbf j(x)|)+\frac12(\mu,|\mathbf j(y)|)+(\lambda,|\mathbf j(y)|)}
\mathbf j(x)\cdot\mathbf j(y)v_{\lambda+\mu}.
\end{multline*}
The assertion is now immediate.
\end{proof}
Given a $U_q(\lie g)$-module $M$, denote by $M^{int}$ the set of all $m\in M$ such that $U_q(\lie g)(m)\in\mathscr O^{int}_q(\lie g)$. The following is well-known
and in fact is easy to check.
\begin{lemma}\label{lem:int elements}
The assignment $M\mapsto M^{int}$ for every $U_q(\lie g)$-module~$M$ and $f\mapsto f$ for any morphism of~$U_q(\lie g)$-modules
defines an additive submonoidal functor from the tensor category of~$U_q(\lie g)$-modules to~$\mathscr O^{int}_q(\lie g)$, that is
$M^{int}\tensor N^{int}\subset (M\tensor N)^{int}$ for any $U_q(\lie g)$-modules $M$, $N$. 
In particular, if $M$ is an algebra object in the category of~$U_q(\lie g)$-modules then $M^{int}$ is its $U_q(\lie g)$-module subalgebra
and an algebra object in the category~$\mathscr O^{int}_q(\lie g)$.
\end{lemma}

\begin{proposition}\label{prop:V lambda in B}
For any $\lambda\in P^+$ the $U_q(\lie g)$-submodule of~$\mathcal A_q(\lie g)v_\lambda$ generated by~$v_\lambda$ is 
naturally isomorphic to~$V_\lambda$ and coincides with $(\mathcal A_q(\lie g)v_\lambda)^{int}$.
\end{proposition}
\begin{proof}
Given $M\in\mathscr O_q(\lie g)$, define $M^\vee=\bigoplus_{\beta\in P} M^\vee(\beta)$ where $M^\vee(\beta)=\Hom_\kk(M(\beta),\kk)$.
Endow $M^\vee$ with a $U_q(\lie g)$-module structure via $(u\cdot f)(m)=f(u^T(m))$, $u\in U_q(\lie g)$, $f\in M^\vee$, $m\in M$, 
where $u\mapsto u^T$, $u\in U_q(\lie g)$ is the unique anti-involution of~$U_q(\lie g)$ such 
that $E_i{}^T=F_i$ and $K_\mu{}^T=K_\mu$, $\mu\in\frac12 P$.
The following is well-known 
\begin{lemma}\label{lem:cat O}
For any symmetrizable Kac-Moody~$\lie g$, we have:
\begin{enumerate}[label={\rm(\alph*)},leftmargin=*]
 \item\label{lem:cat O.b}
 The assignments $M\to M^\vee$, $M\in\mathscr O_q(\lie g)$ define an involutive contravariant functor on~$\mathscr O_q(\lie g)$.
 \item\label{lem:cat O.c'}
 For any $M\in\mathscr O_q(\lie g)$, $(M^{int})^\vee=(M^\vee)^{int}$.

 \item\label{lem:cat O.c} For any $V\in\mathscr O^{int}_q(\lie g)$, $V^\vee$ is naturally isomorphic to~$V$.
\end{enumerate}
\end{lemma}
We need the following well-known fact which essentially coincides with~\cite{BG-schu}*{Lemma~2.10}.
\begin{lemma}\label{lem:pairing}
There exists a unique non-degenerate pairing $\fgfrm{\cdot}{\cdot}:U^-_q(\lie g)\tensor \mathcal A_q(\lie g)\to\kk$ such that 
$\fgfrm{F_i}{x_j}=\delta_{i,j}$, $i,j\in I$,
$\fgfrm{u F_i}{x}=\fgfrm{u}{E_i^*(x)}$, $\fgfrm{F_i u}{x}=\fgfrm{u}{E_i(x)}$, 
$u\in U_q^-(\lie g)$, $x\in \mathcal A_q(\lie g)$ and 
$\fgfrm{u}{x}=0$ for $u\in U_q^-(\lie g)$, $x\in \mathcal A_q(\lie g)$ homogeneous unless $\deg u=|x|$.
\end{lemma}

Denote $M_\lambda$, $\lambda\in P$, the Verma module with highest weight~$\lambda$ (see e.g.~\cite{Lus}*{\S3.4.5}).
For every~$\lambda\in P^+$ we fix~$m_\lambda\in M_\lambda(\lambda)\setminus\{0\}$. Since~$M_\lambda$ is free as a $U_q^-(\lie g)$-module, 
every element of~$M_\lambda$ can be written, uniquely, as $u m_\lambda$ for some~$u\in U_q^-(\lie g)$.
Let $\mathcal M_q(\lie g)=\bigoplus_{\lambda\in P^+} M_\lambda$. 
Define $\lla \cdot,\cdot \rra:\mathcal M_q(\lie g)\tensor \mathcal B_q(\lie g)\to \kk$
by $\lla u(m_\lambda),xv_\mu\rra=\delta_{\lambda,\mu} \fgfrm{u}{\mathbf j(x)}$, $u\in U_q^-(\lie g)$, $x\in \mathcal A_q(\lie g)$, $\lambda,\mu\in P^+$.
It is immediate from the definition that $\lla \mathcal M_q(\lie g)(\beta),\mathcal B_q(\lie g)(\beta')\rra=0$, $\beta,\beta'\in P$, unless $\beta=\beta'$.

The following Lemma seems to be well-known. We provide a proof for reader's convenience.
\begin{lemma}\label{lem:dual Verma}
The pairing $\lla\cdot,\cdot\rra$ is non-degenerate and contragredient, that is 
\begin{equation}\label{eq:contragredient}\lla u'(m),b\rra=\lla m,u'{}^T(b)\rra, \qquad u'\in U_q(\lie g),\,
m\in \mathcal M_q(\lie g),\,b\in \mathcal B_q(\lie g). 
\end{equation}
In particular, $\mathcal A_q(\lie g)v_\lambda$, $\lambda\in P^+$ naturally identifies with~$M_\lambda^\vee$.
\end{lemma}
\begin{proof}
The pairing $\lla\cdot,\cdot\rra$ is non-degenerate as a direct sum of non-degenerate (in view of Lemma~\ref{lem:pairing}) pairings $M_\lambda\tensor \mathcal A_q(\lie g)v_\lambda
\to \kk$. To prove that it is contragredient, it suffices to prove~\eqref{eq:contragredient} for $u'\in \{ K_\mu,E_i,F_i\}$, $\mu\in\frac12P$, $i\in I$.
for all $i\in I$. Moreover, we may assume without loss of generality that $m=u(m_\lambda)$ and~$b=xv_\lambda$ with $u\in U_q^-(\lie g)$, $x\in\mathcal A_q(\lie g)$
homogeneous. We have 
\begin{multline*}
\lla K_\mu(u(m_\lambda)),x v_\lambda\rra=q^{(\mu,\lambda+\deg u)}\delta_{\deg u,|x|}\lla u(m_\lambda),x v_\lambda\rra
\\=\lla u(m_\lambda),K_\mu(xv_\lambda\rra=\lla u(m_\lambda),K_\mu^T(xv_\lambda\rra.
\end{multline*}
Furthermore, by Lemmata~\ref{lem:joseph-map}\ref{lem:joseph-map.a} and~\ref{lem:pairing} we obtain 
\begin{multline*}
\lla F_i(u(m_\lambda)),xv_\lambda\rra=\lla (F_i u)(m_\lambda),xv_\lambda\rra
=\fgfrm{F_iu}{\mathbf j(x)}=\fgfrm{u}{E_i(\mathbf j(x)}\\
=\fgfrm{u}{\mathbf j(E_i(x))}
=\lla u(m_\lambda),E_i(xv_\lambda)\rra=\lla u(m_\lambda),F_i^T(xv_\lambda)\rra.
\end{multline*}
In particular, we proved~\eqref{eq:contragredient} for $u'\in\{ K_\mu,F_i\}$, $\mu\in\frac12 P$, $i\in I$ for all $m\in\mathcal M_q(\lie g)$ and 
$b\in\mathcal B_q(\lie g)$.

It remains to prove that 
$$
\lla E_i(m),b\rra=\lla m,F_i(b)\rra,
$$
for all~$m\in M_\lambda$ and for all $b\in\mathcal A_q(\lie g)v_\lambda$ homogeneous.
We argue by induction on~$\rho^\vee(\lambda-\beta)$ where $m\in M_\lambda(\beta)$. 
If $\beta=\lambda$ then $E_i(m)=0$ while $F_i(b)=b'$ with $|b'|=|b|-\alpha_i$. Since $|b|\in \lambda-Q^+$, $|b'|\not=\beta$ and so $\lla m,b'\rra=0$.
For the inductive step, it suffices to assume that $m=F_j(m')$ where $m'$ is homogeneous.
We have 
\begin{multline*}
\lla E_i(F_j(m')),b\rra=\lla [E_i,F_j](m'),b\rra
+\lla (F_jE_i)(m')),b\rra
\\
=\lla m', [E_j,F_i](b)\rra+\lla m',F_iE_j(b)\rra=\lla m',E_j(F_i(b))\rra
=\lla F_j(m'), F_i(b)\rra,
\end{multline*}
where we used~\eqref{eq:contragredient} for $u'=F_j$ and $u'=[E_i,F_j]=\delta_{i,j}(q_i-q_i^{-1})^{-1}(K_{\alpha_i}-K_{-\alpha_i})=[E_j,F_i]$.

The second assertion of the Lemma is immediate from the first.
\end{proof}

By~\cite{Lus}*{Proposition~3.5.6}, $M_\lambda$ has a unique integrable quotient isomorphic to~$V_\lambda$. 
Applying~${}^{int\vee}$ to the surjection $M_\lambda\to V_\lambda$ we obtain, by Lemmata~\ref{lem:cat O}
and Lemma~\ref{lem:dual Verma}, the desired isomorphism $V_\lambda\cong (M_\lambda^\vee)^{int}=(\mathcal A_q(\lie g)v_\lambda)^{int}$
such that $v_\lambda\mapsto v_\lambda$.
\end{proof}
In view of Proposition~\ref{prop:V lambda in B}, from now on we identify~$V_\lambda$, $\lambda\in P^+$, with the $U_q(\lie g)$-submodule of~$\mathcal A_q(\lie g)v_\lambda$
generated by~$v_\lambda$.
\begin{lemma}\label{lem:mult V lambda}
For any $\lambda,\mu\in P^+$ we have $V_\lambda\cdot V_\mu=V_{\lambda+\mu}$ in~$\mathcal B_q(\lie g)$.
\end{lemma}
\begin{proof}
It is immediate from the definition of~$\mathcal B_q(\lie g)$ 
and Corollary~\ref{cor:mod alg decomp} that $V_\lambda\cdot V_\mu\subset \mathcal A_q(\lie g)v_{\lambda+\mu}$.
Furthermore, since $V_\lambda\cdot V_\mu$ is the image of $V_\lambda\tensor V_\mu$ which is integrable, by Proposition~\ref{prop:V lambda in B} we have 
$V_\lambda\cdot V_\mu\subset (\mathcal A_q(\lie g)v_{\lambda+\mu})^{int}=V_{\lambda+\mu}$
and the latter is a simple $U_q(\lie g)$-module, $V_\lambda\cdot V_\mu=V_{\lambda+\mu}$.
\end{proof}
Denote $\mathcal C_q(\lie g)=\mathcal B_q(\lie g)^{int}$. The following is an immediate corollary of Proposition~\ref{prop:V lambda in B} and
Lemma~\ref{lem:mult V lambda}.
\begin{corollary}\label{cor:Gelfand-Kirillov model}
$\mathcal C_q(\lie g)$ decomposes as $\mathcal C_q(\lie g)=\sum_{\lambda\in P^+} V_\lambda$ as a $U_q(\lie g)$-module algebra.
\end{corollary}

\begin{proposition}\label{lem:joseph-map-kernels}
For any symmetrizable Kac-Moody~$\lie g$ and $\lambda\in P^+$ we have:
\begin{equation}\label{eq:im V lambda}
\mathbf j(V_\lambda)=\bigcap_{i\in I} \ker {E_i^*}^{\lambda(\alpha_i^\vee)+1}.
\end{equation}
\end{proposition}
\begin{proof}
Note that Lemma~\ref{lem:pairing} yields an isomorphism of $\kk$-vector spaces 
$\xi:\mathcal A_q(\lie g)\to U^-_q(\lie g)^\vee:=\bigoplus_{\gamma\in Q^+} \Hom_\kk(U^-_q(\lie g)(-\gamma),\kk)$
defined by $\xi(x)(u)=\fgfrm{u}{x}$, $x\in\mathcal A_q(\lie g)$, $u\in U^-_q(\lie g)$.
Define $\phi_\lambda^\vee:M_\lambda^\vee\to U_q^-(\lie g)^\vee$ by $\phi_\lambda^\vee(f)(u)=f(u(m_\lambda))$
for all $f\in M_\lambda^\vee$, $u\in U_q^-(\lie g)$.

Define an action of~$U_q^+(\lie g)$ on~$U_q^-(\lie g)^\vee$ by $(u_+\cdot f)(u_-):=f(u_+^T u_-)$, $u_\pm\in U_q^\pm(\lie g)$.

\begin{lemma}\label{lem:diagram}
For any $\lambda\in P^+$, $\phi_\lambda^\vee$ is an isomorphism of~$U_q^+(\lie g)$-modules. Moreover,
the following diagram in the category of~$U_q^+(\lie g)$-modules commutes
 \begin{equation}\label{eq:diagram  J}
 \vcenter{\xymatrix{\mathcal B_q(\lie g)\ar[r]^{\mathbf j} &\mathcal A_q(\lie g)\ar[d]^\xi\\
 M^\vee_\lambda\ar[u]\ar[r]^{\phi^\vee_\lambda}& U_q^-(\lie g)^\vee}}
 \end{equation}
 where the left vertical arrow is obtained by the identification $M_\lambda^\vee\cong \mathcal A_q(\lie g)v_\lambda$
 from Proposition~\ref{prop:V lambda in B}.

 \end{lemma}
 Let $\mathcal J_\lambda$, $\lambda\in P^+$, be the kernel of the canonical projection of~$M_\lambda$ on~$V_\lambda$.
 It is well-known (see e.g.~\cite{Lus}*{Proposition~3.5.6}) that 
$\mathcal J_\lambda=\sum_{i\in I} U^-_q(\lie g) F_i^{\lambda(\alpha_i^\vee)+1}(m_\lambda)$. 
Applying~${}^\vee$ to the projection~$M_\lambda\to V_\lambda$ and using that $V_\lambda\cong V_\lambda^\vee$ 
we obtain an embedding $V_\lambda\to M_\lambda^\vee$.
Note that 
$$
\phi_\lambda^\vee(V_\lambda)=\{ f\in U_q^-(\lie g)^\vee\,:\, f\Big(\sum_{i\in I} U_q^-(\lie g)F_i^{\lambda(\alpha_i^\vee)+1}\Big)=\{0\}\}.
$$
Therefore, 
$\phi_\lambda^\vee(V_\lambda)=\bigcap_{i\in I}\mathcal K_i$ where $\mathcal K_i=\{ f\in U_q^-(\lie g)^\vee\,:\,f(U_q^-(\lie g)F_i^{\lambda(\alpha_i^\vee)+1})=0\}$.
By Lemma~\ref{lem:pairing}, $\xi^{-1}(\mathcal K_i)=\ker E_i^*{}^{\lambda(\alpha_i^\vee)+1}$.  Using~\eqref{eq:diagram  J} we obtain $\mathbf j(V_\lambda)
=\bigcap_{i\in I}\xi^{-1}(\mathcal K_i)=\bigcap_{i\in I}\ker E_i^*{}^{\lambda(\alpha_i^\vee)+1}$.
\end{proof}

\subsection{Realization of \texorpdfstring{$\sigma^I$}{eta I} via quantum twist}
Let $v_{w\lambda}=F_{w,\lambda}(v_\lambda)$, $\lambda\in P^+$ where we use the notation from~\S\ref{subs:spec mon} (see also~\S\ref{subs:extrem vecs}).
This notation agrees with that in~\cite{BR}*{(6.3)}.
We need the following 
\begin{lemma}\label{lem:extrem vecs Ore}
Let $\lie g$ be a symmetrizable Kac-Moody algebra. Then for any 
$w,w'\in W$ and $\lambda,\mu\in P^+$ we have 
\begin{enumerate}[label={\rm(\alph*)},leftmargin=*]
\item\label{lem:extrem vecs Ore.a}
$v_{w\lambda}\cdot v_{w\mu}=v_{w(\lambda+\mu)}$. In particular, for any $w\in W$, the assignments
$v_\lambda\mapsto v_{w\lambda}$, $\lambda\in P^+$, define a homomorphism of monoids $g_w:\Gamma\to \mathcal B_q(\lie g)$; 
\item \label{lem:extrem vecs Ore.a'}
if $\ell(w'w)=\ell(w)+\ell(w')$ then we have 
$$v_{w'\mu}\cdot v_{w' w\lambda}=q^{(w\lambda-\lambda,\mu)}v_{w' w\lambda}\cdot v_{w'\mu};
$$
\item\label{lem:extrem vecs Ore.b} if $\ell(s_i w)=\ell(w)-1$, $i\in I$  then
$v_{w\lambda} x_i= q^{(w\lambda,\alpha_i)}x_i v_{w\lambda}$ for all $\lambda\in P^+$.
\item\label{lem:extrem vecs Ore.c} If $\lie g$ is finite dimensional 
then $v_{w_\circ\lambda} x=
q^{-(w_\circ\lambda,|x|)} x v_{w_\circ\lambda}$ for all $x\in \mathcal A_q(\lie g)$ homogeneous. 
\end{enumerate}
\end{lemma}
\begin{proof}
To prove~\ref{lem:extrem vecs Ore.a} we use induction on~$\ell(w)$, the induction base being trivial. For the inductive step, suppose that 
$\ell(s_i w)=\ell(w)+1$. Then by Lemma~\ref{lem:EF lambda w}\ref{lem:EF lambda w.b} and the induction hypothesis
$$
v_{s_iw(\lambda+\mu)}=F_{s_i w,\lambda+\mu}(v_{\lambda+\mu})=F_i^{(w(\lambda+\mu)(\alpha_i^\vee))}(v_{w(\lambda+\mu)})=F_i^{(w(\lambda+\mu)(\alpha_i^\vee))}(v_{w\lambda}\cdot v_{w\mu}).
$$
Using~\eqref{eq:leibnitz rule} and observing that $F_i^{(r)}(v_{w\lambda})F_i^{(s)}(v_{w\mu})=0$ if $r>w\lambda(\alpha_i^\vee)$ or 
$s>w\mu(\alpha_i^\vee)$
we obtain by Lemma~\ref{lem:EF lambda w}\ref{lem:EF lambda w.b}
\begin{multline*}
F_{s_i w(\lambda+\mu)}=\sum_{r+t=w(\lambda+\mu)(\alpha_i^\vee)} q^{\frac12(rw\mu-tw\lambda,\alpha_i)} F_i^{(r)}(v_{w\lambda})\cdot F_i^{(t)}(v_{w\mu})\\=
F_i^{(w\lambda(\alpha_i^\vee))}(v_{w\lambda})\cdot F_i^{(w\mu(\alpha_i^\vee))}(v_{w\mu})=v_{s_i w\lambda}\cdot v_{s_i w\mu}.
\end{multline*}
Part~\ref{lem:extrem vecs Ore.a'} was established in~\cite{BR}*{Lemma~6.4}. 
To prove part~\ref{lem:extrem vecs Ore.b}, 
note that if $\ell(s_i w)=\ell(w)-1$ then $F_i(v_{w\lambda})=0$. Then $x_i K_{\frac12\alpha_i}(v_{w\lambda})-K_{-\frac12\alpha_i}(v_{w\lambda})x_i=0$,
whence $x_i v_{w\lambda}=q^{-(\alpha_i,w\lambda)} v_{w\lambda} x_i=q^{(w\lambda,|x_i|)} v_{w\lambda}x_i$.
In particular, applying part~\ref{lem:extrem vecs Ore.b} with~$w=w_\circ$
we obtain, using an obvious induction on~$-\rho^\vee(|x|)$, 
$$
v_{w_\circ\lambda} x
=q^{-(w_\circ\lambda,|x|)} x v_{w_\circ\lambda},
$$
which yields part~\ref{lem:extrem vecs Ore.c}.
\end{proof}

Following~\cite{BR}*{\S6.1}, define {\em generalized quantum minors} $\Delta_{w\lambda}\in \mathcal A_q(\lie g)$, $w\in W$, $\lambda\in P^+$ by 
$
\Delta_{w\lambda}:=
\mathbf j(v_{w\lambda})$.
In particular,
\begin{equation}\label{eq:extrem minor}
v_{w\lambda}=q^{\frac12(w\lambda-\lambda,\lambda)}\Delta_{w\lambda} v_\lambda.
\end{equation}

We list some properties of generalized quantum minors which will be used in the sequel.
\begin{lemma}\label{lem:Delta facts}
Let $w,w'\in W$, $\lambda,\mu\in P^+$. Then 
\begin{enumerate}[label={\rm(\alph*)},leftmargin=*]
 \item\label{lem:Delta facts.a}
 $\Delta_{w\lambda}\cdot \Delta_{w\mu}=q^{\frac12 (w\mu-w^{-1}\mu,\lambda)}\Delta_{w(\lambda+\mu)}$;
\item\label{lem:Delta facts.b}
$\Delta_{w\mu}\cdot \Delta_{w w'\lambda}
=q^{(w\mu-\mu,ww'\lambda+\lambda)}\Delta_{w w'\lambda}\cdot \Delta_{w\mu}$;
\item\label{lem:Delta facts.c}
If $\lie g$ is finite-dimensional reductive then $\Delta_{w_\circ\lambda}\cdot\Delta_{w_\circ\mu}=\Delta_{w_\circ(\lambda+\mu)}$ and 
$
\Delta_{w_\circ\lambda}x=q^{-(w_\circ\lambda+\lambda,|x|)} x\Delta_{w_\circ\lambda}$ for any $x\in \mathcal A_q(\lie g)$ homogeneous.
\end{enumerate}
\end{lemma}
\begin{proof}
Parts~\ref{lem:Delta facts.a} and~\ref{lem:Delta facts.b}
follow immediately from Lemma~\ref{lem:extrem vecs Ore}\ref{lem:extrem vecs Ore.a} and~\ref{lem:extrem vecs Ore.a'}, respectively,
by applying Lemma~\ref{lem:joseph-map}\ref{lem:joseph-map.b}. The first assertion of part~\ref{lem:Delta facts.c} is a special 
case of~\ref{lem:Delta facts.a}. Finally, using~\eqref{eq:comm rel B}, \eqref{eq:extrem minor} and Lemma~\ref{lem:extrem vecs Ore}\ref{lem:extrem vecs Ore.c}
we can write 
$$
q^{(\lambda,|x|)}\Delta_{w_\circ\lambda}x v_\lambda=\Delta_{w_\circ\lambda}v_\lambda x=
q^{-(w_\circ\lambda,|x|)}x \Delta_{w_\circ\lambda} v_\lambda. 
$$
It remains to apply~$\mathbf j$ and use the fact that~$\mathbf j|_{\mathcal A_q(\lie g)v_\lambda}$ is injective.
\end{proof}

Let $\mathcal S_{w}=\{ \Delta_{w\lambda}\,:\, \lambda\in P^+\}$. It follows from Lemma~\ref{lem:Delta facts}\ref{lem:Delta facts.c} that 
$\mathcal S_{w_\circ}$ is an abelian submonoid of~$\mathcal A_q(\lie g)$ and in fact is an Ore submonoid with
$\Sigma_{\Delta_{w_\circ\lambda}}(x_i)=q^{(\lambda,\alpha_{i^\star}-\alpha_i)} x_i$ for $\lambda\in P^+$, $i\in I$.

Define $\widehat{\mathcal B}_q(\lie g):=\mathcal B_q(\lie g)[\mathcal S_{w_\circ}^{-1}]$ and
let $\widehat{\mathcal A}_q(\lie g)$ be the subalgebra of~$\widehat{\mathcal B}_q(\lie g)$ generated by $\mathcal A_q(\lie g)$, as a subalgebra, and the $\Delta_{w_\circ\lambda}^{-1}$,
$\lambda\in P^+$. Clearly, $\widehat{\mathcal A}_q(\lie g)$ is isomorphic to~$\mathcal A_q(\lie g)[\mathcal S_{w_\circ}^{-1}]$.
The following is the main result of Section~\ref{sect:can bas}.
\begin{theorem}\label{thm:twist}
Let $\lie g$ be finite dimensional. Then 
\begin{enumerate}[label={\rm(\alph*)},leftmargin=*]
 \item\label{thm:twist.a}
the assignments $x_i\mapsto q_i^{\frac12(\delta_{i,i^\star}-1)} E_{i^\star}(\Delta_{w_\circ\omega_i})\Delta_{w_\circ \omega_{i}}^{-1}$, 
$v_\lambda\mapsto v_{w_\circ\lambda}$, $\lambda\in P^+$,
define an injective algebra homomorphism $\widehat\sigma:\mathcal B_q(\lie g)\to \widehat{\mathcal B}_q(\lie g)^{op}$;

\item\label{thm:twist.c}
$\widehat\sigma(V_\lambda)=V_\lambda$ and $\widehat\sigma|_{V_\lambda}=\sigma^I_{V_\lambda}$. In particular,
the restriction of~$\widehat\sigma$ to~$\mathcal C_q(\lie g)$
is an anti-involution on $\mathcal C_q(\lie g)$.
\end{enumerate}
\end{theorem}
\begin{proof}
The first step is to construct a homomorphism of algebras~$\sigma_0:\mathcal A_q(\lie g)\to \widehat{\mathcal A}_q(\lie g)^{op}$.
\begin{proposition}\label{prop:Kimura}
The assignments $$x_i\mapsto q_i^{-\frac12(1-\delta_{i,i^\star})} 
E_{i^\star}(\Delta_{w_\circ\omega_i})\Delta_{w_\circ\omega_i}^{-1}=q_i^{\frac12(1-\delta_{i,i^\star})}\Delta_{w_\circ\omega_i}^{-1}E_{i^\star}(\Delta_{w_\circ\omega_i}),
\quad i\in I,
$$
define  
a homomorphism $\sigma_0:\mathcal A_q(\lie g)\to\widehat{\mathcal A}_q(\lie g)^{op}$ such that $\sigma_0(\mathcal A_q(\lie g)(-\gamma))
\subset \widehat{\mathcal A}_q(\lie g)(-w_\circ\gamma)$, $\gamma\in Q^+$.
\end{proposition}
\begin{proof}
Let $\delta$ be the unique automorphism of~$\mathcal A_q(\lie g)$ defined by $\delta(x_i)=x_{i^\star}$, $i\in I$. Then $\kappa:\mathcal A_q(\lie g)\to \mathcal A_q(\lie g)$ is 
the anti-automorphism defined by $\kappa(x)=\delta(x^*)=(\delta(x))^*$. We need the following
\begin{lemma}\label{lem:kappa Delta}
For any~$\lambda\in P^+$, 
$\kappa(\Delta_{w_\circ\lambda})=\epsilon_\lambda\Delta_{w_\circ\lambda}$ where $\epsilon_\lambda\in\{\pm 1\}$.
\end{lemma}
\begin{remark}
Later we will show that~$\epsilon_\lambda=1$. However, for the purposes of proving Proposition~\ref{prop:Kimura} this is irrelevant.
\end{remark}

\begin{proof}
Define 
$$
\mathcal A_q(\lie g)^\lambda:=\{ x\in \mathcal A_q(\lie g)_{w_\circ\lambda-\lambda}\,:\,
(E_i^*)^{\lambda(\alpha_i^\vee)+1}(x)=0,\,\forall\,i\in I\}.
$$
It follows from the definition and Lemma~\ref{lem:joseph-map} that 
$$
\mathcal A_q(\lie g)^\lambda=\mathbf j(V_\lambda(w_\circ\lambda))=\kk\Delta_{w_\circ\lambda}.
$$
In particular, $E_i^{1-w_\circ\lambda(\alpha_i^\vee)}(\mathcal A_q(\lie g)^\lambda)=0$ for all~$i\in I$. Since $\kappa(E_i(x))=E_{i^\star}^*(\kappa(x))$,
it follows that $\kappa(\Delta_{w_\circ\lambda})\in \mathcal A_q(\lie g)^\lambda$ and so is a multiple of~$\Delta_{w_\circ\lambda}$.
Since~$\kappa$ is an involution, the assertion follows.
\end{proof}
It follows from Lemmata~\ref{lem:lifts auto} and~\ref{lem:kappa Delta} that~$\kappa$ lifts to an anti-involution~$\widehat\kappa$ on
$\widehat{\mathcal A}_q(\lie g)$.
By~\cite{KiO}*{Theorem~5.4}, for any $\mathbf c=(c_i)_{i\in I}\in(\kk^\times)^I$ the assignments
$$x_i\mapsto
c_i E_{i}^*(\Delta_{w_\circ\omega_i})\Delta_{w_\circ\omega_i}^{-1}=c_i q_i^{\delta_{i,i^\star}-1}\Delta_{w_\circ\omega_i}^{-1}E_{i}^*(\Delta_{w_\circ\omega_i}),
\quad i\in I,
$$
define a homomorphism of algebras $\zeta_{\mathbf c}:\mathcal A_q(\lie g)\to \widehat{\mathcal A}_q(\lie g)^{op}$. Let $\mathbf c_0=(q_i^{\frac12(1-\delta_{i,i^\star})})_{i\in I}$
and 
set~$\sigma_0:=\widehat\kappa\circ\zeta_{\mathbf c_0}$. 
Since~$\widehat\kappa$ is an anti-involution, we have 
$$
\sigma_0(x_i)=q_i^{\frac12(1-\delta_{i,i^\star})} 
(\kappa(\Delta_{w_\circ\omega_i}))^{-1}\kappa(E_{i}^*(\Delta_{w_\circ\omega_{i}}))
=q_i^{\frac12(1-\delta_{i,i^\star})} 
\Delta_{w_\circ\omega_i}^{-1}E_{i^\star}(\Delta_{w_\circ\omega_{i}}).
$$
Thus, $\sigma_0$ is the desired homomorphism $\mathcal A_q(\lie g)\to \widehat{\mathcal A}_q(\lie g)^{op}$.
Since $|\sigma_0(x_i)|=\alpha_{i^\star}=-w_\circ\alpha_i$, it follows that $|\sigma_0(x)|=w_\circ|x|$ for all $x\in \mathcal A_q(\lie g)$ homogeneous.
\end{proof}

Now we have all necessary ingredients to prove Theorem~\ref{thm:twist}\ref{thm:twist.a}.
We apply Lemma~\ref{lem:lift homs} with $R=\mathcal A_q(\lie g)$,
$\widehat R=\widehat{\mathcal B}_q(\lie g)^{op}$, $f=\sigma_0$ and $g=g_{w_\circ}$ viewed as a homomorphism 
$\Gamma\to \widehat R$ since~$\Gamma$ is abelian. Take~$x\in \mathcal A_q(\lie g)$ homogeneous. Then the following holds 
in~$\widehat{\mathcal B}_q(\lie g)$
$$
g_{w_\circ}(v_\lambda)\sigma_0(v_\lambda\triangleright x)=
q^{(\lambda,|x|)} v_{w_\circ\lambda}\sigma_0(x)=q^{(\lambda,|x|)-(w_\circ\lambda,w_\circ|x|)}\sigma_0(x)v_{w_\circ\lambda}
=\sigma_0(x)g_{w_\circ}(v_\lambda),
$$
which is~\eqref{eq:cross prod cond} in~$\widehat R$. Then by Lemma~\ref{lem:lift homs}, $\widehat\sigma:\mathcal B_q(\lie g)\to \widehat{\mathcal B}_q(\lie g)^{op}$,
$v_\lambda x\mapsto \sigma_0(x)v_{w_\circ\lambda}$, $x\in \mathcal A_q(\lie g)$, $\lambda\in P^+$, is a well-defined homomorphism of algebras.
Part~\ref{thm:twist.a} of Theorem~\ref{thm:twist} is proven.

Note that the $K_\lambda$, $\lambda\in\frac12 P$, satisfy the assumptions of Lemma~\ref{lem:lifts auto} and so can be 
lifted to automorphisms $\widehat K_{\lambda}$ of~$\widehat{\mathcal B}_q(\lie g)$. 
Define 
\begin{equation}\label{eq:E i as commutator}
\widehat F_i(x)=\frac{ x_i \widehat K_{\frac12\alpha_i}(x)-\widehat K_{-\frac12\alpha_i}(x) x_i}{q_i-q_i^{-1}},\qquad
\widehat E_i(x)=\frac{ \widehat K_{-\frac12\alpha_i}(x) z_i-z_i\widehat K_{\frac12\alpha_i}(x)}{q_i-q_i^{-1}},
\end{equation}
where $z_i=\widehat\sigma(x_{i^\star})=q_i^{\frac12(\delta_{i,i^\star}-1)} E_i(\Delta_{w_\circ\omega_{i^\star}})\Delta_{w_\circ\omega_{i^\star}}^{-1}$.

\begin{proposition}\label{prop:comm with hat eta}
We have for all~$\lambda\in\frac12 P$, $i\in I$:
\begin{enumerate}[label={\rm(\alph*)},leftmargin=*]
\item \label{prop:comm with hat eta.a}
$\widehat K_\lambda|_{\mathcal B_q(\lie g)}=K_\lambda$, $\widehat F_i|_{\mathcal B_q(\lie g)}=E_i$ and $\widehat E_i|_{\mathcal B_q(\lie g)}=E_i$;
 \item \label{prop:comm with hat eta.b}
$\widehat K_\lambda\circ \widehat\sigma=\widehat \sigma\circ K_{w_\circ\lambda}$, $\widehat F_i\circ\widehat \sigma=\widehat\sigma\circ E_{i^\star}$
and $\widehat E_i\circ \widehat\sigma=
\widehat\sigma\circ F_{i^\star}$.
\end{enumerate}
\end{proposition}
\begin{proof}
 The first and the second assertions in part~\ref{prop:comm with hat eta.a} are obvious.
 Furthermore, since $\widehat F_i=D^-_{(q_i-q_i^{-1})^{-1}x_i}$ and $\widehat E_i=D^+_{(q_i-q_i^{-1})^{-1}z_i}$ with $L_\pm =\widehat K_{\pm\frac12\alpha_i}$
 in the notation of~\eqref{eq:D +-}, we 
 immediately obtain the following 
 \begin{lemma}\label{lem:hat F hat E diff}
 $\widehat F_i$ and $\widehat E_i$, $i\in I$ are $\widehat K_{\frac12\alpha_i}$-derivations of~$\widehat{\mathcal B}_q(\lie g)$. 
 \end{lemma}
Thus, by Lemma~\ref{lem:diffs restrs} the last assertion in part~\ref{prop:comm with hat eta.a} is equivalent to
$$
\widehat E_i(v_\lambda)=0,\qquad \widehat E_i(x_j)=\delta_{i,j},\qquad \lambda\in P^+,\,i,j\in I.
$$
Since $|z_i|=\alpha_i$ and $z_i\in\widehat{\mathcal A}_q(\lie g)$, we have 
$$
K_{-\frac12\alpha_i}(v_\lambda)z_i-z_i K_{\frac12\alpha_i}(v_\lambda)=q^{-\frac12(\alpha_i,\lambda)} v_\lambda z_i-q^{\frac12(\lambda,\alpha_i)}z_i v_\lambda 
=0.
$$
Thus, $\widehat E_i(v_\lambda)=0$ for all~$\lambda\in P^+$. We need the following
\begin{lemma}\label{lem:commutator delta spec}
The following identity holds in~$\mathcal A_q(\lie g)$ for all $i,j\in I$
$$
q_i^{\frac12(\delta_{i,i^\star}-1)}(q^{\frac12(\alpha_i,\alpha_j)}x_j E_i(\Delta_{w_\circ\omega_{i^\star}})-q^{-\frac12(\alpha_i,\alpha_j)} q_j^{\delta_{i,j}-\delta_{i^\star,j}}
E_i(\Delta_{w_\circ\omega_{i^\star}}) x_j=\delta_{i,j}(q_i-q_i^{-1})\Delta_{w_\circ\omega_{i^\star}}.
$$
\end{lemma}
\begin{proof}
This follows by a straightforward computation by applying~$E_i$ to the identity
$$x_j\Delta_{w_\circ\omega_{i^\star}}=
q_j^{\delta_{i,j}-\delta_{i^\star,j}}\Delta_{w_\circ\omega_{i^\star}}x_j$$ which is a special case of Lemma~\ref{lem:Delta facts}\ref{lem:Delta facts.c}.
\end{proof}
Since $\Delta_{w_\circ \omega_{i^\star}}^{-1} x_j \Delta_{w_\circ\omega_{i^\star}}=q^{-(w_\circ\omega_{i^\star}+\omega_{i^\star},\alpha_j)}x_j
=q_j^{\delta_{i,j}-\delta_{i^\star,j}}x_j$, we can write 
\begin{multline*}
(q_i-q_i^{-1})\widehat E_i(x_j)\Delta_{w_\circ \omega_{i^\star}}\\=q_i^{\frac12(\delta_{i,i^\star}-1)}
(q^{\frac12(\alpha_i,\alpha_j)}x_j E_i(\Delta_{w_\circ\omega_{i^\star}})-q^{-\frac12(\alpha_i,\alpha_j)} q_j^{\delta_{i,j}-\delta_{i^\star,j}}
E_i(\Delta_{w_\circ\omega_{i^\star}}) x_j).
\end{multline*}
Using Lemma~\ref{lem:commutator delta spec} we conclude that 
$(q_i-q_i^{-1})\widehat E_i(x_j)\Delta_{w_\circ\omega_{i^\star}}=\delta_{i,j}(q_i-q_i^{-1})\Delta_{w_\circ\omega_{i^\star}}$
and so $\widehat E_i(x_j)=\delta_{i,j}$.
Part~\ref{prop:comm with hat eta.a} of Proposition~\ref{prop:comm with hat eta} is proven.

The first assertion in~Proposition~\ref{prop:comm with hat eta}\ref{prop:comm with hat eta.b}
is immediate since $|\widehat\sigma(x)|=w_\circ|x|$ for~$x\in \mathcal B_q(\lie g)$ homogeneous. 
Furthermore, by~\eqref{eq:E i as commutator} we obtain for all $x\in\mathcal B_q(\lie g)$.
\begin{multline*}
\widehat\sigma(F_i(x))=\frac{ \widehat\sigma(x_i \widehat K_{\frac12\alpha_i}(x))-\widehat\sigma(\widehat K_{-\frac12\alpha_i}(x)x_i)}{q_i-q_i^{-1}}
=\frac{ \widehat K_{-\frac12\alpha_{i^\star}}(\widehat\sigma(x)) z_{i^\star}-z_{i^\star} \widehat K_{\frac12\alpha_{i^\star}}(\widehat\sigma(x))}{q_i-q_i^{-1}}
\\=\widehat E_{i^\star}(\widehat\sigma(x)).
\end{multline*}
It remains to prove that $\widehat\sigma(E_i(x))=\widehat F_{i^\star}(\widehat\sigma(x))$ for all~$x\in\mathcal B_q(\lie g)$.
Let $D_i=\widehat\sigma\circ E_i-\widehat F_{i^\star}\circ \widehat\sigma$. Since $\widehat\sigma\circ K_{\pm\frac12\alpha_i}=
\widehat K_{\mp\frac12\alpha_{i^\star}}\circ\widehat\sigma$ it follows that~$D_i$ is a $\widehat K_{\frac12\alpha_{i^\star}}\circ\widehat\sigma$-derivation 
from~$\mathcal B_q(\lie g)$ to~$\widehat{\mathcal B}_q(\lie g)^{op}$. We have 
$$
D_i(v_\lambda)=\widehat\sigma(E_i(v_\lambda))-\widehat F_{i^\star}(v_{w_\circ\lambda})=0,
$$
By Proposition~\ref{prop:comm with hat eta}\ref{prop:comm with hat eta.a} we have  
\begin{equation}
\delta_{i,j}=\widehat E_i(x_j)=\frac{ q^{\frac12(\alpha_i,\alpha_j)}x_j z_i-q^{-\frac12(\alpha_i,\alpha_j)}z_i x_j}{q_i-q_i^{-1}}=
\frac{q_j-q_j^{-1}}{q_i-q_i^{-1}}\,\widehat F_j(z_i)
\end{equation}
and so 
$$
D_i(x_j)=\widehat\sigma(E_j(x_i))-\widehat F_{i^\star}(z_{j^\star})=\delta_{i,j}-\delta_{i^\star,j^\star}=0,
$$
Thus $D_i=0$ on generators of~$\mathcal B_q(\lie g)$. Then~$D_i=0$ by Lemma~\ref{lem:diffs restrs}. This completes the proof of 
Proposition~\ref{prop:comm with hat eta}\ref{prop:comm with hat eta.b}.
\end{proof}

To prove part~\ref{thm:twist.c}, we need to show that~$\widehat \sigma(V_\lambda)\subset V_\lambda$. The following Lemma follows
from Proposition~\ref{prop:comm with hat eta}\ref{prop:comm with hat eta.a} by an obvious induction.
\begin{lemma}\label{lem:monomials in E_i}
For any $b\in \mathcal B_q(\lie g)$, $r\ge 1$, 
$(i_1,\dots,i_r)\in I^r$, $\widehat E_{i_1}\cdots\widehat E_{i_r}(b)=E_{i_1}\cdots E_{i_r}(b)\in \mathcal B_q(\lie g)$.
In particular, for any~$v\in V_\lambda$, $\lambda\in P^+$ we have $\widehat E_{i_1}\cdots\widehat E_{i_r}(v)\in V_\lambda$.
\end{lemma}

Since $V_\lambda$ is spanned 
the $F_{i_1}\cdots F_{i_r}(v_\lambda)$, $r\ge 0$, $(i_1,\dots,i_r)\in I^r$, it suffices to show that 
$\widehat \sigma( F_{i_1}\cdots F_{i_r}(v_\lambda))\in V_\lambda$. 
We have by Proposition~\ref{prop:comm with hat eta}\ref{prop:comm with hat eta.b}
$$
\widehat \sigma( F_{i_1}\cdots F_{i_r}(v_\lambda))=\widehat E_{i_1^\star}\cdots \widehat E_{i_r^\star}(v_{w_\circ\lambda})\in V_\lambda 
$$
by Lemma~\ref{lem:monomials in E_i} applied with~$v=v_{w_\circ\lambda}$. 

Consider the operator~$\sigma^I_{V_\lambda}\circ \widehat\sigma$. Clearly, it maps~$v_\lambda$ to itself and commutes with the 
$U_q(\lie g)$-action by Proposition~\ref{prop:comm with hat eta}\ref{prop:comm with hat eta.b}. Since~$V_\lambda$ is a simple $U_q(\lie g)$-module generated by~$v_\lambda$, it follows that $\widehat\sigma|_{V_\lambda}=
(\sigma^I_{V_\lambda})^{-1}=\sigma^I_{V_\lambda}$. In particular, $\widehat\sigma$ is an involution on each~$V_\lambda$ and hence an anti-involution on~$\mathcal C_q(\lambda)$. 
\end{proof}
\begin{corollary}
Let~$\lie g$ be finite-dimensional reductive. Then 
$\sigma^I$ is an anti-involution on the algebra $\mathcal C_q(\lie g)$.
\end{corollary}

\subsection{\texorpdfstring{$\sigma^I$}{eta I} on upper global crystal basis}
Denote $\mathbf B^{up}$ the dual canonical basis in~$\mathcal A_q(\lie g)$ and denote $\mathbf B_\lambda$ the upper global crystal basis of~$V_\lambda$.
Let $\mathbf B=\bigsqcup_{\lambda\in P^+} \mathbf B_\lambda$ be the upper global crystal basis of~$\mathcal C_q(\lie g)$.

\begin{theorem}\label{thm:eta-bas}
For any finite dimensional reductive $\lie g$ we have 
$\widehat\sigma(\mathbf B)=\mathbf B$. In particular, $\sigma^I_{V_\lambda}(\mathbf B_\lambda)=\mathbf B_\lambda$.
\end{theorem}
\begin{proof}
We need the following 
\begin{lemma}[see e.g. \cite{KiO}*{Proposition~2.33}]\label{lem:jos-map-bas}
For any $\lambda\in P^+$, 
$\mathbf j(\mathbf B_\lambda)\subset \mathbf B^{up}$.
\end{lemma}
\noindent
In particular, since $v_{w_\circ\lambda}\in \mathbf B_\lambda$, it follows that $\Delta_{w_\circ\lambda}\in\mathbf B^{up}$.

Denote
$$
\widehat{\mathbf B^{up}}=\{ q^{\frac12(w_\circ\lambda+\lambda,|b|)}b\Delta_{w_\circ\lambda}^{-1}\,:\, b\in\mathbf B^{up},\,\lambda\in P^+\}
=\{ q^{-\frac12(w_\circ\lambda+\lambda,|b|)}\Delta_{w_\circ\lambda}^{-1}b\,:\, b\in\mathbf B^{up},\,\lambda\in P^+\}.
$$
We need the following 
\begin{lemma}\label{lem:eta 0 on B up}
In the notation of Proposition~\ref{prop:Kimura}, 
$\sigma_0(\mathbf B^{up})\subset \widehat{\mathbf B^{up}}$.
\end{lemma}
\begin{proof}
Note that $\kappa(\mathbf B^{up})=\mathbf B^{up}$ since it is a composition of two involutions preserving~$\mathbf B^{up}$. In particular,
$\kappa(\Delta_{w_\circ\lambda})=\Delta_{w_\circ\lambda}$ for all~$\lambda\in P^+$.

Let $b\in\mathbf B^{up}$, $\lambda\in P^+$. We have  
\begin{multline*}
\widehat\kappa( q^{\frac12(w_\circ\lambda+\lambda,|b|)}b\Delta_{w_\circ\lambda}^{-1})=
q^{\frac12(w_\circ\lambda+\lambda,|b|)}(\kappa(\Delta_{w_\circ\lambda}))^{-1} \kappa(b)
\\=
q^{-\frac12(w_\circ\lambda+\lambda,w_\circ|\kappa(b)|}\Delta_{w_\circ\lambda}^{-1}\kappa(b)
\in \widehat{\mathbf B^{up}}.
\end{multline*}
By~\cite{KiO}*{Theorem~5.4} we have $\zeta_{\mathbf c_0}(\mathbf B^{up})\subset \widehat{\mathbf B^{up}}$ where
$\zeta_{\mathbf c_0}:\mathcal A_q(\lie g)\to \widehat{\mathcal A}_q(\lie g)$ is as in the proof of Proposition~\ref{prop:Kimura}. Since~$\sigma_0=\widehat\kappa\circ\zeta_{\mathbf c_0}$,
the assertion follows.
\end{proof}

Define 
$$
\widetilde{\mathbf B}=\{ q^{\frac12(\lambda,|b|)} b v_{\lambda}\,:\, b\in\mathbf B^{up},\, \lambda\in P^+\}.
$$
It is immediate from the definition that~$\mathbf j(\widetilde{\mathbf B})=\mathbf B^{up}$ and that~$\widetilde{\mathbf B}$ is 
a basis in~$\mathcal B_q(\lie g)$. Moreover, it follows from Lemma~\ref{lem:jos-map-bas} that $\mathbf B\subset \widetilde{\mathbf B}$
and 
\begin{equation}\label{eq:intersect B A}\mathbf B=\mathcal C_q(\lie g)\cap \widetilde{\mathbf B}.
\end{equation}

Finally, define 
$$
\widehat{\mathbf B}=\{  q^{-\frac12(w_\circ\lambda,|b|)} bv_{w_\circ\lambda}\,:\, b\in \widehat{\mathbf B^{up}},\, \lambda\in P^+\}\subset \widehat{\mathcal B}_q(\lie g).
$$
\begin{proposition}\label{prop:other defn of hat B}
$\widehat{\mathbf B}=\{ q^{\frac12(\lambda,|b|)} b v_\lambda\,:\, b\in\widehat{\mathbf B^{up}},\,\lambda\in P^+\}$. In particular,
$\widehat{\mathbf B}$ is 
a basis of~$\widehat{\mathcal B}_q(\lie g)$. Finally,
$\widetilde{\mathbf B}\subset \widehat{\mathbf B}$.
\end{proposition}
\begin{proof}
We need the following
\begin{lemma}\label{lem:fact about bases}
Let $R$ be a $\kk$-algebra and let $S\subset R\setminus\{0\}$ be a commutative Ore submonoid.  
Let $B$ be a basis of~$R$ and suppose that
$\widehat B=\{ \tau_s(b)s^{-1}\,:\, b\in B,\, s\in S\}$ is a basis of~$R[S^{-1}]$ where $\tau_s:R\to R$ is some family of automorphisms 
satisfying $\tau_{ss'}=\tau_s\circ\tau_{s'}$, $\tau_s|_S=\id_S$. 
Then $\hat\tau_s(\widehat B)s^{-1}=\widehat B$ 
for any $s\in S$, where $\hat\tau_s$ is the unique lifting of~$\tau_s$ to~$R[S^{-1}]$ provided by Lemma~\ref{lem:lifts auto}.
\end{lemma}
\begin{proof}
Define $f_s:R[S^{-1}]\to R[S^{-1}]$ by $f_s(x)=\hat\tau_s(x)s^{-1}$, $x\in R[S^{-1}]$.
We claim that 
\begin{equation}\label{eq:comp-f}
f_s\circ f_{s'}=f_{ss'},\qquad s,s'\in S
\end{equation}
and 
$f_s$ is invertible with $f_s^{-1}(x)=\hat\tau_s^{-1}(x)s$. Indeed, for all $x\in R[S^{-1}]$ we have 
$$
f_s(f_{s'}(x))=\hat\tau_s(\hat \tau_{s'}(x)s'{}^{-1})
=\hat\tau_{ss'}(x)s'{}^{-1}s^{-1}=f_{ss'}(x).
$$
and also $f_s(\hat \tau_s^{-1}(x)s)=x$ and $\hat\tau_{s^{-1}}(f_s(x))=x=f_s(\hat\tau_{s^{-1}}(x))$.

We have $\widehat B=\{ f_s(b)\,:\, b\in B,\,s\in S\}$ and the assertion is equivalent to $f_s(\widehat B)=\widehat B$ for all~$s\in S$. 
Clearly, \eqref{eq:comp-f} implies that $f_s(\widehat B)\subset\widehat B$ for all~$s\in S$.
To prove the opposite inclusion, let $\hat b\in\widehat B$. 
Write 
$$
f^{-1}_s(\hat b)=\sum_{\hat b'\in \widehat B} \lambda_{\hat b'} \hat b'.
$$
Then
$$
\hat b=\sum_{\hat b'\in\widehat B} \lambda_{\hat b'}f_s(\hat b')=\sum_{\hat b''\in f_s(\widehat B)\subset \widehat B} \lambda_{f^{-1}_s(\hat b'')} \hat b'',
$$
where we used that $f_s(\widehat B)\subset \widehat B$. Since~$\widehat B$ is a basis, this implies that $\lambda_{f^{-1}_s(\hat b'')}=\delta_{\hat b,\hat b''}$
and so $\hat b\in f_s(\widehat B)$. Therefore, $\widehat B\subset f_s(\widehat B)$.
\end{proof}
Apply this lemma to $R=\mathcal A_q(\lie g)$, $S=\mathcal S_{w_\circ}$, $B=\mathbf B^{up}$ and $\widehat B=\widehat{\mathbf B^{up}}$.
By~\cite{KiO}*{Proposition~3.9},
$\widehat{\mathbf B^{up}}$ is a basis of $\widehat{\mathcal A}_q(\lie g)$\footnote{\label{ftn:1}The difference between our notation and that of~\cites{KiO,BZ2} is 
in the linear automorphism of~$\mathcal A_q(\lie g)$ defined on homogeneous elements 
$x$ by $x\mapsto q^{\frac12(|x|,|x|)-(|x|,\rho)}x$}.
We have $\tau_{\Delta_{w_\circ\lambda}}(b)=q^{\frac12(w_\circ\lambda+\lambda,|b|)}b$. Then all assumptions of Lemma~\ref{lem:fact about bases}
are satisfied. Indeed, $\tau_{\Delta_{w_\circ\lambda}}\tau_{\Delta_{w_\circ\mu}}(b)=q^{\frac12(w_\circ(\lambda+\mu)+\lambda+\mu,|b|)}b=\tau_{\Delta_{w_\circ(\lambda+\mu)}}(b)$
and $\tau_{\Delta_{w_\circ\lambda}}(\Delta_{w_\circ\mu})=q^{\frac12(w_\circ\lambda+\lambda,w_\circ\mu-\mu)}\Delta_{w_\circ\mu}=\Delta_{w_\circ\mu}$, $\lambda,\mu\in P^+$.
Thus by Lemma~\ref{lem:fact about bases} we have, for any~$\lambda\in P^+$, 
$\widehat{\mathbf B^{up}}=\hat\tau_{\Delta_{w_\circ\lambda}}^{-1}(\widehat{\mathbf B^{up}})\Delta_{w_\circ\lambda}
=\{ q^{-\frac12(w_\circ\lambda+\lambda,|b|)}b\Delta_{w_\circ\lambda}\,:\, b\in \widehat{\mathbf B^{up}}\}$.

We have $v_{w_\circ\lambda}=q^{\frac12(w_\circ\lambda-\lambda,\lambda)}\Delta_{w_\circ\lambda}v_\lambda$.
\begin{multline*}
\widehat{\mathbf B}=\{ q^{-\frac12(w_\circ\lambda,|b|)}b v_{w_\circ\lambda}\,:\, b\in \widehat{\mathbf B^{up}},\,\lambda\in P^+\}
=\{ q^{\frac12(\lambda,|b|)} \tau_{\Delta_{w_\circ\lambda}}^{-1}(b) v_{w_\circ\lambda}\,:\, b\in \widehat{\mathbf B^{up}},\,\lambda\in P^+\}\\
=\{ q^{\frac12(\lambda,|b|+w_\circ\lambda-\lambda)}\tau^{-1}_{\Delta_{w_\circ\lambda}}(b)\Delta_{w_\circ\lambda}v_\lambda\,:\, 
b\in \widehat{\mathbf B^{up}},\,\lambda\in P^+\}\\=
\{ q^{\frac12(\lambda,|b'|)}b'v_\lambda\,:\, 
b'\in \widehat{\mathbf B^{up}},\,\lambda\in P^+\}
\end{multline*}
where we denoted $b'=\tau_{\Delta_{w_\circ\lambda}}^{-1}(b)\Delta_{w_\circ\lambda}$ and observed that $|b'|=|b|+w_\circ\lambda-\lambda$. This proves the 
first assertion of Proposition~\ref{prop:other defn of hat B}.
The second and third assertions are now immediate.
\end{proof}
Now we can complete the proof of Theorem~\ref{thm:eta-bas}. It follows from Proposition~\ref{prop:other defn of hat B} and Lemma~\ref{lem:eta 0 on B up}
that 
for any $b\in \mathbf B^{up}$, $\lambda\in P^+$ we have 
$$
\widehat\sigma(q^{\frac12(\lambda,|b|)}b v_\lambda)=q^{\frac12(\lambda,|b|)}v_{w_\circ\lambda}\sigma_0(b)=q^{-\frac12(w_\circ\lambda,|\sigma_0(b)|)}\sigma_0(b)v_{w_\circ\lambda}
\in \widehat{\mathbf B}.
$$
Thus, $\widehat\sigma(\widetilde{\mathbf B})\subset \widehat{\mathbf B}$.
Since $\mathbf B\subset \widetilde{\mathbf B}$, it follows that~$\widehat\sigma(\mathbf B)\subset\widehat{\mathbf B}$.
Then by Theorem~\ref{thm:twist}\ref{thm:twist.c} we conclude that $\widehat\sigma(\mathbf B)\subset \widehat{\mathbf B}\cap\mathcal C_q(\lie g)$.
On the other hand, since $\widehat{\mathbf B}$ is linearly independent by Proposition~\ref{prop:other defn of hat B}, its intersection with~$\mathcal C_q(\lie g)$ is also linearly independent.
Since $\widetilde{\mathbf B}\subset \widehat{\mathbf B}$ by Proposition~\ref{prop:other defn of hat B}, it follows from~\eqref{eq:intersect B A} that 
$\mathbf B=\widetilde{\mathbf B}\cap \mathcal C_q(\lie g)\subset \widehat{\mathbf B}\cap\mathcal C_q(\lie g)$. But~$\mathbf B$ is a basis of~$\mathcal C_q(\lie g)$
and so $\mathbf B=\widehat{\mathbf B}\cap\mathcal C_q(\lie g)$. Thus, $\widehat\sigma(\mathbf B)\subset\mathbf B$. Since by Theorem~\ref{thm:twist}\ref{thm:twist.c}
$\widehat\sigma$ is an involution on~$\mathcal C_q(\lie g)$, $\widehat\sigma(\mathbf B)=\mathbf B$
which completes the proof of the first assertion of Theorem~\ref{thm:eta-bas}.
The second assertion is immediate from the first and  Theorem~\ref{thm:twist}\ref{thm:twist.c}.
\end{proof}

\subsection{Proof of Theorem~\ref{thm:main thm3}}\label{subs:pf main thm3}
Let $V$ be any object in~$\mathscr O^{int}_q(\lie g)$. Let $(L^{up},B^{up})$ be an upper crystal basis of~$V$ and 
let $G(B^{up})$ be the corresponding upper global crystal basis (see~\cite{Kas93}). By~\cite{Kas93}*{Theorem~3.3.1} there exists 
a direct sum decomposition $V=\sum_{j} V^j$ such that $\mathbf B^j:=G(B^{up})\cap V^j$ is a basis of~$V^j$ and
each $V^j\cong V_{\lambda_j}$, $\lambda_j\in P^+$. The latter isomorphism identifies $\mathbf B^j$ with $\mathbf B_{\lambda_j}$.
Since by Theorem~\ref{thm:prop-eta}, $\sigma^I_V$ is compatible
with direct sum decompositions, the restriction of~$\sigma^I_V$ to $V^j$ coincides with~$\sigma^I_{V_j}$ and under the above isomorphism 
it identifies with~$\sigma^I_{V_{\lambda_j}}$ and thus
preserves $\mathbf B_{\lambda_j}$ by Theorem~\ref{thm:eta-bas}.

\section{Examples}

\subsection{Thin modules}\label{subs:conj-weyl-evidence}
Let $\lambda\in P^+$. We say that $V_\lambda$ is {\em quasi-miniscule} if $V_\lambda(\beta)\not=0$ implies that $\beta\in W\lambda\cup\{0\}$.
For example, $V_{\omega_i}$, $i\in I$ for $\lie g=\lie{sl}_n$ are (quasi)-miniscule, as well as the quantum analogue of the adjoint representation
of~$\lie g$.
\begin{lemma}\label{lem:Weyl grp}
Conjecture~\ref{conj:weyl} holds for any quasi-miniscule $V=V_\lambda$.  
\end{lemma}
\begin{proof}
Let $v=v(\lambda)\in V_\lambda(\lambda)$. Then in the notation of~\eqref{eq:v_W} we have $V_\lambda=\kk\cdot [v]_W\oplus V_\lambda(0)$.
As shown in Proposition~\ref{prop:ficus action}, the action of~$\mathsf W(V)$ on the basis $[v]_W$ of $\kk\cdot [v]_W$ 
is given by the Weyl group action on~$W/W_{J_\lambda}$. It remains to observe that $\sigma^i|_{V_\lambda(0)}=\id_{V_\lambda(0)}$, $i\in I$.
\end{proof}
This result can be extended to a larger class of modules. We say that $V\in\mathscr O^{int}_q(\lie g)$ is {\em thin} if 
$\dim V(\beta)\le 1$ for all $\beta\in P\setminus\{0\}$. By definition, every quasi-miniscule module is thin. Furthermore, 
all modules $V_{m\omega_1}$, $V_{m\omega_n}$, $m\in\ZZ_{\ge 0}$ are thin for $\lie g=\lie{sl}_{n+1}$.
\begin{theorem}\label{thm:thin modules}
Conjecture~\ref{conj:weyl} holds for thin modules.
\end{theorem}
\begin{proof}
Let $(L,B)$ be an upper crystal basis of~$V\in\mathscr O^{int}_q(\lie g)$ and let $G^{up}(B)\subset V$ be the corresponding global crystal basis.
We say that $b\in G^{up}(B)$ of weight~$\beta\in P$ is thin if either $\beta=0$ or $V(\beta)=\kk b$. Denote 
$G_0^{up}(B)$ the set of thin elements in~$G^{up}(B)$. Clearly, $V$ is thin if and only if $G_0^{up}(B)=G^{up}(B)$.
We need the following
\begin{proposition}\label{prop:thin perms}
For any $b\in G^{up}(B)\cap V(\beta)$, $\beta\in P$ with $\dim V(\beta)=1$
we have $\Phi_V(b)\subset G^{up}(B)$. 
\end{proposition}
\begin{proof}
It suffices to prove that $\sigma^J(b)\in G^{up}(B)$ for all $J\in\mathscr J$. Since $\sigma^J(V(\beta))=V(w_\circ^J\beta)$ and $\dim V(w_\circ^J\beta)=
\dim V(\beta)=1$, it follows that 
$\sigma^J(b)=c b'$ for some~$b'\in G^{up}(B)\cap V(w_\circ^J\beta)$. Let $\underline b$ be the image of~$b$ in $B$ under the quotient map $L\mapsto L/qL$.
By Theorem~\ref{thm:main thm2}, $\tilde\sigma^J(\underline b)=\underline b'$ and so $c\in 1+q \mathbb A$. It follows from Proposition~\ref{prop:sigma bar}
that $\overline{c}=c$ and so~$c=1$.
\end{proof}
Let $g:B\to G^{up}(B)$ be Kashiwara's bijection (cf.~\cite{Kas93}) and let $B_0=g^{-1}(G_0^{up}(B))$.
Then by Theorem~\ref{thm:main thm2} and 
Proposition~\ref{prop:thin perms} we have
$g(\tilde\sigma^i(b))=\sigma^i(g(b))$ for all $b\in B_0$.
Since the action of~$\tilde\sigma^i$ on~$B$ coincides with the action of~$W$ defined in~\cite{Kas94},
it follows from \cite{Kas94}*{Theorem~7.2.2} that $\mathsf W(V)$ is a homomorphic image of~$W$.

We may assume, without loss of generality, that $V=V_\lambda$ and $J(V)=\emptyset$.
In view of Proposition~\ref{prop:ficus action}\ref{prop:ficus action.b}, the action of~$\mathsf W(V)$ on
the set $[v_\lambda]_W$, $v_\lambda\in V_\lambda(\lambda)$ is faithful and coincides with that of~$W$. 
This implies that $\psi_V$ from Theorem~\ref{thm:crystal weyl group} is an isomorphism.
\end{proof}

\subsection{Crystallizing cactus group action for~\texorpdfstring{$\lie g=\lie{sl}_3$}{g=sl\_3}}\label{subs:cryst cact sl3}
We now describe combinatorial consequences of Theorem~\ref{thm:main thm2} for~$\lie g=\lie{sl}_3$. It turns out 
that the corresponding action of~$\Cact_{S_3}$ lifts to the ambient set 
$$
\widehat{\mathbf M}=\{ (m_1,m_2,m_{12},m_{21},m_{01},m_{02})\in\ZZ_{\ge 0}^2\times\ZZ^4\,:\, m_1m_2=0\}
$$
where the crystal basis for~$\mathcal C_q(\lie{sl}_3)$ identifies with 
$\mathbf M=\widehat{\mathbf M}\cap \ZZ_{\ge 0}^6$.

We need some notation. Define $\wt_i:\widehat{\mathbf M}\to \ZZ$ by 
$$\wt_i(\mathbf m)=m_{0i}-m_i+m_j-m_{ij},\qquad \{i,j\}=\{1,2\}.
$$
Furthermore, define $e_i^r:\widehat{\mathbf M}\to \widehat{\mathbf M}$, $i\in\{1,2\}$, $r\in\ZZ$, by 
$$
e_i^r(m_1,m_2,m_{12},m_{21},m_{01},m_{02})=(m'_1,m'_2,m'_{12},m'_{21},m'_{01},m'_{02})
$$
where, for $\mathbf m=(m_1,m_2,m_{12},m_{21},m_{01},m_{02})\in\widehat{\mathbf M}$ we set 
\begin{gather*}
m'_i=[m_i-m_j-r]_+,\quad m'_j=[m_j-m_i+r]_+,\quad m'_{ji}=m_{ji},\quad m'_{0j}=m_{0j},
\\
m'_{ij}=m_{ij}+\min(m_i-r,m_j),
\quad m'_{0i}=m_{0i}+r+\min(m_i-r,m_j),\qquad \{i,j\}=\{1,2\},
\end{gather*}
and $[x]_+:=\max(x,0)$, $x\in\ZZ$.
The following is well-known (cf.~\cite{BK}*{Example~6.26}).
\begin{lemma}\label{lem:props of ops e i r}
The $e_i^r$, $i\in \{1,2\}$, $r\in\ZZ$
satisfy
\begin{equation}\label{eq:props of ops}
e_1^r e_1^s=e_1^{r+s},\quad e_2^r e_2^s=e_2^{r+s},\quad  
e_1^r e_2^{r+s}e_1^s=e_2^s e_1^{r+s}e_2^r,\quad r,s\in\ZZ.
\end{equation}
In particular, $e_i^0=\id$ and $(e_i^r)^{-1}=e_i^{-r}$, $r\in\ZZ$, $i\in\{1,2\}$.
\end{lemma}
\begin{proof}
Define a map $\widehat{\mathbf k}:\widehat{\mathbf M}\to \ZZ^5$ by 
$\widehat{\mathbf k}(m_1,m_2,m_{12},m_{21},m_{01},m_{02})\mapsto (a_1,a_2,a_3,l_1,l_2)$ where 
$$
a_1=m_1+m_{21},\quad a_2=m_2+m_{12}+m_{21},\quad a_3=m_{12},\quad l_i=m_i+m_{3-i,i}+m_{0i}
$$
with $i\in\{1,2\}$. It is easy to see that~$\widehat{\mathbf k}$ is a bijection with its inverse given by 
$$
(a_1,a_2,a_3,l_1,l_2)\mapsto (m_1,m_2,m_{12},m_{21},m_{01},m_{02}),
$$
where $m_1=[a_1+a_3-a_2]_+$, $m_2=[a_2-a_1-a_3]_+$, $m_{12}=a_3$, $m_{21}=\min(a_1,a_2-a_3)$, $m_{01}=l_1-a_1$, $m_{02}=l_2-a_2+\min(a_1,a_2-a_3)$.

The action of operators $e_i^r$, $i\in\{1,2\}$, $r\in\ZZ$ on $\ZZ^5$ induced by this bijection
coincides with the action constructed in~\cite{BK}*{Example~6.26}
$$e_1^{r}(a_1,a_2,a_3,l_1,l_2)=(a_1+[\delta-r]_+-[\delta]_+,a_2,a_3+[\delta]_+-\max(\delta,r),l_1,l_2)
$$
where $\delta=a_1+a_3-a_2$, and 
$$e_2^r(a_1,a_2,a_3,l_1,l_2)=(a_1,a_2-r,a_3,l_1,l_2).
$$
The identities from the Lemma are now easy to obtain by using tropicalized relations for the $e_i$ given after Definition~2.20
in~\cite{BK} in the context of~\cite{BK}*{Example~6.26}.
\end{proof}

Define $\underline\sigma=\underline\sigma^{\{1,2\}}:\widehat{\mathbf M}\to \widehat{\mathbf M}$ by 
$$
(m_1,m_2,m_{12},m_{21},m_{01},m_{02})\mapsto (m_1,m_2,m_{02},m_{01},m_{21},m_{12}).
$$
Clearly, $\underline\sigma$ is an involution and $\underline\sigma(\mathbf M)=\mathbf M$. Furthermore, 
define $\underline\sigma^i:\widehat{\mathbf M}\to \widehat{\mathbf M}$, $i\in\{1,2\}$ by 
$$
\underline\sigma^i(\mathbf m)=e_i^{-\wt_i(\mathbf m)}(\mathbf m),\qquad \mathbf m\in\widehat{\mathbf M}.
$$

\begin{proposition}\label{prop:underl eta}
The following identities hold in~$\operatorname{Bij}(\widehat{\mathbf M})$
\begin{gather*}
\underline\sigma^i\circ\underline\sigma^i=\id,\qquad \underline\sigma^i\circ e_i^r=e_i^{-r}\circ\underline\sigma^i,\qquad
\underline\sigma\circ e_i^r=e_j^{-r}\circ\underline\sigma,\\
\underline\sigma^i\circ \underline\sigma^j\circ\underline\sigma^i=\underline\sigma^j\circ\underline\sigma^i\circ\underline\sigma^j,
\qquad \underline\sigma^i\circ\underline\sigma=\underline\sigma\circ\underline\sigma^j,
\end{gather*}
where $\{i,j\}=\{1,2\}$.
In particular, the assignments $\tau_{i,i+1}\mapsto \underline\sigma^i$, $i\in\{1,2\}$, $\tau_{13}\mapsto \underline\sigma$ 
define an action of~$\Cact_{S_3}$ on~$\widehat{\mathbf M}$.
\end{proposition}
\begin{proof}
Since $\wt_i(e_i^r(\mathbf m))=\wt_i(\mathbf m)+2r$ for any $\mathbf m\in\widehat{\mathbf M}$, we have 
$$
\underline\sigma^i\circ\underline\sigma^i(\mathbf m)=e_i^{-(\wt_i(\mathbf m)-2\wt_i(\mathbf m))-\wt_i(\mathbf m)}(\mathbf m)=\mathbf m,
$$
while
$$
\underline\sigma^i\circ e_i^r(\mathbf m)=e_i^{-r-\wt_i(\mathbf m)}(\mathbf m)=e_i^{-r}\circ\underline \sigma^i(\mathbf m).
$$
To prove the third identity, note that $e_i^r(\underline\sigma(\mathbf m))=(\tilde m_1,\tilde m_2,\tilde m_{12},\tilde m_{21},
\tilde m_{01},\tilde m_{02})$, where 
\begin{gather*}
\tilde m_i=[m_i-m_j-r]_+,\quad \tilde m_j=[m_j-m_i+r]_+,\quad \tilde m_{ji}=m_{0i},\quad \tilde m_{0j}=m_{ij},
\\
\tilde m_{ij}=m_{0j}+\min(m_i-r,m_j),
\quad \tilde m_{0i}=m_{ji}+\min(m_i,m_j+r),\qquad \{i,j\}=\{1,2\},
\end{gather*}
which is easily seen to coincide with $\underline\sigma(e_j^{-r}(\mathbf m))$. The braid identity follows from the last relation in~\eqref{eq:props of ops} (known as Verma relations)
and the identity $\wt_j(e_i^r(\mathbf m))=\wt_j(\mathbf m)-r$, $\{i,j\}=\{1,2\}$. Finally,
$$
\underline\sigma\circ\sigma^j(\mathbf m)=\underline\sigma\circ e_j^{-\wt_j(\mathbf m)}(\mathbf m)=e_i^{\wt_j(\mathbf m)}(\underline\sigma(\mathbf m))
=e_i^{-\wt_i(\underline\sigma(\mathbf m))}(\underline\sigma(\mathbf m))=\underline\sigma^i\circ \underline\sigma(\mathbf m),
$$
where we used the identity~$\wt_i(\underline\sigma(\mathbf m))=m_{ji}-m_i+m_j-m_{0j}=-\wt_j(\mathbf m)$.
\end{proof}

\begin{remark}
It would be interesting to define analogues of~$\widehat{\mathbf M}$ for other~$\lie g$ and study the action 
of the corresponding cactus groups on~$\widehat{\mathbf M}$. We plan to study this in a subsequent publication
via the approach of~\cite{BK}.
\end{remark}

Given $l_1,l_2\in\ZZ$ define 
$$
\widehat{\mathbf M}_{l_1,l_2}=\{ (m_1,m_2,m_{12},m_{21},m_{01},m_{02})\in\widehat{\mathbf M}\,:\,
m_{01}+m_1+m_{21}=l_1,\, m_{02}+m_2+m_{12}=l_2\}.
$$
Clearly, $\underline\sigma$, $\underline\sigma^i$, $e_i^r$, $i\in\{1,2\}$, $r\in\ZZ$ preserve $\widehat{\mathbf M}_{l_1,l_2}$
for any $l_1,l_2\in\ZZ$. Set $\mathbf M_{l_1,l_2}=\widehat{\mathbf M}_{l_1,l_2}\cap\mathbf M$.

In view of~\cite{BK}*{Example~6.26}, $\widehat{\mathbf k}(\mathbf M_{l_1,l_2})$, where~$\widehat{\mathbf k}$ is 
defined in the proof of Lemma~\ref{lem:props of ops e i r}, identifies with the upper crystal basis
$B^{up}(V_{l_1\omega_1+l_2\omega_2})$ of~$V_{l_1\omega_1+l_2\omega_2}$. In particular, $\widehat{\mathbf k}(\mathbf M)$
identifies with the upper crystal basis $B^{up}(\mathcal C_2)=\bigsqcup_{\lambda\in P^+} B^{up}(V_\lambda)$ of~$\mathcal C_2=\mathcal C_q(\lie{sl}_3)$. 
We use this identification throughout the rest of this chapter.
\begin{proposition}
Under the above identification,
the restrictions of $\underline\sigma$, $\underline\sigma^i$, $i\in\{1,2\}$ to~$\mathbf M$
coincide with the action of~$\Cact_{S_3}$ on~$B^{up}(\mathcal C_2)$
provided by 
Theorem~\ref{thm:main thm2} with $\lie g=\lie{sl}_3$ and~$V=\mathcal C_2$.
\end{proposition}
\begin{proof}
It follows from Corollary~\ref{cor:char eta} applied to $f=\underline\sigma$ extended to~$B_\lambda\cup\{0\}$, 
and Proposition~\ref{prop:underl eta} that $\underline\sigma$ coincides with $\tilde\sigma^{\{1,2\}}_{V_\lambda}$ for all~$\lambda\in P^+$.
On the other hand, by Remark~\ref{rem:crystal Weyl group} we have $\underline\sigma^i=\tilde\sigma^{\{i\}}_{V_\lambda}$, $i\in\{1,2\}$,
$\lambda\in P^+$.
\end{proof}

\subsection{Gelfand-Kirillov model for \texorpdfstring{$\lie g=\lie{sl}_3$}{g=sl\_3}}
Our goal here is to illustrate results and constructions from Section~\ref{sect:can bas} for $\lie g=\lie{sl}_3$ and provide
some evidence for Conjecture~\ref{conj:weyl}. We freely use the notation from Section~\ref{sect:can bas} and~\S\ref{subs:cryst cact sl3}. 
In this case the algebra $\mathcal A_2=\mathcal A_q(\lie g)$ is generated by the $x_i$, $i\in\{1,2\}$ subject to the relations 
\begin{equation}\label{eq:Serre sl3}
x_i^2 x_j-(q+q^{-1})x_i x_j x_i+x_j x_i^2=0,\qquad \{i,j\}=\{1,2\}.
\end{equation}
Define 
$$
x_{ij}=\frac{q^{\frac12} x_i x_j-q^{-\frac12} x_j x_i}{q-q^{-1}},\qquad \{i,j\}=\{1,2\}.
$$
Then $x_i x_j=q^{\frac12} x_{ij}+q^{-\frac12}x_{ji}$, $\{i,j\}=\{1,2\}$ and~\eqref{eq:Serre sl3} 
is equivalent to $x_i x_{ij}=q x_{ij}x_i$ or $x_i x_{ji}=q^{-1} x_{ji}x_i$, $\{i,j\}=\{1,2\}$.
The following is well-known${}^{\ref{ftn:1}}$ (see e.g.~\cite{BZ2}).
\begin{lemma}\label{lem:B up}
The dual canonical basis~$\mathbf B^{up}$ in the algebra~$\mathcal A_2$ is
$$
\mathbf B^{up}=\{ q^{\frac12(m_1-m_2)(m_{21}-m_{12})}x_1^{m_1} x_2^{m_2} x_{12}^{m_{12}}x_{21}^{m_{21}}\,:\,
(m_1,m_2,m_{12},m_{21})\in\ZZ_{\ge0}^4,\, m_1m_2=0\}.
$$
\end{lemma}

We have~$\widehat{\mathcal A_2}=\mathcal A_2[\mathcal S_{w_\circ}^{-1}]=\mathcal A_2[x_{12}^{-1},x_{21}^{-1}]$.
It follows from Lemma~\ref{lem:B up} that 
$$
\widehat{\mathbf B^{up}}=\{ q^{\frac12(m_1-m_2)(m_{21}-m_{12})}x_1^{m_1} x_2^{m_2} x_{12}^{m_{12}}x_{21}^{m_{21}}\,:\,
(m_1,m_2,m_{12},m_{21})\in\ZZ_{\ge0}^2\times \ZZ^2,\, m_1m_2=0\}.
$$
The following is immediate
\begin{lemma}
\begin{enumerate}[label={\rm(\alph*)},leftmargin=*]
 \item 
The algebra $\mathcal B_2:=\mathcal B_q(\lie g)$ is generated by~$\mathcal A_2$ and $\kk[v_1,v_2]$, where $v_i=v_{\omega_i}$, as subalgebras
subject to the relations
$$
v_i x_j=q^{-\delta_{i,j}}x_j v_i,\qquad i,j\in\{1,2\}.
$$
\item $\mathcal B_2$ is a $U_q(\lie g)$-module algebra with the $U_q(\lie g)$-action defined in Lemma~\ref{lem:action}.
\end{enumerate}
\label{lem:B 2 present}
\end{lemma}

Abbreviate $z_i=F_i(v_i)=v_{s_i\omega_i}$ and $z_{ij}=F_iF_j(v_j)=v_{s_is_j\omega_j}$. Clearly,
\begin{equation}\label{eq:z via x}
z_i=q^{-\frac12} x_i v_i,\qquad z_{ij}=q^{-\frac12} x_{ij}v_j,\qquad \{i,j\}=\{1,2\}.
\end{equation}
The following Lemma is an immediate consequence of Lemma~\ref{lem:B 2 present}
\begin{lemma}
\begin{enumerate}[label={\rm(\alph*)},leftmargin=*]
 \item \label{lem:C present.a}
The algebra $\mathcal C_2=\mathcal C_q(\lie{sl}_3)$ is generated by $v_1$, $v_2$, $z_1$, $z_2$,
$z_{12}$ and $z_{21}$
subject to the relations 
$$
v_1 v_2=v_2 v_1,\quad 
v_i z_j=q^{-\delta_{i,j}} z_j v_i,\quad v_i z_{12}=q^{-1}z_{12} v_i,\quad v_i z_{21}=q^{-1}z_{21}v_i,\quad i,j\in\{1,2\}.
$$
$$
z_i z_j=q v_i z_{ij}+q^{-1}z_{ji}v_j,\quad z_k z_{ij}=q^{-\delta_{j,k}}z_{ij}z_k,\quad \{i,j\}=\{1,2\},\, k\in \{1,2\}
$$
and $z_{12}z_{21}=z_{21}z_{12}$.
\item \label{lem:C present.b} $\mathcal C_2$ is a $U_q(\lie g)$-module algebra via 
\begin{alignat*}{4}
&E_k(v_i)=0,&\quad&E_k(z_i)=\delta_{k,i}v_i,&\quad& E_k(z_{ji})=\delta_{k,j}z_i,\\
&F_k(v_i)=\delta_{i,k}z_i,&& F_k(z_i)=\delta_{k,j}z_{ji},&&  F_k(z_{ij})=0,&\quad& \{i,j\}=\{1,2\},\, k\in\{1,2\}.
\end{alignat*}
\item \label{lem:C present.c} The $P$-grading on~$\mathcal C_q(\lie g)$ is given by $|v_i|=\omega_i$, $|z_i|=s_i\omega_i=\omega_i-\alpha_i$, $|z_{ij}|=
s_js_i\omega_j=\omega_j-\alpha_i-\alpha_j$,
$\{i,j\}=\{1,2\}$.
\end{enumerate}
\label{lem:C present}
\end{lemma}

The following is an immediate corollary of Theorem~\ref{thm:twist}.
\begin{corollary}\label{cor:eta explicit}
The assignments
$$
v_i\mapsto z_{ji},\quad z_i\mapsto z_i,\quad z_{ij}\mapsto v_j,\qquad \{i,j\}=\{1,2\}.
$$
define an anti-involution of~$\mathcal C_2$ which coincides with~$\sigma=\sigma^{\{1,2\}}_{\mathcal C_2}$.
\end{corollary}
Given $\mathbf m\in\widehat{\mathbf M}$, define $\mathbf b_{\mathbf m}\in\mathcal C_2$ as 
\begin{equation}\label{eq: elt of bas}
\mathbf b_{\mathbf m}=
q^{\frac12(m_1(m_{21}-m_{01})+m_2(m_{12}-m_{02})-(m_{12}+m_{21})(m_{01}+m_{02}))} z_1^{m_1} z_2^{m_2} z_{12}^{m_{12}}z_{21}^{m_{21}} v_1^{m_{01}}v_2^{m_{02}}
\end{equation}
if $\mathbf m\in\mathbf M$ and $\mathbf b_{\mathbf m}=0$ if $\mathbf m\in\widehat{\mathbf M}\setminus\mathbf M$.
Thus, $|\mathbf b_{\mathbf m}|=(m_{01}+m_1+m_{21})\omega_1+(m_{02}+m_2+m_{12})\omega_2-(m_1+m_{12}+m_{21})\alpha_1-(m_2+m_{12}+m_{21})\alpha_2
=\wt_1(\mathbf m)\omega_1+\wt_2(\mathbf m)\omega_2$,
$\mathbf m\in \mathbf M$.

The following is a consequence of Lemma~\ref{lem:C present}, \eqref{eq:intersect B A} and Proposition~\ref{prop:other defn of hat B}
\begin{lemma}
The upper global crystal basis~$\mathbf B$ of~$\mathcal C_2$ is $\mathbf B=\{ \mathbf b_{\mathbf m}\,:\, \mathbf m\in\mathbf M\}$.
Furthemore, for each $\lambda=l_1\omega_1+l_2\omega_2\in P^+$ we have 
$\mathbf B_\lambda=\{ \mathbf b_{\mathbf m}\,:\, \mathbf m\in \mathbf M_{l_1,l_2}\}$.
Moreover, under the identification of~$\mathbf M$ with the upper crystal basis of~$\mathcal C_q(\lie g)$ (cf.~\S\ref{subs:cryst cact sl3}),
the map $\mathbf M\to\mathbf B$ defined by $\mathbf m\mapsto \mathbf b_\mathbf m$, $\mathbf m\in\mathbf M$
is Kashiwara's bijection~$G$ (\cite{Kas93}) 
between an upper crystal basis of~$\mathcal C_2$ and its upper global crystal basis. 
\end{lemma}
The following is an explicit form of Theorem~\ref{thm:eta-bas} for~$\lie g=\lie{sl}_3$ and follows from~\eqref{eq: elt of bas} and Corollary~\ref{cor:eta explicit}. 
\begin{lemma}
We have $
\sigma(\mathbf b_{\mathbf m})=\mathbf b_{\underline\sigma(\mathbf m)}$ for all $\mathbf m\in\mathbf M$.
\end{lemma}
\begin{remark}\label{rem:tau 13 not in subgroup}
It is easy to check that 
$$
|\mathbf B_\lambda(0)|=|\{\mathbf m\in \mathbf M_{l_1,l_2}\,:\, \wt_1(\mathbf m)=\wt_2(\mathbf m)=0\}|=
\begin{cases}
\min(l_1,l_2)+1,& l_1\equiv l_2\pmod 3\\
0,&\text{otherwise}.
\end{cases}
$$
It follows from the definition of~$\underline\sigma$ that $\sigma$ is trivial on~$\mathbf B_\lambda(0)$ if and only if~$\dim V_\lambda(0)=1$
(that is, if and only if $\min(l_1,l_2)=0$ and $\max(l_1,l_2)\in 3\ZZ_{>0}$). Thus, $\tau_{1,3}\notin\mathsf K_{\lie{sl}_3}$ in the notation 
introduced after Problem~\ref{prob:kernel}.
On the other hand, it is immediate from the definitions
that the $\sigma^i$, $i\in\{1,2\}$ act trivially on~$V_\lambda(0)$ for any $V\in\mathscr O^{int}_q(\lie g)$. In particular, $\sigma$ is not contained
in~$\mathsf W(V_\lambda)$ if $\dim V_\lambda(0)>1$.
\end{remark}

In order to calculate~$\sigma^i$, $i\in\{1,2\}$ be need the following result.
\begin{lemma}\label{lem:E act basis}
For any $\mathbf m\in\mathbf M$, $r\ge 0$ and $i\in\{1,2\}$ we have 
\begin{align*}
&E_i^{(r)}(\mathbf b_{\mathbf m})=\binom{m_i+m_{ij}}{r}_q \mathbf b_{e_i^r(\mathbf m)}+
\sum_{1\le t\le r} C^{(r)}_t(m_j+m_{ij},m_i+m_{ij}) \mathbf b_{e_i^{r}(\mathbf m)+t\mathbf a^+_i}
\\
&F_i^{(r)}(\mathbf b_{\mathbf m})=\binom{m_j+m_{0i}}{r}_q \mathbf b_{e_i^{-r}(\mathbf m)}+
\sum_{1\le t\le r} C^{(r)}_t(m_i+m_{0i},m_j+m_{0i}) \mathbf b_{e_i^{-r}(\mathbf m)-t\mathbf a^+_i}
\end{align*}
where $\mathbf a^+_1=(0,0,-1,1,-1,1)$, $\mathbf a^+_2=-\mathbf a^+_1$ and 
\begin{equation}\label{eq:coefs C^r_t}
C^{(r)}_t(c,d)=\begin{cases}
\displaystyle
           \binom{c}{t}_q\binom{d-t}{r-t}_q,& d-c\ge r,\\
\displaystyle           \binom{d-c}{t}_q \binom{d-t}{r}_q,& d-c<r,
          \end{cases}
\end{equation}
with the convention that $\binom{k}{l}_q=0$ if $k<l$.
\end{lemma}
\begin{proof}
Both identities can be proven by induction on~$r$ using Lemma~\ref{lem:C present} and the fact that the $E_i$, $F_i$
act on $\mathcal C_2$ by $K_{\frac12\alpha_i}$-derivations.
\end{proof}

Denote
$$
\mathbf b^{(i)}_{\mathbf m}=E_i^{(r_i)}(\mathbf b_{e_i^{-r_i}(\mathbf m)}),\qquad r_i=m_j+m_{0i},\, i\in \{1,2\},\, \mathbf m\in\mathbf M.
$$
The following is immediate from Lemma~\ref{lem:E act basis}.
\begin{lemma}\label{lem:sigma permutes GT}
For each $i\in\{1,2\}$, $\mathbf B^{(i)}=\{ \mathbf b^{(i)}_{\mathbf m}\,:\, \mathbf m\in\mathbf M\}$ 
is a basis of~$\mathcal C_2$. Moreover, for each $i\in \{1,2\}$, $\lambda=l_1\omega_1+l_2\omega_2\in P^+$,
$
\mathbf B^{(i)}_\lambda:=\{ \mathbf b^{(i)}_{\mathbf m},\,:\, \mathbf m\in\mathbf M_{l_1,l_2}\}
$
is a basis of~$V_\lambda$.
Finally,
$$
\sigma^i(\mathbf b^{(i)}_{\mathbf m})=\mathbf b^{(i)}_{\underline\sigma^i(\mathbf m)},\qquad i\in \{1,2\},\,\mathbf m\in \mathbf M.
$$
In particular, $\sigma^i(\mathbf B^{(i)}_\lambda)=\mathbf B^{(i)}_\lambda$.
\end{lemma}

Thus, $\sigma^i$, $i\in\{1,2\}$ are easy to calculate in respective bases~$\mathbf B^{(i)}$. To attack Conjecture~\ref{conj:weyl} we 
need to find the matrix of both of them in a same basis. Note the following consequence of Lemma~\ref{lem:E act basis}.

\begin{corollary}
For each $\lambda=l_1\omega_1+l_2\omega_2\in P^+$ we have 
$$
\mathbf b^{(i)}_{\mathbf m}=\sum_{ \mathbf m'\in \mathbf M_{l_1,l_2}}
C^{i;\lambda}_{\mathbf m',\mathbf m}\mathbf b_{\mathbf m'}
$$
where $C^{i;\lambda}$ is an $\mathbf M_{l_1,l_2}\times \mathbf M_{l_1,l_2}$-matrix given by
$$
C^{i;\lambda}_{\mathbf m',\mathbf m}=\begin{cases}
                                      \displaystyle\binom{m_i+m_{0i}+m_j+m_{ij}}{m_i+m_{ij}}_q,& \mathbf m'=\mathbf m,\\
                                      C^{(m_j+m_{0i})}_{t}(m_j+m_{ij},m_i+m_{0i}+m_j+m_{ij}),& \mathbf m'-\mathbf m=t\mathbf a^+_i,\,t\in\ZZ_{>0},\\
                                      0,& \text{otherwise}.
                                     \end{cases}
$$
\end{corollary}
\begin{remark}\label{rem:GT bas and sigma}
The bases~$\mathbf B^{(i)}_\lambda$, $i\in\{1,2\}$, $\lambda\in P^+$ are in fact Gelfand-Tsetlin bases. The matrices~$C^{i;\lambda}$
appeared first in the classical limit ($q=1$) in~\cite{GelZel}. According to \cite{GelZel}*{Theorem~10} their 
entries 
are closely related to Clebsch-Gordan coefficients, and so one should expect that our matrices 
are related to quantum Clebsch-Gordan coefficients.
It is easy to see that $\sigma^i(\mathbf B_\lambda)=\mathbf B_\lambda$ if and only if~$V_\lambda$ is thin. 
\end{remark}

\begin{theorem}
For each $\lambda=l_1\omega_1+l_2\omega_2\in P^+$ and $i\in\{1,2\}$ the matrix~$N^{i;\lambda}$ of $\sigma^i$ with respect to the basis $\mathbf B_\lambda$ of~$V_\lambda$
is given by 
$$
N^{i;\lambda}=C^{i;\lambda}P^{i;\lambda} (C^{i;\lambda}){}^{-1}
$$
where $P^{i;\lambda}=(P^{i;\lambda}_{\mathbf m',\mathbf m})_{\mathbf m,\mathbf m'\in \mathbf M_{l_1,l_2}}$ with 
$$
P^{i;\lambda}_{\mathbf m',\mathbf m}=\delta_{\mathbf m',\underline\sigma^i(\mathbf m)},\qquad \mathbf m,\mathbf m'\in \mathbf M_{l_1,l_2}.
$$
\end{theorem}
\begin{conjecture}[Conjecture~\ref{conj:weyl} for $\lie g=\lie{sl}_3$]
For each $\lambda=l_1\omega_1+l_2\omega_2\in P^+$ we have 
$$
(N^{1;\lambda}N^{2;\lambda})^3=1.
$$
\end{conjecture}
\noindent
This was verified using Mathematica\textsuperscript{\textregistered} for all $l_1,l_2\in\ZZ_{\ge 0}$ such that $l_1+l_2\le 14$.

\renewcommand{\PrintDOI}[1]{%
DOI \href{http://dx.doi.org/#1}{#1}}
\renewcommand{\eprint}[1]{\href{http://arxiv.org/abs/#1}{arXiv:#1}}

\end{document}